\documentclass[12pt, reqno]{amsart}
\usepackage{amsmath,amssymb,amsthm,amsxtra, setspace}
\usepackage{amssymb}
\usepackage{mathrsfs}
\usepackage{alltt}
\usepackage{graphicx,type1cm,xcolor}
\usepackage{mathtools, accents}
\usepackage{mathrsfs}
\usepackage{hyperref}
\usepackage{alltt}
\usepackage{fouridx}
\usepackage{dsfont}
\usepackage{cancel}
\usepackage[mathcal]{euscript}
\usepackage[margin=1in]{geometry}
\usepackage{shadethm}
\usepackage{float}

\usepackage{todonotes}

\usepackage{ragged2e}
\justifying

\allowdisplaybreaks

\newcommand{\eps}{\varepsilon}

\newcommand{\supp}{\mathrm{supp}}
\numberwithin{equation}{section}

\newcommand{\wtilde}{\widetilde}
\newshadetheorem{Theorem}{Theorem}
\newshadetheorem{Proposition}{Proposition}
\newshadetheorem{Lemma}{Lemma}
\newshadetheorem{Assumption}{Assumption}
\newshadetheorem{Corollary}{Corollary}
\newshadetheorem{Definition}{Definition}
\newshadetheorem{Example}{Example}
\newshadetheorem{Remark}{Remark}
\newshadetheorem{Note}{Note}
\newshadetheorem{Hypothesis}{Hypothesis}
\newshadetheorem{Explanation}{Explanation}
\newtheorem{theorem}{Theorem}[section]
\newtheorem{lemma}[theorem]{Lemma}
\newtheorem{proposition}[theorem]{Proposition}

\newtheorem{definition}[theorem]{Definition}

\newtheorem{remark}[theorem]{Remark}

\newcommand{\embed}{\hookrightarrow}

\def\Ls{\mathcal{L}}
\def\X{\mathbb{X}}

\def\A{\mathcal{A}}

\def\Fn{\mathcal{F}}
\def\Gn{\mathscr{G}}

\def\Ln{\mathscr{L}}
\def\bA{\mathscr{A}}
\def\bL{\mathcal{L}}

\def\E{\mathbb{E}}

\def\bU{\mathcal{U}}

\def\Bb{\mathbb{B}}
\def\0{\boldsymbol{0}}

\def\Y{\mathbb{Y}}

\def\N{\mathbb{N}}

\def\bM{\mathcal{M}}
\def\H{\mathrm{\mathbb{H}}}
\def\Hb{\mathcal{H}}

\def\V{\mathbb{V}}

\def\bS{\mathcal{S}}

\def\bO{\mathcal{O}}
\def\bE{\mathcal{E}}

\def\bh{\mathfrak{h}}

\newcommand{\norm}[1]{\left\Vert#1\right\Vert}
\newcommand{\abs}[1]{\left\vert#1\right\vert}

\newcommand{\R}{\mathbb{R}}

\newcommand{\loc}{\mathop{\mathrm{loc}}}

\renewcommand{\d}{\/\mathrm{d}\/}

\let\originalleft\left
\let\originalright\right
\renewcommand{\left}{\mathopen{}\mathclose\bgroup\originalleft}
\renewcommand{\right}{\aftergroup\egroup\originalright}
\newcommand{\Nn}{\mathcal{N}}

\newcommand{\ip}[2]{\fourIdx{}{0}{}{\!x}{\mathcal{ X}}}

\newcommand\dela[1]{}

\hypersetup{colorlinks=true,%
	citecolor=red,%
	filecolor=blue,%
	linkcolor=blue,%
}
\usepackage{graphicx}

\usepackage{orcidlink}

\usepackage{xpatch}
\makeatletter   
\xpatchcmd{\@tocline}
{\hfil\hbox to\@pnumwidth{\@tocpagenum{#7}}\par}
{\ifnum#1<0\hfill\else\dotfill\fi\hbox to\@pnumwidth{\@tocpagenum{#7}}\par}
{}{}
\makeatother    

\makeatletter
\def\l@subsection{\@tocline{2}{0pt}{4pc}{6pc}{}}
\def\l@subsubsection{\@tocline{3}{0pt}{8pc}{8pc}{}}
\makeatother

\makeatletter
\def\l@section{\@tocline{1}{12pt}{0pt}{}{\bfseries}}
\makeatother
\usepackage[utf8]{inputenc}
\usepackage[T1]{fontenc}
\mathtoolsset{showonlyrefs}
\usepackage{nomencl}
\makenomenclature
\usepackage{amssymb}

\usepackage[symbol]{footmisc}
\usepackage{centernot}
\title[Global Well-posedness and Asymptotic analysis of a heat equation]{Global well-posedness and Asymptotic analysis of a nonlinear heat equation with constraints of finite codimension}

\begin{document}
\maketitle
\begin{center}
	\author{Ashish Bawalia\footnote[4]{Department of Mathematics, Indian Institute of Technology Roorkee-IIT Roorkee, Haridwar Highway, Roorkee, Uttarakhand 247667, India.\\
			\textit{e-mail:} Ashish Bawalia: \email{ashish1@ma.iitr.ac.in; ashish1441chd@gmail.com}.\\
			\textit{e-mail:} Manil T. Mohan: \email{maniltmohan@ma.iitr.ac.in; maniltmohan@gmail.com}.}\orcidlink{0009-0002-9141-2766}, Zdzis\l{}aw Brze\'zniak\footnotemark[2]$^\ast$\orcidlink{0000-0001-8731-6523} and Manil T. Mohan\footnotemark[4]\orcidlink{0000-0003-3197-1136}}
	\footnotetext[2]{Department of Mathematics, The University of York, Heslington, York YO10 5DD, UK.\\
		\textit{e-mail:} Zdzis\l{}aw Brze\'zniak: \email{zdzislaw.brzezniak@york.ac.uk}.\\
		$\hspace{2mm} ^\ast$Corresponding author.
			}
\end{center}

\begin{abstract}
	We prove the global existence and the uniqueness of the $L^p\cap H_0^1-$valued ($2\leq p < \infty$) strong solutions of a nonlinear heat equation with constraints over bounded domains in any dimension $d\geq 1$. Along with the \textit{Faedo-Galerkin} approximation method and the compactness arguments, we utilize the monotonicity and the hemicontinuity properties of the nonlinear operators to establish the well-posedness results. In particular, we show that a Hilbertian manifold $\bM$, which is the unit sphere in $L^2$ space, describing the constraint is invariant. 
	Finally, in the asymptotic analysis, we generalize the recent work of [P. Antonelli, et. al. \emph{Calc. Var. Partial Differential Equations}, 63(4), 2024] to any bounded smooth domain in $\mathbb{R}^d$, $d\geq1$, when the corresponding nonlinearity is a damping. In particular, we show that, for positive initial datum and any $2\le p < \infty$, the unique positive strong solution of the above mentioned nonlinear heat equation with constraints converges in $L^p\cap H_0^1$ to the unique positive ground state.
\end{abstract}


\tableofcontents

\section{Introduction}\label{Sec-Intro}
Let $(\Hb,\|\cdot\|_{\Hb})$ be a Hilbert space with the inner product $(\cdot,\cdot)_{\Hb}$ and $\bM$ denotes unit sphere in $\Hb$, i.e., 
$$\bM:=\left\{\bh\in\Hb:\|\bh\|_{\Hb} =1 \right\}.$$
The tangent space to $\bM$ at a point $\bh \in \Hb$ is characterized by $T_\bh\bM = \{ z \in \Hb : (\bh, z)_{\mathcal{H}} = 0 \},$ that is, it consists of all elements in $\Hb$ orthogonal to $\bh$. Let $F : \Hb \supset D(F) \to \R$ be a scalar field defined on $\Hb$, possibly only densely defined. Then, for every $u_0\in \Hb$, there exists a unique global solution to the initial value problem
\begin{align*}
	\left\{\begin{aligned}
		\frac{d u(t)}{d t} & = F(u(t)), \ \ t> 0,\\
		u(0) & = u_0.
	\end{aligned}
	\right.
\end{align*}
In general, the semi-flow associated with the initial value problem described above, say $\{ \varphi(t, u_0) \}_{t \geq 0}$, does not remain confined to the manifold $\bM$, even when the initial condition satisfies $u_0 \in \bM$. This lack of invariance arises from the fact that the scalar field $F$ is not, in general, tangent to $\bM$; that is, the condition
\begin{align}\label{eqn-f-non-tangent}
	F(\bh) \in T_\bh \bM, \ \text{ for all }\ \bh \in D(F) \cap \bM,
\end{align}
is not necessarily satisfied. Consequently, the evolution dictated by the flow may immediately exit $\bM$, thereby preventing the manifold from being invariant under the dynamics unless additional compatibility conditions are imposed on the vector field $F$.
However, it is straightforward to construct a modified scalar field $\wtilde{F}$ from the original function $F$ such that the tangency condition \eqref{eqn-f-non-tangent} is satisfied. This modification can be accomplished by employing a mapping that projects $F(\bh)$ onto the tangent space $T_\bh \bM,$ for each $\bh \in D(F) \cap \bM$, i.e.,
$\pi_\bh : D(F) \to \Ls(\Hb, \Hb)$, defined by
\begin{align*}
	\pi_\bh := \{\Hb \ni z \mapsto z - (\bh, z)_{\Hb} \bh \in \Hb\} \in \Ls(\Hb,\Hb),\   \mbox{ for all } \ \bh\in D(F).
\end{align*}
A key property of the mapping $\pi_\bh(\cdot)$ is that, for each $\bh \in \bM$, the associated linear map $\pi_\bh : D(F) \to T_\bh \bM$ coincides with the orthogonal projection onto the tangent space $T_\bh \bM$. As a consequence, the modified scalar field $\wtilde{F}$, defined by
\begin{align*}
	\wtilde{F} : D(F) \ni \bh \mapsto \pi_\bh [F(\bh)] \in \Hb,
\end{align*}
for $\bh\in D(F)\cap \bM$, indeed satisfies
\begin{align*}
	( \wtilde{F}(\bh), \bh ) _{\Hb}& = ( \pi_\bh[F(\bh)], \bh )_{\Hb} = ( F(\bh) - ( \bh, F(\bh) )_{\Hb}\, \bh, \bh )_{\Hb}\\
	& = ( F(\bh),\bh )_{\Hb} - ( \bh, F(\bh))_{\Hb} \|\bh\|_{\Hb}^2 =0.
\end{align*}
Consequently, the modified scalar field $\wtilde{F}$ satisfies the tangency condition \eqref{eqn-f-non-tangent}. Moreover, if the original scalar field $F$ is globally defined, i.e., $D(F) = \mathcal{H}$, and locally Lipschitz continuous, then the modified field $\wtilde{F}$ inherits these properties as well. Under these assumptions, the initial value problem
\begin{align}\label{eqn-modified-IVP}
	\begin{cases}
		\displaystyle \frac{d u(t)}{dt} = \wtilde{F}(u(t)), & t > 0, \\
		u(0) = u_0,
	\end{cases}
\end{align}
possesses a local-in-time solution, for every initial datum $u_0 \in D(F)$. Furthermore, if the initial condition $u_0\in\bM$, then the corresponding solution $u(t)\in\bM$, for all $t$, in its interval of existence.
On the other hand, the analysis becomes more subtle when the scalar field $F$ is only densely defined. In this setting, we shall examine the following two distinguished cases:

Let $\bO \subset \R^d$, for $d \in \N$, be a bounded domain with boundary of class $C^2$. Let $A$ denote the negative Laplacian operator $-\Delta$ equipped with homogeneous Dirichlet boundary conditions. Define the scalar field $F : D(A) \subset \Hb \to \Hb$ by $F(u) = -A u$, where $(\Hb, (\cdot, \cdot)_\Hb) = (L^2(\bO), (\cdot,\cdot))$.
In the this case, we see that
\begin{align*}
	\wtilde{F}(u) = \pi_u [F(u)] = -Au - ( u, -Au ) u = -Au - ( u, \Delta u ) u = -Au + \norm{\nabla u}_{L^2(\bO)}^2 u.
\end{align*}
The second case corresponds to the scalar field defined by $F(u) = - \abs{u}^{p-2}u $, where $p\in[2,\infty)$. In this scenario, we assert
\begin{align*}
	\wtilde{F}(u) = \pi_u [F(u)] = -|u|^{p-2}u  + ( u, |u|^{p-2}u)u = -|u|^{p-2}u + \norm{ u}_{L^p}^p u.
\end{align*}
Broadly speaking, the principal aim of this manuscript is to deliver a unified and rigorous treatment of both examples simultaneously. More specifically, we establish the existence and uniqueness result to the initial value problem \eqref{eqn-modified-IVP} with 
$$\wtilde{F}(u)=\Delta u - \abs{u}^{p-2}u +\big( \norm{\nabla u}_{L^2(\bO)}^2+\norm{ u}_{L^p}^p\big) u.$$


\begin{remark}
	Observe that one can also consider $\Hb-$constraint with different values; in that case, the above projection will be
	\begin{align}\label{Def-proj-modified}
		\pi_\bh := \bigg\{\Hb \ni z \mapsto z - \frac{(\bh, z)_{\Hb}}{\norm{\bh}_{\Hb}^2} \bh \in \Hb\bigg\} \in \Ls(\Hb,\Hb),\   \mbox{ for } \ 0\ne \bh\in D(F).
	\end{align}
	For example, in \cite{PA+PC+BS-24}, for $\Hb=L^2(\bO)$, the constrained is considered as $\norm{u(t)}_{L^2(\bO)} = \norm{u_0}_{L^2(\bO)}$, for all $t>0$.
\end{remark}

\subsection{Informal description of the results}
This article makes two key contributions. First, it investigates the well-posedness of a generalized constrained system, specifically a nonlinear heat equation with a polynomial-type damping term under Dirichlet boundary conditions. In the second part, the study focuses on the asymptotic behavior of strong solutions to the same problem, providing a detailed analysis of their long-term dynamics.   
Given $2\leq p<\infty$ and $u_0\in L^p(\bO) \cap H_0^1(\bO) \cap \bM$, 
where
\begin{align}\label{eqn-Sphere-bM}\bM:=\left\{u\in L^2(\bO):\|u\|_{L^2(\bO)}=1 \right\},\end{align}
we consider the following initial value problem:
\begin{align}\label{eqn-main-problem}
	\left\{\begin{aligned}
		\frac{\partial u(t)}{\partial t} & = \Delta u(t) -|u(t)|^{p-2}u(t) + \left( \norm{\nabla u(t)}_{L^2(\bO)}^2 + \norm{u(t)}_{L^p(\bO)}^p\right) u(t),\ \ t>0,\\
		u(0) & = u_0, \\
		u(t)|_{\partial\bO} & = 0,
	\end{aligned}\right.
\end{align}
where $u:[0,\infty) \times\bO \to \R$. Moreover, $$\|u(t)\|_{L^2(\bO)}  = 1,\ \text{ for every }\ t\geq 0.$$  

By the method of Faedo-Galerkin approximations, the compactness arguments and  Minty-Browder type techniques, 
we prove the following:
\begin{itemize}
	\item[(i)] there exists a unique $L^p(\bO)\cap H_0^1(\bO)-$valued ($2\leq p < \infty$) strong solution to the system \eqref{eqn-main-problem}, by using a sequence of projections $\{P_m\}_{m\in\N}$  and  self-adjoint operators $\{S_m\}_{m\in\N}$ (for e.g., \cite{ZB+BF+MZ-24, ZB+FH+LW-19, ZB+FH+UM-20, FH-18}) which is bounded in $\bL(L^p(\bO)),$ and the fact that Hilbert spaces and $L^p(\bO)$ spaces, for $p\in(1,\infty)$, are uniformly convex and they satisfy the \textit{Radon-Riesz property};
\item[(ii)] the strong solution is \textit{invariant} in the manifold $\bM$, i.e., when $u_0$ is in $\bM,$ then all its corresponding trajectories $u(t)$, for $t\in[0,T)$, stay in $\bM$;
\item[(iii)] the energy $\bE$ corresponding to the above problem \eqref{eqn-main-problem} is dissipative, i.e.,
\begin{align*}
	\bE(u(t)) \leq \bE(u_0),\ \text{ for all }\ t>0,
\end{align*}
\end{itemize}
where
\begin{equation}\label{eqn-Energy}
\bE(u(t)) := \frac{1}{2} \int_\bO \abs{\nabla u(t,x)}^2 dx + \frac{1}{p} \int_\bO \abs{ u(t,x)}^p dx.
\end{equation}
\begin{remark}
	Note that in equation \eqref{eqn-main-problem}, the projected component already lies within the corresponding tangent space. Hence, the invariance of the manifold is inherently linked to ensuring the well-posedness of the unique strong solution.
\end{remark}
The present study provides a novel and comprehensive framework for establishing the well-posedness of global strong solutions in $L^p(\bO) \cap H_0^1(\bO)$ for constrained nonlinear heat equations, where the nonlinear damping term is of polynomial type, a setting that, to our knowledge, has not been addressed in prior literature.
For $L^p(\bO)\cap H_0^1(\bO)$ initial datum, our work breaks the restriction on the power of the nonlinearity to $2\leq p < \infty$, which was $2 \leq p \leq \frac{2d}{d-2}$ in \cite[Theorem 1.1]{ZB+JH-24}, \cite[Theorem 1.2]{PA+PC+BS-24} and \cite[Theorem 2.2]{JH-23} due to Sobolev's embedding. 
One of the  novelties of this work is that we use the \textit{Minty-Browder} type techniques (with fully local monotone coefficients) to prove the existence and uniqueness of the strong solution to the problem \eqref{eqn-main-problem}.

In the second part of this work,  we
study the asymptotic analysis of the strong solution to the problem \eqref{eqn-main-problem}
such that $u(t) \in\bM$, for all $t\geq 0$. For the asymptotic behaviour of positive unique strong solutions of \eqref{eqn-main-problem}, we adapt the \textit{Antonelli's et al.} \cite{PA+PC+BS-24} approach. Building on their strategy, for any bounded smooth domain $\bO$, we show that:
\begin{itemize}
	\item[(i)]  there exists a positive ground state solution to the problem
	\begin{align}\label{eqn-stationary-copy}
		\Delta u -|u|^{p-2}u+ \big( \norm{\nabla u}_{L^2(\bO)}^2 + \norm{u}_{L^p(\bO)}^p \big) u = 0,
	\end{align}
	which also solves the minimization problem 
	\begin{align}\label{min-problem}
		\min_{u\in L^p(\bO) \cap H_0^1(\bO)} \big\{\bE(u) : \norm{u}_{L^2(\bO)} = 1\big\},
	\end{align}
	where  the map $\bE$, defined in \eqref{eqn-Energy},  is weakly lower semicontinuous and coercive on the space $L^p(\bO) \cap H_0^1(\bO)$;
	
	\item[(ii)] by following the work of \textit{Ouyang} \cite{TO-92}, the solution obtained in (i) is unique;
	
	\item[(iii)] for positive $u_0\in L^p(\bO) \cap H_0^1(\bO)\cap\bM$, the unique positive strong solution to the problem \eqref{eqn-main-problem}
	converges strongly to the unique positive ground state solution of the problem \eqref{eqn-stationary-copy}, as time goes to infinity. 
\end{itemize}
For positive initial datum in $L^p(\bO)\cap H_0^1(\bO)\cap \bM$, our work also generalize the existence and uniqueness result on ground states (i.e., \cite[Theorem 1.7]{PA+PC+BS-24}), to any dimension $d\geq 1$ and for any $2 \leq p < \infty$, on any smooth bounded domain. In \cite{PA+PC+BS-24}, the aforementioned result is established on a ball under the assumption $2 \leq p < \frac{2d}{d-2}$.

\begin{remark}
	Note that our results also work with the following general constraint, i.e., if we consider the constrained set as
	\begin{align*}
		\wtilde{\bM} := \big\{u\in L^2(\bO): \|u\|_{L^2(\bO)} = \|u_0\|_{L^2(\bO)} \big\},
	\end{align*}
	and projection $\pi_u(\cdot)$ as defined in Remark \ref{Def-proj-modified}.
\end{remark}

\subsection{Main results}
Let us choose and fix $d\geq 1$ and $2\leq p<\infty$. Let $\bO \subset \R^d$ be a bounded domain with boundary of $C^2-$class. For a fixed $T>0$, we consider the modified heat equation with damping on the time interval $(0,T)$ as described in equation \eqref{eqn-main-problem}, along with the constraint set defined in \eqref{eqn-Sphere-bM}.
We first introduce the notion of a strong solution corresponding to the problem \eqref{eqn-main-problem}.
\begin{definition}\label{def-strong-soln}
	Let $T>0$ and $u_0 \in L^p(\bO) \cap H_0^1(\bO)\cap\bM$ be fixed.
	A function
	\begin{align*}
		u \in C ([0,T]; L^p(\bO) \cap H_0^1(\bO)\cap\bM)\cap L^2(0,T; D(A))\cap L^{2p-2}(0,T; L^{2p-2}(\bO)),
	\end{align*}
	is called a \emph{strong solution} of the system \eqref{eqn-main-problem} on time interval $[0,T]$, if and only if
	the following three conditions are satisfied:
	\begin{itemize}
		\item[(i)] The time derivative of $u$, in the weak sense, $\frac{\partial u}{\partial t}$ is in $L^2(0,T; L^2(\bO))$ and $ \Gn(u) \in L^2(0,T; L^2(\bO))$, where
		\begin{align}\label{eqn-nonlinear}\Gn(u) := \Delta u -|u|^{p-2}u + \big(\norm{\nabla u}_{L^2(\bO)}^2  + \norm{u}_{L^p(\bO)}^p \big)u.
		\end{align}
		\item[(ii)] The following equation is satisfied:
		\begin{align}\label{eqn-test}
			\int_0^T\bigg(\frac{\partial u(t)}{\partial t} - \Gn(u(t)),\varphi(t)\bigg)dt=0,\ \text{ for all }\ \varphi\in L^2(0,T; L^2(\bO)).
		\end{align}
		\item[(iii)] The initial condition is satisfied
		\begin{align}
			u(0) = u_0 \ \text{ in }\  L^2(\bO).
		\end{align}
	\end{itemize}
\end{definition}
\begin{remark}
	In other words, the equation \eqref{eqn-test} can  be understood as
	\begin{align}\label{eqn-102}
		\displaystyle \frac{\partial u(t)}{\partial t}  = \Gn(u(t)),
	\end{align}
	is satisfied in $L^2(0,T;L^2(\bO))$, i.e., $\frac{\partial u}{\partial t}, \Delta u, \abs{u}^{p-2}u \in L^2(0,T; L^2(\bO))$ and equation \eqref{eqn-102} holds true in $L^2(0,T;L^2(\bO))$.
\end{remark}
The following is one of the main results of this work:
\begin{theorem}\label{Thm-main-wellposed}
		Let $T>0$ be fixed.
		For every fixed $u_0 \in L^p(\bO) \cap H_0^1(\bO) \cap \bM$, there exists a unique function $u : [0, \infty) \to L^p(\bO) \cap H_0^1(\bO)\cap \bM$ such that
	$$u\in C([0,T]; L^p(\bO) \cap H_0^1(\bO) \cap \bM) \cap L^2(0,T; D(A))\cap L^{2p-2}(0,T;L^{2p-2}(\bO)),$$
	which solves the problem \eqref{eqn-main-problem} in the sense of Definition \ref{def-strong-soln}. In particular, the function $u$ stays on $\bM$, i.e., $$u(t) \in \bM,\ \text{ for all }\ t\geq 0.$$  
	Moreover, it satisfies
	\begin{align}\label{eqn-Eu}
		\bE(u(t))+\int_0^t\|\nabla_{\bM}\bE(u(t))\|_{L^2(\bO)}^2ds = \bE(u_0), \ \text{ for all }\ t\geq 0,
	\end{align}
	where \begin{align}\label{Eu}
		\bE(u) := \frac{1}{2} \int_\bO \abs{\nabla u(x)}^2 dx + \frac{1}{p} \int_\bO \abs{ u(x)}^p dx,
	\end{align}
	and $\nabla_{\bM}$ is gradient of $\bE$ tangent to the $\bM$, i.e., $$\nabla_{\bM}\bE(u)=\pi_u(\nabla\bE(u)).$$
\end{theorem}
\begin{remark}\label{Rmk-En-dissipative}
	Note from the equation \eqref{eqn-Eu} that
	\begin{align*}
		\bE(u(t)) \leq  \bE(u_0), \;\; \text{for every}\ t\geq 0.
	\end{align*}
	It shows that the energy functional $\bE$ is dissipative in time.
\end{remark}
\begin{remark}
	Let us emphasize here that one can also consider the problem \eqref{eqn-main-problem} with more general nonlinearity. For example, by adding $``\sum_{q_k=1}^M |u|^{q_k-2}u -\sum_{q_k=1}^M \|u\|_{L^{q_k}(\mathcal{O})}^{q_k}"$ on the right-hand side of \eqref{eqn-main-problem}  with $2 \leq q_k <p$, which is an ongoing work.
\end{remark}
We postpone the  proof of Theorem \ref{Thm-main-wellposed} in Section \ref{Sec-Proof-main-Thm}.
\noindent

We then turn our attention to the asymptotic analysis of the nonlinear heat equation \eqref{eqn-main-problem}, discussed in the following two important results:

	\begin{proposition}\label{Prop-Stationary-copy}
		Let us choose and fix 
		$u_0 \in L^p(\bO) \cap H_0^1(\bO)\cap\bM$. Suppose
		\begin{align*}
			u & \in C([0,\infty); L^p(\bO) \cap H_0^1(\bO)\cap \bM),
		\end{align*}
		is the unique strong solution to \eqref{eqn-main-problem} guaranteed by Theorem \ref{Thm-main-wellposed} such that
		\begin{align*}
			\frac{\partial u}{\partial t} \in L^2(0,\infty; L^2(\bO)).
		\end{align*}
		Then, one may extract a sequence of times $\{\tau_n\}_{n\in \N}$, $\lim_{n\to\infty}\tau_n =\infty$, such that
		\begin{align}
			u(\tau_n) \to u_{\infty} \ \text{ in }\  L^p(\bO) \cap H_0^1(\bO), \ \ \bS(u(\tau_n)) \to \bS(u_{\infty}) \ \text{ in }\ \R,\ \text{ as }\ n \to \infty,
		\end{align}
		where $\bS(u)= \norm{\nabla u}_{L^2(\bO)}^2 + \norm{u}_{L^p(\bO)}^p$ and $u_{\infty}\in\bM$ solves 
		\begin{align*}
			\Delta u_{\infty} - \abs{u_{\infty}}^{p-2} u_{\infty}
			+ \bS(u_{\infty}) u_{\infty} = 0 \ \text{ in }\ L^{\frac{p}{p-1}}(\bO)+H^{-1}(\bO).
		\end{align*}
\end{proposition}

\begin{theorem}\label{Thm-main-assymp-copy}
		Suppose $\bU  \in L^p(\bO) \cap H_0^1(\bO)\cap\bM$ is the unique positive solution to the minimization problem \eqref{min-problem}, which also solves the stationary equation \eqref{eqn-stationary-copy}.
		If $u_0 \in L^p(\bO) \cap H_0^1(\bO)\cap \bM^+$, where $$ \bM^+ := \{u\in L^2(\bO) : u \geq 0 \ \text{ and }\  \norm{u}_{L^2(\bO)} = 1\},$$ and
		\begin{align*}
			&u \in C([0,\infty);  L^p(\bO) \cap H_0^1(\bO)\cap \bM^+)
		\end{align*}
		is the unique strong solution to the problem \eqref{eqn-main-problem} (see Theorem \ref{Thm-main-wellposed} and the maximum principle \ref{sec-max-principle}), then
		\begin{equation}\label{eqn-convergence full set-copy}
			u(t) \rightarrow \bU  \ \mbox{ in } \  L^p(\bO) \cap H_0^1(\bO),\ \mbox{ as } \ t\rightarrow \infty.
		\end{equation}
\end{theorem}

\subsection{Literature review}
Firstly, Rybka \cite{PR-06} and Caffarelli and Lin \cite{LC+FL-09} studied the heat equation in $L^2(\Omega)$ projected on the Manifold $M$, where
\begin{align*}
	M = \left\{\bh \in L^2(\Omega) \cap C(\Omega) : \int_\Omega \bh^j(x) dx = K_j,\ 1\le j \le N\right\},
\end{align*}
and $\Omega$ is a bounded domain in $\R^2$. 
%
Rybka proved that, for sufficiently smooth initial data, the following nonlinear heat equation
$$
\frac{\partial u}{\partial t} = \Delta u - \sum_{j=1}^N \lambda_j(u) u^{j-1},
$$
with Neumann boundary condition, admits the unique global solution. The functions $\lambda_j(u)$ are determined so that $\frac{\partial u}{\partial t} \perp \mathrm{span}\{u^{j-1}\}$.
In their seminal work, Caffarelli and Lin \cite{LC+FL-09} established the existence and uniqueness of global energy-conserving solutions to the classical heat equation. They further extended their results to a more general class of singularly perturbed, non-local parabolic systems. Their analysis demonstrated that solutions to these perturbed systems converge strongly to weak solutions of a limiting constrained non-local heat flow, where the target space is singular.
Ma and Cheng \cite{LM+LC-09} studied two kinds of $L^2-$norm preserved non-local heat flows on closed Riemannian manifolds. First, they proved the global existence, stability and asymptotic behaviour of such non-local heat flows, then the gradient estimates of positive solutions to these heat flows. In \cite{LM+LC-13}, Ma and Cheng showed the global existence of positive solutions to the norm-preserving non-local heat flow of the porous-media type equations on the compact Riemannian manifold with positive Cauchy data. Moreover, by applying Sobolev's embedding and the Moser iteration method, the limit of a solution, 
is an eigenfunction of the Laplace operator.
Recently, Brze\'zniak and Hussain \cite{ZB+JH-24} studied the existence and  invariance of a manifold for the  unique global strong solutions of a non-linear heat equation of gradient type with the help of semigroup theory and fixed point arguments. The evolution equation arises from the projection of the Laplace operator, subject to Dirichlet boundary conditions, together with a polynomial nonlinearity of degree $2p - 1$, onto the tangent space of a unit sphere embedded in a Hilbert space. Hussain \cite{JH-23} proved the existence and uniqueness of strong solutions of a constrained heat equation taking values in a Hilbert manifold by Faedo-Galerkin approximations and  compactness method. 
Note that, in the above two works, there is a restriction on the power of nonlinearity (i.e., $p \in \left[2,\frac{2d}{d-2}\right]]$) because of Sobolev's embedding, which we have overcome in this work.

In \cite{PA+PC+BS-24}, Antonelli et al. proved the local and global well-posedness of a nonlinear heat equation, which preserves the $L^2-$norm of the solution on a bounded domain and the Euclidean space for dimensions less than equal to four, by the semigroup idea. Moreover, in the case of an open ball, the authors demonstrated that the unique strong solution corresponding to positive initial data converges to the ground state solution. More recently, Shakarov \cite{BS-24+} also proved the local and global existence and uniqueness of the weak solution to a nonlinear heat equation that forces solutions to stay on  an $L^2-$sphere through a nonlocal term in both bounded domains and whole space. They used the Schauder fixed point method and the Contraction principle in bounded domains with $C^2$ boundary and in the whole space $\R^d$, respectively. Specifically, their power of nonlinearity is restricted from $0$ to $\frac{2d}{d-2}$.  One can check \cite{QD+FL-09} for a numerical application of the method of Caffarelli and Lin \cite{LC+FL-09} and reference therein for applications in other fields like population biology, ecology, material science, etc.

On the other hand, Brze\'zniak et al. \cite{ZB+GD+MM-18} established the existence and uniqueness of global energy-conserving solutions to the Navier-Stokes equations with constrained forcing on $\mathbb{T}^2$ and $\R^2$, employing fixed point methods. Additionally, they demonstrated that the solution to the constrained problem converges to the Euler equations solution in the vanishing viscosity limit $\nu \to 0$, under the condition that the initial vorticity is uniformly bounded in $L^{\infty}(\mathbb{T}^2)$.
The first article in the case of a stochastic version of the above-stated problems \cite{ZB+JH-24, JH-23} was introduced by Brze\'zniak and Hussain \cite{ZB+JH-20}, in which the authors studied the existence of a unique mild solution to a constrained stochastic nonlinear heat equation, perturbed by Stratonovich-type noise, in a smooth bounded domain in two spatial dimensions.
In \cite{ZB+GD-21}, the authors studied the existence of martingale solutions and established pathwise uniqueness of  probabilistic weak solution to the Navier-Stokes equations perturbed by  a multiplicative Gaussian noise, in two-dimensions. Moreover, they have shown the existence of a strong solution by using Yamada-Watanabe-type results.
Further, Brze\'zniak and Cerrai \cite{SC+ZB-23} analyzed a class of stochastic damped wave-type equations in Hilbert spaces and proved their well-posedness under the geometric constraint that the solution takes values on the unit sphere. In addition, they demonstrated that the asymptotic behavior of the solution in the vanishing mass regime corresponds to that of a constrained stochastic parabolic equation.
Very recently, Cerrai and Xie \cite{SC+MX-25} studied the small-mass limit, also known as the Smoluchowski-Kramers diffusion approximation, for a stochastic damped wave equations, whose solution is constrained to live in the unitary sphere of $L^2(0, L)$.

\subsection{Structure of the paper}

The remainder of this manuscript is arranged in the following manner. Section \ref{Sec-Preliminaries} gives some preliminaries, such as the functional settings, linear operator, nonlinear operator, and their properties, which will be helpful in upcoming sections. Next, in Section \ref{Sec-Loc-Lip-and-MBTT}, we prove useful results, such as locally-Lipschitz, local monotonicity, and hemicontinuity properties of the nonlinear operator $\Gn$ defined in \eqref{eqn-nonlinear} (Lemmas \ref{Lem-Fn-loc-Lip}, \ref{Lem-Gn-loc-mono} and \ref{Lemma-lhemicontinuity}). In Section \ref{Sec-Proof-main-Thm}, we begin by employing the Littlewood-Paley decomposition to define a projection operator and construct a special sequence of self-adjoint operators $\{S_m\}_{m \in \mathbb{N}}$ which is uniformly bounded in the operator spaces $\mathcal{L}(L^2(\bO))$, $\mathcal{L}(H_0^1(\bO))$, and $\mathcal{L}(L^p(\bO))$ (Proposition \ref{Prop-Sm}). The combination of projection and $S_m$ operators facilitates a modified Faedo-Galerkin scheme combined with compactness arguments and Minty-Browder type techniques to establish the existence of solutions.
In particular, the Radon-Riesz property (Proposition \ref{Prop-Radon-Riesz}) of uniformly convex Banach spaces plays a crucial role in establishing a strong convergence in $L^p(\bO)\cap H_0^1(\bO)$. Furthermore, we show that the energy $\bE$ (defined in \eqref{Eu}) is dissipative in time. To conclude this section, we show the invariance and prove the uniqueness. 

In the second part of this article, we begin Section \ref{Sec-Asy-anal-Antonelli} by stating the definition of ground states and some results on the energy functional $\bE$, like weakly lower semicontinuity and coercivtiy (Lemma \ref{Lem-E-wls-coer}). Then, we prove a result on the existence of a ground state solution, which is a minimizer to the minimization problem \eqref{eqn-min-energy-problem} (Theorem \ref{Thm-ground-soln}). 
Further, as a consequence of a famous result \cite{TO-92}, we show that the positive classical solution to problem \eqref{eqn-stationary-trans} is unique (Proposition \ref{Prop-uniq-positive}), which implies that the positive ground state (minimizer) is unique. Therefore, utilizing the uniqueness of the positive minimizer, we prove that the unique positive strong solution to the problem \eqref{eqn-main-problem} converges to the unique positive ground state solution in strong topology (Theorem \ref{Thm-main-assymp}). 

Lastly, in Section \ref{Appendix}, we conclude this paper by providing some well-known and important results like the Aubin-Lions Lemma, the Strauss Lemma, the Lions Lemma and the Radon-Riesz property. Thereafter, we provide the  maximum principle for a nonlinear  heat equation with damping  which helps in proving the result Theorem \ref{Thm-main-assymp}.

\section{Preliminaries}\label{Sec-Preliminaries}
In this section, we present the preliminaries used in the proof of Theorem \ref{Thm-main-wellposed}, which includes the function spaces, some linear and nonlinear operators and their properties.

\subsection{Functional setting}
For any $1\leq p < \infty$, we define  $L^p(\bO)$ to be collection of equivalence classes $[f]$ containing   Lebesgue measurable functions $f : \bO \to\mathbb{R}$ such that $\int_{\bO}|f(x)|^pdx<\infty.$ The $L^p-$norm of $f \in L^p(\bO)$ is defined by $\|f\|_{L^p(\bO)}:=\left(\int_{\bO}|f(x)|^pdx\right)^{1/p}$.  For $p=2$, $L^2(\bO)$ is a Hilbert space and the inner product in  $L^2(\bO)$  is denoted by $(\cdot,\cdot)$. 	Moreover, let $H_0^1(\bO)$ denote the Sobolev space (also denoted as $W^{1,2}_0(\bO)$), which is defined as the collection  of equivalence classes of Lebesgue measurable functions $f \in L^2(\bO)$ such that its weak derivative $\frac{\partial f}{\partial x_i} \in L^2(\bO)$ and $f$ is of trace zero.  The $H_0^1(\bO)-$norm is defined by $\norm{f}_{H_0^1(\bO)} := \left(\int_{\bO}|\nabla f(x)|^2dx\right)^{1/2}$ by using the Poincar\'e inequality.
Then, we define  $H^{-1}(\bO) := (H_0^1(\bO))'$, i.e., dual of the Sobolev space $H_0^1(\bO)$, with norm \[\norm{g}_{H^{-1}(\bO)} := \sup\big\{
{\langle g, f\rangle}:  f \in H_0^1(\bO),\ {\norm{f}_{H_0^1(\bO)}}\leq 1\big\}.\]
We denote the second order Hilbertian Sobolev spaces by $H^2(\bO)$.

Let us define the sum  and intersection spaces which will be used in this work. Note that  $L^{p^\prime}(\bO)$ and $H_0^{-1}(\bO)$ are Banach spaces with the norms $\norm{\cdot}_{L^{p^\prime}(\bO)}$ and $\norm{\cdot}_{H_0^{-1}(\bO)}$, respectively, with $\frac{1}{p} + \frac{1}{p'} =1$. Moreover, the intersection $L^p(\bO) \cap H_0^1(\bO)$ is a dense subspace of both $L^{p^\prime}(\bO)$ and $H_0^{-1}(\bO)$ in the corresponding norms.

Then the sum space
\begin{align*}
	L^{p^\prime}(\bO) + H^{-1}(\bO) := \{u_1 + u_2 : u_1 \in L^{p^\prime}(\bO), u_2 \in H^{-1}(\bO)\}
\end{align*}
is a well-defined Banach space with the norm
\begin{align*}
	\norm{u}_{L^{p^\prime}(\bO) + H^{-1}(\bO)}  = \inf\{\norm{u_1}_{L^{p^\prime}(\bO)} + \norm{u_2}_{H^{-1}(\bO)} : u = u_1 + u_2, u_1 \in L^{p^\prime}(\bO), u_2 \in H^{-1}(\bO)\}.
\end{align*}
The intersection space $L^p(\bO) \cap H_0^1(\bO)$ is a Banach space with the norm
\[\norm{u}_{L^p(\bO) \cap H_0^1(\bO)} := \max\{\norm{u}_{L^p(\bO)}, \norm{u}_{H_0^1(\bO)}\},\]
which is equivalent to $\norm{u}_{L^p(\bO)}+ \norm{u}_{H_0^1(\bO)}$ and $\big(\norm{u}_{L^p(\bO)}^2+ \norm{u}_{H_0^1(\bO)}^2\big)^{1/2}$. Moreover, the dual space $L^{p^\prime}(\bO) + H^{-1}(\bO)$ is given by
\begin{align*}
	(L^{p^\prime}(\bO) + H^{-1}(\bO))^\prime \cong L^p(\bO) \cap H_0^1(\bO),
\end{align*}
with the natural pairing
\begin{align*}
	\langle u, f \rangle = \langle u_1, f \rangle + \langle u_2, f \rangle,
\end{align*}
for all $u=u_1 +u_2 \in L^{p^\prime}(\bO) + H^{-1}(\bO)$ and $f\in L^p(\bO) \cap H_0^1(\bO)$. Thus it holds that \cite[cf. Section 2]{RF+HK+HS-05}
\begin{align*}
	\norm{u}_{L^{p^\prime}(\bO) + H^{-1}(\bO)}
	= \sup & \big\{
	{\langle u_1 + u_2, f\rangle}
	:  f \in L^p(\bO) \cap H_0^1(\bO),\ {\norm{f}_{L^p(\bO) \cap H_0^1(\bO)}}  \leq 1 \big\}.
\end{align*}
In the sequel, the notation $a\lesssim b$ means that $a\leq Cb$, for some positive constant $C$. 

\subsection{Linear operator}
Let us define a bilinear form 
\[a: H_0^1(\bO)\times H_0^1(\bO) \to \R \;\; \mbox{by}\; \; a(u,v) := (\nabla u, \nabla v),\; \; \mbox{for}\;\; u,v \in H_0^1(\bO).\]
From the definition of $a(\cdot, \cdot)$, it is clear that $a(\cdot, \cdot)$ is $H_0^1(\bO)-$continuous, i.e., 
\[\abs{a(u,v)} \le \norm{u}_{H_0^1(\bO)} \norm{v}_{H_0^1(\bO)},\ \text{ for all } \ u,v \in H_0^1(\bO).\] 
Hence by the Riesz representation Theorem, there exists a unique linear operator $\A : H_0^1(\bO) \to H^{-1}(\bO)$, such that
\begin{align*}
	a(u,v) =
	{\left\langle\A u, v\right\rangle},\ \mbox{ for all }\ u,v \in H_0^1(\bO).
\end{align*}
Moreover, the form $a(\cdot, \cdot)$ is $H_0^1(\bO)-$coercive, i.e., it satisfies $a(u,u) \geq \alpha \norm{u}^2_{H_0^1(\bO)}$, for all $u\in H_0^1(\bO)$ and some $ \alpha (= 1)>0$. Therefore, by the Lax-Milgram Theorem, the operator $\A : H_0^1(\bO) \to H^{-1}(\bO)$ is an isomorphism. Now, we define an unbounded linear operator $A$ in $L^2(\bO)$ as follows:
\begin{align*}
	A u & := \A u, \ \mbox{ for all } \  u \in D(A):= \{u\in H_0^1(\bO): \A u \in L^2(\bO)\}.
\end{align*}
\begin{lemma}\label{Lem-D(A)}
	Let $A$ be the Dirichlet-Laplacian operator as defined above, then $D(A) = H^2(\bO)\cap H_0^1(\bO)$.
\end{lemma}
\subsection{Nonlinear operator}
Let us choose and fix $2\le p <\infty$, then define a nonlinear operator
\begin{align}\label{Def-Nn}
	\Nn: L^p( \bO)\to L^{p^\prime}( \bO )\;\; \mbox{by} \;\; \Nn(u) := \abs{u}^{p-2} u, \; \; \mbox{where}\;\; p^\prime = \frac{p}{p-1}.
\end{align}
One can show that, \cite[e.g. see Sec. 2.4 and 2.5]{SG+MTM-24+}, the operator $\Nn$ is monotone, given in the sense below. For any $2\le p < \infty$, we have
\begin{align*}
	\langle \Nn(u) - \Nn(v), u-v \rangle
	\geq \int_\bO \left(\abs{u(x)}^{p-1}- \abs{v(x)}^{p-1}\right) (\abs{u(x)} - \abs{v(x)})\, dx \geq 0.
\end{align*}
Furthermore, we also have the following estimate \cite[cf. p. 626]{MTM-22}:
\begin{align}
	\langle \Nn(u) - \Nn(v), u-v \rangle \geq \frac{1}{2} \big\|\abs{u}^{\frac{p-2}{2}}(u-v)\big\|^2_{L^2(\bO)} + \frac{1}{2} \big\|\abs{v}^{\frac{p-2}{2}}(u-v)\big\|^2_{L^2(\bO)}.\label{eqn-mono-2}
\end{align}
Hence, the monotonocity of $\Nn$ is a consequence of \eqref{eqn-mono-2}.
\begin{proposition}[{\cite[Section 2.4]{MTM-21+}}]
	Let $\Nn$ be a nonlinear operator defined in \eqref{Def-Nn}. Then, for any $u,v\in L^p(\bO)$, the following inequality holds:
	\begin{align}\label{eqn-nonlinear-est}
		\langle \Nn(u) - \Nn(v), u-v \rangle \geq \frac{1}{2^{p-2}} \norm{u-v}^{p}_{L^p(\bO)}.
	\end{align}
\end{proposition}
\begin{proof}
	Let us choose and fix $u,v\in L^p(\bO)$. We calculate
	\begin{align}
		\norm{u-v}_{L^p(\bO)}^p
		& = \int_\bO \abs{u(x) - v(x)}^{p-2} \abs{u(x) - v(x)}^2 dx\\
		&\leq 2^{p-3} \int_\bO (\abs{u(x)}^{p-2} + \abs{v(x)}^{p-2}) \abs{u(x) - v(x)}^2 dx\\
		& = 2^{p-3} \big\|\abs{u}^{\frac{p-2}{2}}(u-v)\big\|^2_{L^2(\bO)} + 2^{p-3} \big\|\abs{v}^{\frac{p-2}{2}}(u-v)\big\|^2_{L^2(\bO)}.
	\end{align}
	Hence the proof is immediate from \eqref{eqn-mono-2}.
\end{proof}
\
\section{Local Lipschitz and local monotonicity properties}\label{Sec-Loc-Lip-and-MBTT}

This section will provide preliminary results on local Lipschitzness, local monotonicity, and hemicontinuity properties  of some nonlinear operators, which will help to prove the well-posedness result, Theorem \ref{Thm-main-wellposed}.  Let us first choose and fix $2\le p <\infty$, $d\ge1$ and $\bO$ to be any bounded smooth domain in $\R^d$ for the entire section.

\subsection{Local Lipschitz property} Let us first show that the nonlinear part in our projected problem \eqref{eqn-main-problem} is Lipschitz on balls in the following lemma:
\begin{lemma}\label{Lem-Fn-loc-Lip}
	The map \dela{$\Fn$   $\to  L^{2}(\bO)$ defined by }
	\begin{equation*}
		\begin{aligned}
			\Fn &: L^{2p -2}(\bO)\cap H_0^1(\bO) \ni u  \mapsto   -|u|^{p-2}u +  \norm{\nabla u}_{L^2(\bO)}^2 u + \norm{u}_{L^p(\bO)}^p u  \in L^{2}(\bO)
		\end{aligned}
	\end{equation*}		
	is well-defined and Lipschitz on balls, i.e., for every $R>0,$ there exists a constant $C(R)=C(p,|\bO|,R)>0$ such that, for all
	$u,v \in L^{2p-2}(\bO)\cap H_0^1(\bO) $  with
	$$\|u\|_{L^{2p-2}(\bO)\cap H_0^1(\bO)}, \|v\|_{L^{2p-2}(\bO)\cap H_0^1(\bO)}\leq R,$$ the following inequality holds:
	\begin{align}\label{eqn-lip-1}
		\|\Fn(u)-\Fn(v)\|_{L^2(\bO)}\leq C(R) \left(\|u-v\|_{L^{2p-2}(\bO)}+\|u-v\|_{H_0^1(\bO)}\right).
	\end{align}
\end{lemma}
The proof of result stated above is deduced directly from the three supporting lemmas proved below.
\begin{lemma}\label{Lem-F_1}
	The map
	\begin{equation}\label{eqn-F_1}
		\begin{aligned}
			F_1 : L^{2p -2}(\bO) \ni u  \mapsto   -|u|^{p-2}u   \in L^{2}(\bO)
		\end{aligned}
	\end{equation}		
	is well-defined and
	Lipschitz on balls.
\end{lemma}
\begin{proof}
	For fixed $u,v\in L^{2p-2}(\bO)$. Let us consider $F_1$ 
	\begin{align*}
		\norm{F_1(u)}_{L^2(\bO)} = \norm{-|u|^{p-2}u}_{L^2(\bO)} = \norm{|u|^{p-1}}_{L^2(\bO)} = \norm{u}_{L^{2p-1}(\bO)}^{p-1}<\infty,
	\end{align*} 
	hence	it is well-defined. Denote $\Phi(u) = \abs{u}^{p-2}u$, then
	for some $0<\theta<1$, we have 
	\begin{align*}
		\Phi(u) - \Phi(v) & = \int_0^1 \frac{d}{d\theta}\Phi(\theta u + (1-\theta)v) d\theta = \int_0^1 \Phi^\prime(\theta u + (1-\theta)v) (u-v)d\theta\\
		& = (p-1)(u-v)\int_0^1\abs{\theta u + (1-\theta) v}^{p-2} d\theta \leq (p-1)|u-v|(|u|+|v|)^{p-2}.
	\end{align*}
	Therefore, by H\"older's inequality (with exponents $p-1$ and $(p-1)/(p-2)$), we immediately have
	\begin{align*}
		\norm{F_1(u) - F_1(v)}_{L^2(\bO)}&=\|\abs{u}^{p-2}u - \abs{v}^{p-2}v\|_{L^2(\bO)}\\&\leq (p-1)\||u-v|(|u|+|v|)^{p-2}\|_{L^2(\bO)}\\&\leq (p-1)\|u-v\|_{L^{2p-2}(\bO)}(\|u\|_{L^{2p-2}(\bO)}+\|v\|_{L^{2p-2}(\bO)})^{2p-2}.
	\end{align*}
	By taking  $\|u\|_{L^{2p-2}(\bO)\cap H_0^1(\bO)}, \|v\|_{L^{2p-2}(\bO)\cap H_0^1(\bO)}\leq R$, we get
	\begin{align*}
		\norm{F_1(u) - F_1(v)}_{L^2(\bO)}&\leq (p-1)\|u-v\|_{L^{2p-2}(\bO)}(\|u\|_{L^{2p-2}(\bO)}+\|v\|_{L^{2p-2}(\bO)})^{2p-2}\\
		& \leq C(R)\|u-v\|_{L^{2p-2}(\bO)}.
	\end{align*}
	Thus, the operator \textit{$F_1$ is  Lipschitz on balls} from $L^{2p-2}(\bO)$ to $L^{2}(\bO)$.
\end{proof}

\begin{lemma}\label{Lem-F_2}
	The map
	\begin{equation}\label{eqn-F_2}
		\begin{aligned}
			F_2 : H_0^1(\bO) \ni u  \mapsto   \norm{\nabla u}_{L^2(\bO)}^2 u  \in L^{2}(\bO)
		\end{aligned}
	\end{equation}		
	is well-defined and	Lipschitz on balls.
\end{lemma}
\begin{proof}
	Let us choose $u,v \in H_0^1(\bO)$. Then, we have 
	\begin{align*}
		\norm{F_2(u)}_{L^2(\bO)} = \big\|\norm{\nabla u}_{L^2(\bO)}^2 u\big\|_{L^2(\bO)} = \norm{\nabla u}_{L^2(\bO)}^2 \norm{ u}_{L^2(\bO)}< \infty.
	\end{align*}
	It implies that $F_2$ is well-defined. By the triangle and Poincar\'e's inequalities, we infer
	\begin{align*}
		\norm{ F_2(u) - F_2(v)}_{L^2(\bO)}
		& = \big\|\norm{ u}_{H_0^1(\bO)}^2 u - \norm{ v}_{H_0^1(\bO)}^2 v\big\|_{L^{2}(\bO)}\\
		& = \big\|\norm{ u}_{H_0^1(\bO)}^2 (u -v )+( \norm{ u}_{H_0^1(\bO)}^2  - \norm{ v}_{H_0^1(\bO)}^2) v\big\|_{L^2(\bO)}\\
		& \leq \norm{ u}_{H_0^1(\bO)}^2 \norm{u-v}_{L^2(\bO)} + \big(\norm{ u}_{H_0^1(\bO)}^2 - \norm{ v}_{H_0^1(\bO)}^2\big) \norm{v}_{L^2(\bO)}\\
		& =  \norm{ u}_{H_0^1(\bO)}^2 \norm{u-v}_{L^2(\bO)}\\
		& \quad + \big(\norm{ u}_{H_0^1(\bO)} + \norm{ v}_{H_0^1(\bO)}\big)\big(\norm{ u}_{H_0^1(\bO)} - \norm{ v}_{H_0^1(\bO)}\big) \norm{v}_{L^2(\bO)}\\
		&\leq \left(\frac{1}{\sqrt{\lambda_1}}\norm{ u}_{H_0^1(\bO)}^2+ (\norm{ u}_{H_0^1(\bO)} + \norm{ v}_{H_0^1(\bO)})\norm{v}_{L^2(\bO)}\right) \norm{ u - v}_{H_0^1(\bO)},
	\end{align*}
	where $\lambda_1$ is the smallest eigenvalue of the Dirichlet-Laplacian operator. 	
	By taking \break $\|u\|_{L^{2p-2}(\bO)\cap H_0^1(\bO)}$, $\|v\|_{L^{2p-2}(\bO)\cap H_0^1(\bO)}\leq R$ and using Poincar\'e's inequality, we obtain
	\begin{align*}
		\norm{ F_2(u) - F_2(v)}_{L^2(\bO)}
		&\leq \left(\frac{1}{\sqrt{\lambda_1}}\norm{ u}_{H_0^1(\bO)}^2+ (\norm{ u}_{H_0^1(\bO)} + \norm{ v}_{H_0^1(\bO)})\norm{v}_{L^2(\bO)}\right) \norm{ u - v}_{H_0^1(\bO)}\\
		&  \leq C(R)\norm{ u - v}_{H_0^1(\bO)},
	\end{align*}
	so that, the operator \textit{$F_2$ is Lipschitz on balls} from $H_0^1(\bO)$ to $L^2(\bO)$.
\end{proof}
\begin{lemma}\label{Lem-F_3}
	The map
			\[	F_3 : L^{2p -2}(\bO) \ni u  \mapsto   \norm{u}_{L^p(\bO)}^p u  \in L^{2}(\bO)\]
			is well-defined and Lipschitz on balls.
		\end{lemma}
		\begin{proof}
			Assume $u,v \in L^{2p -2}(\bO)$, then we have 
			\begin{align*}
				\norm{F_3(u)}_{L^2(\bO)} = \big\|\norm{u}_{L^p(\bO)}^p u\big\|_{L^2(\bO)} = \norm{u}_{L^p(\bO)}^p \norm{u}_{L^2(\bO)} <\infty.
			\end{align*}
			It implies that $F_3$ is a well-defined map. Next, an application of 
			H\"older's inequality (with exponents $p$ and $p/(p-1)$) asserts
			\begin{align*}
				&	\norm{ F_3(u) - F_3(v)}_{L^2(\bO)} \\
				& =\big\|\norm{u}_{L^p(\bO)}^{p}u - \norm{v}_{L^p(\bO)}^{p}v\big\|_{L^{2}(\bO)}
				= \big\|\norm{u}_{L^p(\bO)}^{p}(u -v )+( \norm{u}_{L^p(\bO)}^{p} - 	\norm{v}_{L^p(\bO)}^{p})v\big\|_{L^{2}(\bO)}\\
				& \leq \norm{u}_{L^p(\bO)}^{p} \norm{u -v}_{L^{2}(\bO)} +  (\norm{u}_{L^p(\bO)}^{p} - \norm{v}_{L^p(\bO)}^{p})\norm{v}_{L^{2}(\bO)}.
			\end{align*}
			Consider $\Phi(u) = \norm{u}_{L^p(\bO)}^p$, then we have 
			\begin{align*}
				\Phi(u) - \Phi(u) = p\int_0^1 (\abs{\theta u + (1-\theta)v}^{p-2}(\theta u + (1-\theta)v), u-v ) d \theta.
			\end{align*}
			Therefore, the above equality implies
			\begin{align*}
				&\norm{ F_3(u) - F_3(v)}_{L^2(\bO)}\\
				&=\norm{u}_{L^p(\bO)}^{p} \norm{u -v}_{L^{2}(\bO)} +p \int_0^1 (\abs{\theta u + (1-\theta)v}^{p-2}(\theta u + (1-\theta)v), u-v )d\theta\norm{v}_{L^{2}(\bO)}\\
				&\leq \norm{u}_{L^p(\bO)}^{p} \norm{u -v}_{L^{2}(\bO)} +p\|u-v\|_{L^p(\bO)}\left(\|u\|_{L^p(\bO)}+\|v\|_{L^p(\bO)}\right)^{p-1}\norm{v}_{L^{2}(\bO)}.
			\end{align*}
			By taking  $\|u\|_{L^{2p-2}(\bO)\cap H_0^1(\bO)}, \|v\|_{L^{2p-2}(\bO)\cap H_0^1(\bO)}\leq R$ and the Lebesgue embedding $L^{2p-2}(\bO) \embed L^p(\bO)\embed L^2(\bO)$ for $2\leq p < \infty$, we deduce
			\begin{align*}
				&\norm{ F_3(u) - F_3(v)}_{L^2(\bO)}\\
				&\leq \norm{u}_{L^p(\bO)}^{p} \norm{u -v}_{L^{2}(\bO)} +p\|u-v\|_{L^p(\bO)}\left(\|u\|_{L^p(\bO)}+\|v\|_{L^p(\bO)}\right)^{p-1}\norm{v}_{L^{2}(\bO)}\\
				& \leq C(R)\|u-v\|_{L^p(\bO)}.
			\end{align*}
			It implies that the operator \textit{$F_3$ is Lipschitz on balls} from $L^{p}(\bO)$ to $L^{2}(\bO)$.
		\end{proof}
		\begin{proof}[Proof of Lemma \ref{Lem-Fn-loc-Lip}]
			Let us fix $u,v\in L^{2p-2}(\bO)\cap H_0^1(\bO)$ and set
			\begin{align}\label{eqn-F_n}
				\Fn(u) &:=  -|u|^{p-2}u +  \norm{\nabla u}_{L^2(\bO)}^2 u + \norm{u}_{L^p(\bO)}^p u  =: F_1(u) + F_2(u) + F_3 (u) .
			\end{align}
			By combining the proofs of Lemmas \ref{Lem-F_1}, \ref{Lem-F_2}, \ref{Lem-F_3}, one can deduce that
			\begin{align*}
				\norm{\Fn(u)-\Fn(v)}_{L^2(\bO)} \leq C(R) \left(\|u-v\|_{L^{2p-2}(\bO)}+\|u-v\|_{H_0^1(\bO)}\right),
			\end{align*}
			where  $\|u\|_{L^{2p-2}(\bO)\cap H_0^1(\bO)}, \|v\|_{L^{2p-2}(\bO)\cap H_0^1(\bO)}\leq R$.
			Hence the operator $\Fn$ is Lipschitz on balls from $L^{2p-2}(\bO)\cap H_0^1(\bO)$ to $L^2(\bO)$.
		\end{proof}
		
		\begin{remark}\label{Rmk-Gn-loc-Lip}
			Let us define the operator 
			\begin{align}
				L^{p}(\bO)\cap H_0^1(\bO) \ni u \mapsto \Gn(u) 
				& := \Delta u -|u|^{p-2}u \\
				& \quad+  \norm{\nabla u}_{L^2(\bO)}^2 u + \norm{u}_{L^p(\bO)}^p u \in L^{p'}(\bO)+H^{-1}(\bO).\label{eqn-op-G}
			\end{align}
			If we consider the operator $\Gn: L^{2p-2}\cap D(A) \to L^2(\bO)$, defined in \eqref{eqn-op-G}, then by Lemma \ref{Lem-Fn-loc-Lip}, for any $u,v \in L^{2p-2}\cap D(A)$, we deduce
			\begin{align*}
				\norm{\Gn(u) - \Gn(v)}_{L^2(\bO)} 
				& \leq \norm{\Delta (u-v)}_{L^2(\bO)} + \norm{\Fn(u) - \Fn(v)}_{L^2(\bO)}\\
				& \leq \norm{u-v}_{D(A)} + C(p,\abs{\bO},R) \left(\|u-v\|_{L^{2p-2}(\bO)}+\|u-v\|_{H_0^1(\bO)}\right)\\
				& \leq C(R) \left(\|u-v\|_{L^{2p-2}(\bO)}+\|u-v\|_{D(A)}\right).
			\end{align*}
			Thus, the operator $\Gn$, i.e., $\Gn= \Delta + \Fn$, is well-defined and Lipschitz on balls from $L^{2p-2}\cap D(A)$ to $L^2(\bO)$.
		\end{remark}

		\subsection{Local monotonicity }\label{sec5}
		Let us now prove that the operator $\Gn$, defined in \eqref{eqn-op-G}, is well-defined and fully local monotone.
		\begin{lemma}\label{Lem-Gn-loc-mono} 
			The  map $\Gn$
			defined in \eqref{eqn-op-G}	is well-defined and fully local monotone, i.e., for every $R>0$ there exists  $C(R)=C(p, |\bO|,R)>0$ such that
			\begin{align}\label{eqn-loc-mon}
				\langle \Gn(u)-\Gn(v),u-v\rangle\leq C(p, |\bO|,R)\|u-v\|_{L^2(\bO)}^2,
			\end{align}
			provided that   $\|u\|_{L^{p}(\bO)\cap H_0^1(\bO)}, \|v\|_{L^{p}(\bO)\cap H_0^1(\bO)}\leq R$.
		\end{lemma}
		
		\begin{proof}
			\textit{Step 1.} For the well-definedness of $\Gn$, we need to show that for any $u\in L^{p}(\bO)\cap H_0^1(\bO)$,
			$\Gn(u) \in L^{p'}(\bO)+H^{-1}(\bO)$. Let us consider each term one-by-one from the equation \eqref{eqn-op-G}. For $v\in L^{p}(\bO)\cap H_0^1(\bO)$, first consider
			\begin{align*}
				\langle \Delta u, v\rangle \leq \norm{u}_{H_0^1(\bO)} \norm{v}_{H_0^1(\bO)}.
			\end{align*}
			Now using H\"older's inequality (with exponent $p/p-1$ and $p$), we estimate the second term
			\begin{align*}
				\langle \abs{u}^{p-2}u, v\rangle \leq \norm{u}_{L^p(\bO)}^{p-1} \norm{v}_{L^p(\bO)}.
			\end{align*}
			Next, let us estimate the projected terms
			\begin{align*}
				\langle \norm{\nabla u}_{L^2(\bO)}^{2}u, v\rangle \leq \norm{\nabla u}_{L^2(\bO)}^{2}\|u\|_{L^2(\bO)}\|v\|_{L^2(\bO)}. 
			\end{align*}
			Similarly, we deduce
			\begin{align*}
				\langle \norm{u}_{L^p(\bO)}^{p}u, v\rangle \leq \norm{u}_{L^p(\bO)}^{p}\norm{ u}_{L^2(\bO)} \norm{v}_{L^2(\bO)}.
			\end{align*}
			Now, it is immediate from the above estimates that
			\begin{align*}
				\norm{\Gn(u)}_{L^{p'}(\bO)+H^{-1}(\bO)} = \sup \{\langle \Gn(u), v\rangle : \norm{v}_{L^{p}(\bO)\cap H_0^1(\bO)} = 1 \} \leq C(p,\norm{u}_{L^{p}(\bO)\cap H_0^1(\bO)}).
			\end{align*}
			Hence $\Gn$ is well-defined.
			
			\vskip 1mm
			\noindent
			\textit{Step 2. }We consider
			\begin{align}
				& \langle \Gn(u)-\Gn(v),u-v\rangle\\
				&= -\|\nabla(u-v)\|_{L^2(\bO)}^2-\langle|u|^{p-2}u -|v|^{p-2}v, u-v\rangle\\
				& \quad+( \norm{\nabla u}_{L^2(\bO)}^2 u- \norm{\nabla v}_{L^2(\bO)}^2 v,u-v)+(\norm{u}_{L^p(\bO)}^p u-\norm{v}_{L^p(\bO)}^p v,u-v)\\
				&\leq -\|\nabla(u-v)\|_{L^2(\bO)}^2-\frac{1}{2}\||u|^{\frac{p-2}{2}}(u-v)\|_{L^2(\bO)}^2-\frac{1}{2}\||v|^{\frac{p-2}{2}}(u-v)\|_{L^2(\bO)}^2\\
				&\quad+\norm{\nabla u}_{L^2(\bO)}^2\|u-v\|_{L^2(\bO)}^2+\underbrace{(\norm{\nabla u}_{L^2(\bO)}^2-\norm{\nabla v}_{L^2(\bO)}^2)(v,u-v)}_{I_1}\\
				&\quad+\norm{u}_{L^p(\bO)}^p\|u-v\|_{L^2(\bO)}^2+\underbrace{(\norm{u}_{L^p(\bO)}^p-\norm{v}_{L^p(\bO)}^p)(v,u-v)}_{I_2},\label{eqn-loc}
			\end{align}
			where we have used \eqref{eqn-mono-2}. Now, let us estimate each $I_i$, for $i=1,2$.
			
			\vskip 1mm
			\noindent
			\textit{Step 3.} First let us work with $I_1$
			\begin{align*}
				I_1& = (\norm{\nabla u}_{L^2(\bO)}^2-\norm{\nabla v}_{L^2(\bO)}^2)(v,u-v)\\
				& = (\norm{\nabla u}_{L^2(\bO)} + \norm{\nabla v}_{L^2(\bO)})(\norm{\nabla u}_{L^2(\bO)} -\norm{\nabla v}_{L^2(\bO)})(v,u-v)\\
				& \leq (\norm{\nabla u}_{L^2(\bO)} + \norm{\nabla v}_{L^2(\bO)})\norm{\nabla (u - v)}_{L^2(\bO)}\norm{ v}_{L^2(\bO)}\norm{ u-v}_{L^2(\bO)}\\
				& \leq \frac{1}{2}\big[\norm{\nabla (u - v)}_{L^2(\bO)}^2 + (\norm{\nabla u}_{L^2(\bO)} + \norm{\nabla v}_{L^2(\bO)})^2 \norm{ v}_{L^2(\bO)}^2\norm{ u-v}_{L^2(\bO)}^2\big].
			\end{align*}
			By the same approach used in the proof of Lemma \ref{Lem-F_3}
			and H\"older's inequality (with exponent $2$ and $2$), we estimate the $I_2$ as
			\begin{align*}
				I_2 & = (\norm{u}_{L^p(\bO)}^p-\norm{v}_{L^p(\bO)}^p)(v,u-v)\\
				& \leq p \int_0^1 (\abs{\theta u + (1-\theta)v}^{p-2}(\theta u + (1-\theta)v), u-v) d\theta \norm{v}_{L^2(\bO)}\norm{u-v}_{L^2(\bO)}\\
				& \leq
				p2^{p-2} (\abs{u}^{p-1} + \abs{v}^{p-1}, \abs{u-v})\norm{v}_{L^2(\bO)}\norm{u-v}_{L^2(\bO)} \\
				& = p2^{p-2} \big[\big(\abs{u}^{\frac{p-1}{2}} \abs{u-v}, \abs{u}^{\frac{p-1}{2}}\big) + \big(\abs{v}^{\frac{p-1}{2}} \abs{u-v}, \abs{v}^{\frac{p-1}{2}}\big)\big] \norm{v}_{L^2(\bO)}\norm{u-v}_{L^2(\bO)}\\
				& \leq \frac{1}{4} \Big(\big\|\abs{u}^{\frac{p-1}{2}} \abs{u-v}\big\|_{L^2(\bO)}^2 + \big\|\abs{v}^{\frac{p-1}{2}} \abs{u-v}\big\|_{L^2(\bO)}^2 \Big)\\
				& \quad + p^2 2^{2p-4}\abs{\bO}^{\frac{1}{p}}\big(\norm{u}_{L^p(\bO)}^{p-1} + \norm{v}_{L^p(\bO)}^{p-1}\big)\norm{v}_{L^2(\bO)}^2\norm{u-v}_{L^2(\bO)}^2.
			\end{align*}
			\textit{Step 4.} Now substituting the estimates of $I_1$ and $I_2$ in \eqref{eqn-loc}, we obtain
			\begin{align}
				&	\langle \Gn(u)-\Gn(v),u-v\rangle\\
				&\leq -\frac{1}{2}\|\nabla(u-v)\|_{L^2(\bO)}^2-\frac{1}{4}\||u|^{\frac{p-2}{2}}(u-v)\|_{L^2(\bO)}^2-\frac{1}{4}\||v|^{\frac{p-2}{2}}(u-v)\|_{L^2(\bO)}^2\\
				&\quad+ \bigg[\norm{\nabla u}_{L^2(\bO)}^2 + \frac{1}{2}(\norm{\nabla u}_{L^2(\bO)} + \norm{\nabla v}_{L^2(\bO)})^2 \norm{ v}_{L^2(\bO)}^2\\
				&\qquad \quad+\norm{u}_{L^p(\bO)}^p + p^2 2^{2p-4}\abs{\bO}^{\frac{1}{p}}\left(\norm{u}_{L^p(\bO)}^{p-1} + \norm{v}_{L^p(\bO)}^{p-1}\right)\norm{v}_{L^2(\bO)}^2\bigg]\norm{u-v}_{L^2(\bO)}^2\\
				&\leq \bigg[\norm{\nabla u}_{L^2(\bO)}^2 + \frac{1}{2}(\norm{\nabla u}_{L^2(\bO)} + \norm{\nabla v}_{L^2(\bO)})^2 \norm{ v}_{L^2(\bO)}^2\\
				&\qquad \quad+\norm{u}_{L^p(\bO)}^p + C(p,\abs{\bO})\Big(\norm{u}_{L^p(\bO)}^{p-1} + \norm{v}_{L^p(\bO)}^{p-1}\Big)\norm{v}_{L^2(\bO)}^2 \bigg]\norm{u-v}_{L^2(\bO)}^2\label{eqn-G-mono-inequality}\\
				& \leq C(R)\norm{u-v}_{L^2(\bO)}^2.
			\end{align} 
			where $\|u\|_{L^{p}(\bO)\cap H_0^1(\bO)}, \|v\|_{L^{p}(\bO)\cap H_0^1(\bO)}\leq R$. Hence the operator $\Gn: L^{p}(\bO)\cap H_0^1(\bO) \to    L^{p'}(\bO)+H^{-1}(\bO)$ is fully local monotone. 
		\end{proof}
		
		\begin{remark}\label{Rmk-Gn-loc-mono}
			Notice that if we consider the operator $\Gn$ without the projected terms, i.e., without the term $$\norm{\nabla u}_{L^2(\bO)}^2 u + \norm{u}_{L^p(\bO)}^p u,$$ where $\Gn$ is defined in \eqref{eqn-op-G}, then from the proof of Lemma \ref{Lem-Gn-loc-mono}, in particular using \eqref{eqn-loc},
			for any $u,v \in L^{p}(\bO)\cap H_0^1(\bO)$, we obtain
			\begin{align*}
				\langle \Gn(u)-\Gn(v),u-v\rangle & \leq -\|\nabla(u-v)\|_{L^2(\bO)}^2-\frac{1}{2}\||u|^{\frac{p-2}{2}}(u-v)\|_{L^2(\bO)}^2-\frac{1}{2}\||v|^{\frac{p-2}{2}}(u-v)\|_{L^2(\bO)}^2\\
				& \leq 0.
			\end{align*}
			Thus, the operator $\Gn$ is monotone  from $L^{p}(\bO)\cap H_0^1(\bO) \to L^{p'}(\bO)+H^{-1}(\bO)$.
		\end{remark}
		
		\begin{remark}
			The relation \eqref{eqn-loc-mon} satisfied by $\Gn$ in  Lemma \ref{Lem-Gn-loc-mono} is  of fully local monotone type, as  we need to assume both  $\|u\|_{L^{p}(\bO)\cap H_0^1(\bO)}$ and $ \|v\|_{L^{p}(\bO)\cap H_0^1(\bO)}$ are bounded by $R$.
		\end{remark}
		
		\subsection{Hemicontinuity} Let us now prove that the operator $\Gn$, defined in \eqref{eqn-op-G}, is hemicontinuous \cite[see Definition 1.2]{VB-93}.
		
		\begin{lemma}\label{Lemma-lhemicontinuity}
			The map $\Gn: L^{p}(\bO)\cap H_0^1(\bO) \to  L^{p'}(\bO)+H^{-1}(\bO)$, defined in \eqref{eqn-op-G}, is hemicontinuous, i.e., for every $\psi, \zeta, \eta \in L^{p}(\bO)\cap H_0^1(\bO)$
			\begin{align}\label{eqn-hemicont}
				\abs{\langle \Gn(\psi + \lambda\zeta) - \Gn(\psi), \eta \rangle} \to 0
				\text{ as }\ \lambda\to 0.
			\end{align}
		\end{lemma}
		
		\begin{proof}
			To prove hemicontinuity of $\Gn$, it is enough to establish that $\Gn$ is demicontinuous \cite[see Theorem 1]{TK-64}, i.e.,
			\begin{align*}
				\abs{\langle \Gn(\psi_n) - \Gn(\psi), \eta \rangle} \to 0,
			\end{align*}
			for any $\psi_n \to \psi$ in $L^{p}(\bO)\cap H_0^1(\bO)$ and for every $\eta \in L^{p}(\bO)\cap H_0^1(\bO)$.
			
			Let us choose $\{\psi_n\}_{n\in\N} \in L^p(\bO) \cap H_0^1(\bO)$ such that $ \psi_n \to \psi$ in $L^p(\bO) \cap H_0^1(\bO)$. Suppose $\eta \in L^p(\bO) \cap H_0^1(\bO)$ and consider
			\begin{align*}
				&\abs{\langle \Gn(\psi_n) - \Gn(\psi), \eta \rangle}\\
				& \leq  \abs{\langle \Delta(\psi_n - \psi), \eta \rangle} + \abs{\langle \abs{\psi_n}^{p-2}\psi_n - \abs{\psi}^{p-2}\psi, \eta \rangle}\\
				&\quad + \big|\langle \norm{\nabla\psi_n}_{L^2(\bO)}^2\psi_n - \norm{\nabla\psi}_{L^2(\bO)}^2\psi, \eta \rangle\big| + \big|\langle \norm{\psi_n}_{L^p(\bO)}^p\psi_n - \norm{\psi}_{L^p(\bO)}^p\psi, \eta \rangle\big|\\
				& =: \mathcal{I}_1 + \mathcal{I}_2 +\mathcal{I}_3 + \mathcal{I}_4.
			\end{align*}
			Now, we estimate each $\mathcal{I}_i$, for $i =1,\dots,4$. Let us estimate one-by-one as
			\begin{align*}
				\mathcal{I}_1 = \abs{\langle \Delta(\psi_n - \psi), \eta \rangle}  \leq \norm{\psi_n - \psi}_{H_0^1(\bO)} \norm{\eta}_{H_0^1(\bO)} \to 0\ \text{as}\ n \to \infty.
			\end{align*}
			Next, using H\"older's inequality twice (with exponent $p$ and $p/p-1$, and then $p-1$ and $p-1/p-2$), we deduce 
			\begin{align*}
				\mathcal{I}_2 & = \abs{\langle \abs{\psi_n}^{p-2}\psi_n - \abs{\psi}^{p-2}\psi, \eta \rangle} \leq (p-1)\norm{\psi_n - \psi}_{L^p(\bO)} \big(\norm{\psi_n}_{L^p(\bO)} + \norm{\psi}_{L^p(\bO)}\big)^{p-2} \norm{\eta}_{L^p(\bO)}\\
				& \to 0\ \text{ as }\ n \to \infty.
			\end{align*}
			Further, note that
			\begin{align*}
				\mathcal{I}_3 & = \big|\langle \norm{\nabla\psi_n}_{L^2(\bO)}^2\psi_n - \norm{\nabla\psi}_{L^2(\bO)}^2\psi, \eta \rangle\big|\\
				& \leq \norm{\nabla\psi_n}_{L^2(\bO)}^2 \norm{\psi_n - \psi}_{L^2(\bO)}\norm{\eta}_{L^2(\bO)} + \big(\norm{\nabla\psi_n}_{L^2(\bO)}^2 - \norm{\nabla\psi}_{L^2(\bO)}^2 \big)\norm{\psi}_{L^2(\bO)} \norm{\eta}_{L^2(\bO)}\\
				& \leq \norm{\nabla\psi_n}_{L^2(\bO)}^2 \norm{\psi_n - \psi}_{L^2(\bO)}\norm{\eta}_{L^2(\bO)}\\
				& \quad + \big(\norm{\nabla\psi_n}_{L^2(\bO)} + \norm{\nabla\psi}_{L^2(\bO)} \big)\norm{\nabla\psi_n - \nabla\psi}_{L^2(\bO)}\norm{\psi}_{L^2(\bO)} \norm{\eta}_{L^2(\bO)}\\
				& \to 0\ \text{as}\ n \to \infty.
			\end{align*}
			By the same approach used in the proof of Lemma \ref{Lem-F_3} and
			H\"older's inequality (with exponent $2$ and $2$),  we estimate $I_4$ as
			\begin{align*}
				\mathcal{I}_4 & = |\langle \norm{\psi_n}_{L^p(\bO)}^p\psi_n - \norm{\psi}_{L^p(\bO)}^p\psi, \eta \rangle| \\
				& = |\langle \norm{\psi_n}_{L^p(\bO)}^p(\psi_n - \psi) + (\norm{\psi_n}_{L^p(\bO)}^p - \norm{\psi}_{L^p(\bO)}^p)\psi, \eta \rangle| \\
				& \leq \norm{\psi_n}_{L^p(\bO)}^p\norm{\psi_n - \psi}_{L^2(\bO)}\norm{\eta}_{L^2(\bO)}  + (\norm{\psi_n}_{L^p(\bO)}^p - \norm{\psi}_{L^p(\bO)}^p) \norm{\psi}_{L^2(\bO)} \norm{\eta}_{L^2(\bO)} \\
				& \leq \norm{\psi_n}_{L^p(\bO)}^p\norm{\psi_n - \psi}_{L^2(\bO)}\norm{\eta}_{L^2(\bO)}\\
				&\quad  + p(\norm{\psi_n}_{L^p(\bO)} + \norm{\psi}_{L^p(\bO)})^{p-1} \norm{\psi_n - \psi}_{L^p(\bO)} \norm{\psi}_{L^2(\bO)} \norm{\eta}_{L^2(\bO)} \\
				& \to 0\ \text{as}\ n \to \infty.
			\end{align*}
			Hence it is immediate from the above four estimates that$\abs{\langle \Gn(\psi_n) - \Gn(\psi), \eta \rangle} \to 0$ for every $\eta \in L^{p}(\bO)\cap H_0^{1}(\bO)$, which concludes that $\Gn$ is demicontinuous.
		\end{proof}

		\section{Proof of Theorem \ref{Thm-main-wellposed}}\label{Sec-Proof-main-Thm}
		This section is dedicated for proving the well-posedness results to the problem \eqref{eqn-main-problem}, where we  introduce a projection and a sequence of self-adjoint operators that play a fundamental role in the Faedo-Galerkin approximation used to establish existence, uniqueness, and the invariance in a manifold. To begin, we choose and fix $2\leq p<\infty$ and assume $\bO$ to be a bounded domain in $\R^d$, throughout the section.

		\subsection{Projection operator}\label{subsec-proj-op}
		We define a sequence of self-adjoint operators generated through the Littlewood-Paley decomposition corresponding to the operator to $A$.
		Let $S=-\Delta=A$, where $A$ is the Dirichlet-Laplacian. The compactness of the resolvent of $S$ ensures the existence of a complete orthonormal basis $\{w_n\}$ in $L^2(\bO)$, consisting of eigenfunctions associated with a sequence of strictly positive eigenvalues $\{\lambda_n\}$, increasing to infinity, such that
		\begin{align}\label{eqn-s-rep}
			Sv := \sum_{n=1}^{\infty}\lambda_n(v, w_n) w_n,\ \ v\in D(S):=\bigg\{x\in L^2(\bO):\sum_{n=1}^{\infty}\lambda_n^2|(x,w_n)|^2<\infty\bigg\}.
		\end{align}
		In addition, $S$ is a strictly positive, self-adjoint operator that commutes with $A$, and for large enough $k$ ($k>\frac{d}{2}$), we have the continuous embedding $D(S^k)\embed L^p(\bO)\cap H_0^1(\bO)$. By the functional calculus \cite[see e.g.]{EZ-12}, we define the operators $P_m:L^2(\bO)\to L^2(\bO)$ by
		\begin{align}\label{eqn-P_m-S}
			P_m:=\mathds{1}_{(0,2^{m+1})}(S)\ \text{ for }\ m\in\mathbb{N}_0=\mathbb{N}\cup\{0\}.
		\end{align}
		Since $S$ has the representation given in \eqref{eqn-s-rep}, we observe that $P_m$ is the orthogonal projection from $L^2(\bO)$ to $V_m:=\mathrm{span}\left\{w_n:n\in\N, \lambda_n<2^{m+1}\right\},$  and
		\begin{align}
			P_m v=\sum_{\lambda_n<2^{m+1}}(v,w_n)w_n,\ v\in L^2(\bO). \label{def-projection}
		\end{align}
		Note that $w_n\in\bigcap_{k\in\N}D(S^k)$, for $n\in \N$. Since $ D(S) \embed H_0^1(\bO)$, we infer that $V_m$ is a closed subspace of $H_0^1(\bO)$ for $m\in\N$.
		Using the fact that the operators $S$ and $A$ commute, we immediately deduce that $P_m$ and $A^{1/2}$ commute. Therefore, we have
		\begin{align*}
			\|P_m v\|_{L^2(\bO)} & \leq \|v\|_{L^2(\bO)}, \ v\in L^2(\bO),\\
			\|P_m v\|_{H_0^1(\bO)}^2 & = \|A^{1/2}P_m v\|_{L^2(\bO)}^2 = \|P_mA^{1/2}v\|_{L^2(\bO)}^2 \leq \|A^{1/2}v\|_{L^2(\bO)}^2 = \|v\|_{H_0^1(\bO)}^2,\ v\in H_0^1(\bO).
		\end{align*}
		Moreover, we have
		\begin{align*}
			\lim_{m\to\infty}\|P_m v-v\|_{L^2(\bO)}=0,\  v\in L^2(\bO)\ \text{ and }\ \lim_{m\to\infty}\|P_m v-v\|_{H_0^1(\bO)}=0,  \ v\in H_0^1(\bO).
		\end{align*}
		Unfortunately, the operators $P_m$, for $m\in\N$, does not enjoy uniformly boundedness as maps from $L^p(\bO)$ to itself. Since we are expecting to deal with $L^p(\bO)\cap H_0^1(\bO)-$valued solutions, this property becomes essential for establishing a priori estimates in the $L^p-$norm. To address this issue, in the next proposition, we construct a sequence of self-adjoint operators $\{S_m\}_{m\in\N}$ which enjoys the required properties.
		
		A similar construction has been used in the works \cite{ZB+BF+MZ-24, ZB+FH+UM-20, ZB+FH+LW-19, LH-18}. 
		Particularly, in \cite[Section 3]{LH-18}, the author constructed a sequence of self-adjoint operator corresponding to $(I-\Delta_H)$ (where $\Delta_H$ denotes the Hodge-Laplacian) to extract a solution of a stochastic nonlinear Maxwell equation, by estimates in $L^q$ for some $q > 2$. 
		
		This suggests that the approach used in this work may greatly expand the range of applications for the traditional Faedo–Galerkin method.
		
		\begin{proposition}\label{Prop-Sm}
			For every $\psi\in L^p(\bO)\cap H_0^1(\bO)$, there exists a sequence $\{S_m\}_{m\in\N}$ of self-adjoint operators $S_m:L^2(\bO)\to V_m$ for $m\in\N$  such that 
			$$S_m\psi\to \psi\ \text{ in }\ L^p(\bO)\cap H_0^1(\bO),\ \mbox{ as }\ m\to\infty,$$ 
			and the following uniform norm estimates hold:
			\begin{align}
				\sup_{m\in\N}\|S_m\|_{\bL(L^2(\bO))}\leq 1, \ 	\sup_{m\in\N}\|S_m\|_{\bL(H_0^1(\bO))}\leq 1,\
				\sup_{m\in\N}\|S_m\|_{\bL(L^p(\bO))}<\infty.
			\end{align}
		\end{proposition}
		
		\begin{remark}
			Observe from the definition of $P_m$, given in \eqref{eqn-P_m-S}, that indeed $S_m$ can be interpreted as a smoothed approximation of the characteristic functions $\mathds{1}_{(0,2^{m+1})}$, which allows us to prove the uniform $L^p-$boundedness of the sequence $S_m$ by the spectral multiplier theorems. The proof given below is motivated from \cite[Proposition 5.2]{ZB+BF+MZ-24}.
			In \cite{ZB+FH+LW-19}, authors proved the same result by the abstract Littlewood-Paley theory without using the spectral multiplier theorems. The proof can be traced back to \cite[Proposition 10]{ZB+FH+UM-20}, but here we do not use the results from \cite{PCK+MU-15}, instead we use the classical estimate from \cite{EMO-05}. Similar proof given below can be found in \cite[Proposition 5.2]{ZB+FH+LW-19}, \cite[Proposition 5.2]{ZB+BF+MZ-24}.
		\end{remark}
		
		\begin{proof}[Proof of Proposition \ref{Prop-Sm}]
			Choose a function $\rho\in C_0^{\infty}(0,\infty)$ supported in $[\frac{1}{2},2]$ and\break $\sum_{m\in\mathbb{Z}}\rho(2^{-m}t)=1,$   $t>0$. Let $m\in\N_0$ be fixed, we define 
			\begin{align*}
				s_m:(0,\infty)\to\mathbb{R},\ \ s_m(\gamma):=\sum_{n=-\infty}^m\rho(2^{-n}\gamma)
			\end{align*}
			and we see that
			\begin{align*}
				s_m(\gamma)=\left\{
				\begin{array}{cl}1 & \gamma\in(0,2^m),\\
					\rho(2^{-m}\gamma) & \gamma\in[2^m,2^{m+1}),\\
					0 & \gamma\geq 2^{m+1}.
				\end{array}
				\right.
			\end{align*}
			With the help of the self-adjoint functional calculus, we define the operator $S_m := s_m(S)$. This, in turn, yields the representation
			\begin{align}\label{def-S_m-operator}
				S_m v = \sum_{\lambda_n<2^{m}}(v,w_n)w_n + \sum_{\lambda_n\in[2^m,2^{m+1})}\rho(2^{-m}\lambda_n)(v, w_n)w_n,\ v\in L^2(\bO),
			\end{align}
			from which it is immediate that the range of $S_m$ is contained in $V_m$. Since, for each $m\in\N$, the function $s_m$ is real and $\abs{s_m}\leq 1$, the self-adjoint operator $S_m$ satisfies $ \|S_m\|_{\bL(L^2(\bO))} \leq 1$.  Additionally, since $S_m$ and $A$ commute, as $m\to \infty$, for fixed $\psi \in H_0^1(\bO)$, we find
			$$\|S_m\|_{\bL(H_0^1(\bO))}\leq 1\ \text{ and }\ S_m\psi\to\psi \ \text{ in }\ H_0^1(\bO),$$
			via functional calculus’s convergence property. In particular, since $0\leq \rho(2^{-m}t)\leq 1$, for every $t>0$, we have
			\begin{align*}
				\|S_m\psi-\psi\|_{H_0^1(\bO)}^2
				&=\bigg\|\sum_{ \lambda_n\in[2^m,2^{m+1})}[1-\rho(2^{-m}\lambda_n)](\psi,w_n)w_n+\sum_{ \lambda_n\geq 2^{m+1}}(\psi,w_n)w_n\bigg\|_{H_0^1(\bO)}^2\\
				&=\bigg\|\sum_{\lambda_n\in[2^m,2^{m+1})}\lambda_n^{1/2}[1-\rho(2^{-m}\lambda_n)](\psi,w_n)w_n+\sum_{\lambda_n \geq 2^{m+1}}\lambda_n^{1/2}(\psi,w_n)w_n\bigg\|_{L^2(\bO)}^2\\
				& =\sum_{\lambda_n\in[2^m,2^{m+1})} \lambda_n\underbrace{[1-\rho(2^{-m}\lambda_n)]^2}_{\leq 1}|(\psi,w_n)|^2+\sum_{\lambda_n \geq2^{m+1}}\lambda_n|(\psi,w_n)|^2\\
				&\leq \sum_{\lambda_n\geq 2^{m}}\lambda_n|(\psi,w_n)|^2\to 0\ \text{ as }\ m\to\infty,
			\end{align*}
			where we have used the fact that $\psi\in H_0^1(\bO)$ and  $P_m \psi \to \psi$ in $H_0^1(\bO)$.
			
			By applying the spectral multiplier theorem (\cite[Theorem 7.23]{EMO-05}) and the Marcinkiewicz interpolation Theorem \cite[Theorem 9.8]{DG+NST-01}, the uniform boundedness in $L^p(\bO)$ can be obtained. It is sufficient to demonstrate that $s_m$ fulfills the Mihlin condition \cite[see Equation (7.69)]{EMO-05}, i.e.,
			\begin{align}
				\sup_{\gamma>0}|\gamma^ks_m^{(k)}(\gamma)|<\infty,\ k\in\N\cup\{0\}.
			\end{align}
			Indeed, for all $k\in\mathbb{N}$, we have
			\begin{align}
				\sup_{\gamma>0}|\gamma^ks_m^{(k)}(\gamma)|
				&=\sup_{\gamma\in[2^m,2^{m+1})}|\gamma^ks_m^{(k)}(\gamma)|=\sup_{\gamma\in[2^m,2^{m+1})}\bigg|\gamma^k\frac{d^k}{d\gamma^k}\rho(2^{-m}\gamma)\bigg|\\
				&\leq 2^k\sup_{\gamma>0}|\rho^{(k)}(\gamma)|<\infty.
			\end{align}
			Finally, we show that $S_m\psi\to\psi$ as $m\to\infty$, for all $\psi\in L^p(\bO)$. By Sobolev's embedding, we know that $D(S^{k/2})\embed L^{\infty}(\bO)\embed L^p(\bO)$ for $k>\frac{d}{2}$. Using this fact and H\"older's inequality, for all $\psi\in D(S^k)$, we have
			\begin{align}
				&	\|S_m\psi-\psi\|_{L^p(\bO)}^2\\
				&=\bigg\|\sum_{ \lambda_n\in[2^m,2^{m+1})}[1-\rho(2^{-m}\lambda_n)](\psi,w_n)w_n+\sum_{ \lambda_n\geq 2^{m+1}}(\psi,w_n)w_n\bigg\|_{L^p(\bO)}^2\\
				&\leq C\bigg\|\sum_{\lambda_n\in[2^m,2^{m+1})}\lambda_n^{k/2}[1-\rho(2^{-m}\lambda_n)](\psi,w_n)w_n+\sum_{\lambda_n\geq 2^{m+1}}\lambda_n^{k/2}(\psi,w_n)w_n\bigg\|_{L^2(\bO)}^2\\
				&=C\sum_{\lambda_n\in[2^m,2^{m+1})}\lambda_n^{k}[1-\rho(2^{-m}\lambda_n)]^2|(\psi,w_n)|^2+\sum_{\lambda_n\geq 2^{m+1}}\lambda_n^{k}|(\psi,w_n)|^2\\
				&\leq C\sum_{\lambda_n\geq 2^{m}}\lambda_n^{k}|(\psi,w_n)|^2\to 0\ \text{ as }\ m\to\infty.\label{eqn-S_m-cgs-in-L^p}
			\end{align}
			Since the embedding $D(S^{k/2})\embed  L^p(\bO)$ is dense for $k>\frac{d}{2}$ and $2\leq p <\infty$,  for any $\psi\in L^p(\bO)$ and $\eps>0$, there exists a $\psi_{\eps}\in D(S^{k/2})$ such that
			\begin{align*}
				\|\psi_{\eps}-\psi\|_{L^p(\bO)}<\eps.
			\end{align*}
			Therefore, for all $\psi\in L^p(\bO)$, we obtain
			\begin{align}
				\|S_m\psi-\psi\|_{L^p(\bO)}
				& \leq 	\|S_m(\psi-\psi_{\eps})\|_{L^p(\bO)}+\|S_m\psi_{\eps}-\psi_{\eps}\|_{L^p(\bO)}+\|\psi_{\eps}-\psi\|_{L^p(\bO)}\\
				& \leq \left(\|S_m\|_{\bL(L^p(\bO))}+1\right) \|\psi_{\eps}-\psi\|_{L^p(\bO)}+\|S_m\psi_{\eps}-\psi_{\eps}\|_{L^p(\bO)}\\
				& \leq \left(\|S_m\|_{\bL(L^p(\bO))}+1\right)\eps+\|S_m\psi_{\eps}-\psi_{\eps}\|_{L^p(\bO)},
			\end{align}
			and on taking limit supremum as $m\to\infty$, we deduce
			\begin{align}
				\limsup_{m\to\infty}	\|S_m\psi-\psi\|_{L^p(\bO)}\leq C\eps.
			\end{align}
			Since $\eps>0$ is arbitrary, one can  complete the proof.
		\end{proof}
		
		Next, we recall a result from \cite{LH-18}, that provides the relation between $P_m$ and $S_m$.
		
		\begin{proposition}[{\cite[Proposition 3.2]{LH-18}}]\label{Prop-P_m}
			The operators $P_m$ and $S_m$ satisfy the following properties: 
			\begin{enumerate}
				\item[(i)] $P_m$ is projection, i.e., we have $P_m^2 = P_m$, for all $n\in\N$.
				\item[(ii)] The operators $P_m, S_m$ are self-adjoint with $\norm{P_m}_{\Ls(L^2(\bO))} = \norm{S_m}_{\Ls(L^2(\bO))} = 1$, for every $m\in\N$.
				\item[(iii)] $P_m$ and $S_n$ commute, for every $n,m\in\N$.
				\item[(iv)] The ranges of $P_m$ and $S_m$ are finite dimensional.
				\item[(v)] Also, $R(S_{m-1}) \subset R(P_m) \subset R(S_m)$, $S_mP_m = P_m$ and $P_mS_{m-1} = S_{m-1},$ for every $m\in\N$, where $R(\cdot)$ denotes the range space.
				\item[(vi)] $\lim_{m\to\infty} P_m v = \lim_{m\to\infty} S_m v =v,$ for every $v\in L^2(\bO)$.
			\end{enumerate}
		\end{proposition}

		\subsection{Existence} We shall demonstrate the existence of a strong solution for the problem \eqref{eqn-main-problem} in the next nine steps. 
		
		\begin{proof}
			Let us choose and fix $T\in(0,\infty)$. To prove the existence, we follow a modified Faedo-Galerkin approximation method by considering finite-dimensional subspaces of $H_0^1(\bO)$ and $L^2(\bO)$.
			
			\vskip 1mm
			\noindent
			\textit{Step 1.} Recall from subsection \ref{subsec-proj-op} that, there is a nondecreasing sequence $\{\lambda_n\}_{n\in\mathbb{N}}$ of positive numbers tending to infinity and an orthonormal basis $\{w_n\}_{n\in\mathbb{N}}$ $L^2(\bO)$
			Also, consider the following finite dimensional subspaces of $L^2(\bO)$ as 
			\[V_{m}=\mathrm{span}\{w_{k}: k\in \N, \; \lambda_k < 2^{m+1}\}.\]
			Then  $V_{m}\subset V_{m+1}\subset H_0^1(\bO)$, for every $m\in\N$.
			For $u_m\in V_{m}$, we know from \eqref{def-projection} that
			\begin{align}
				u_m= P_m u_m = \sum_{\lambda_n<2^{m+1}}(u_m,w_n)w_n. \label{eqn-def-u_m}
			\end{align}
			The original problem \eqref{eqn-main-problem} is approximated by a system of ODEs in $V_{m}$, by employing $P_m$ and $S_m$, $m\in\N$, as follows:
			\begin{align}\label{eqn-FDWF-CP}
				\left\{
				\begin{aligned}
					u'_{m}(t) & = \Gn_m(u_{m}(t)),\, t\in[0,T],\\
					u_m(0) & = \frac{S_{m-1}u_0}{\norm{S_{m-1}u_0}_{L^2(\bO)}},
				\end{aligned}
				\right.
			\end{align}
				where  for  $v\in V_m$,	$\Gn_m(v) :=P_m\Gn(v)$, with
				\begin{align*}
					\Gn_m(v) & := - A v - P_m(|v|^{p-2}v)  +\big( \norm{\nabla v}_{L^2(\bO)}^2   + \norm{v}_{L^p(\bO)}^p \big) v,
				\end{align*}
				and 
				\begin{align}
					S_{m-1} u_0 = \sum_{\lambda_n<2^{m-1}}(u_0,w_n)w_n + \sum_{\lambda_n\in[2^{m-1},2^{m})}\rho(2^{-{(m-1)}}\lambda_n)(u_0,w_n)w_n. \label{eqn-def-u0m}
				\end{align}
				\begin{remark}
					Note that $P_m u_m(0) = \frac{P_mS_{m-1} u_0}{\norm{S_{m-1}u_0}_{L^2(\bO)}} = \frac{S_{m-1}u_0}{\norm{S_{m-1}u_0}_{L^2(\bO)}} = u_m(0)$, where we used Proposition \ref{Prop-P_m} (v).
				\end{remark}
				
				Recall from the Remark \ref{Rmk-Gn-loc-Lip} that the nonlinear operator $\Gn: L^{2p-2}(\bO)\cap D(A) \to  L^{2}(\bO)$ is Lipschitz on balls. On the other hand, observe that the operator $\Gn_m$ is a mapping from $V_m$ to $V_m$. In finite-dimensional space $V_m$, all the norms are equivalent. This deduce that operator $\Gn_m$ is Lipschitz on balls. Thus, by Picard-Lindel\"of Theorem, the nonlinear differential system \eqref{eqn-FDWF-CP} is solvable  on balls, for some $0<t_m\leq T$. By the a priori estimates in the next step, one can conclude that the solution is global and $t_m = T$.
				
				\vskip 1mm
				\noindent
				\textit{Step 2.} By using \eqref{eqn-FDWF-CP} and the fact that $u_m\in C^1((0,T))$, we can write
				\begin{align*}
					\frac{d}{dt}\big(\norm{u_m(t)}_{L^2(\bO)}^2 -1 \big) 
					& = 2\left\langle u_m(t), \frac{d u_m(t)}{dt} \right\rangle
					= 2 \left\langle u_m(t), \Gn_m(u_m(t))\right\rangle\\
					& = -2 \norm{\nabla u_m(t)}_{L^2(\bO)}^2 -2 \norm{ u_m(t)}_{L^p(\bO)}^p \\
					& \quad + 2\big(\norm{\nabla u_m(t)}_{L^2(\bO)}^2 + \norm{ u_m(t)}_{L^p(\bO)}^p \big) \norm{ u_m(t)}_{L^2(\bO)}^2\\
					& = 2\big(\norm{\nabla u_m(t)}_{L^2(\bO)}^2 + \norm{ u_m(t)}_{L^p(\bO)}^p \big)( \norm{ u_m(t)}_{L^2(\bO)}^2 -1),
				\end{align*}
				for $t\in(0,T)$. Let us denote $\theta(t) = \big(\norm{u_m(t)}_{L^2(\bO)}^2 -1 \big)$. Therefore, the above equation becomes
				\begin{align*}
					\frac{d\theta(t)}{dt} = 2\big(\norm{\nabla u_m(t)}_{L^2(\bO)}^2 + \norm{ u_m(t)}_{L^p(\bO)}^p\big)\theta(t).
				\end{align*}
				On solving the above differential equation for $\theta$, we deduce
				\begin{align*}
					\theta(t) = \theta(0) \exp\left[2\int_0^t \big(\norm{\nabla u_m(s)}_{L^2(\bO)}^2 + \norm{ u_m(s)}_{L^p(\bO)}^p\big) ds\right].
				\end{align*}
				Since $\theta(0) = \norm{u_m(0)}_{L^2(\bO)}^2 -1= \norm{\frac{S_{m-1} u_0}{\norm{S_{m-1} u_0}_{L^2(\bO)}}}_{L^2(\bO)}^2 - 1=0,$ we immediately have $\norm{u_m(t)}_{L^2(\bO)}^2 -1 = 0,\ \text{ for all }\ t\in [0,T].$ Thus, it shows that
				\begin{align}\label{eqn-||um||=1}
					u_m(t)\in \bM ,\ \text{ for all }\ t\in [0,T],
				\end{align}
				 if and only if
				\begin{align}
					\int_0^T \big(\norm{\nabla u_m(s)}_{L^2(\bO)}^2 + \norm{ u_m(s)}_{L^p(\bO)}^p \big) ds < \infty. \label{eqn-finite-S}
				\end{align}
				Moreover, the fact $\norm{u_m(t)}_{L^2(\bO)} = 1$, for all $t\in[0,T]$, implies
				\begin{align}\label{eqn-der}
					\frac{d}{dt} \norm{u_m(t)}_{L^2(\bO)}^2 = (u_m(t), u'_m(t)) = 0,\ \text{for a.e.}\ t\in[0,T].
				\end{align}
				
				\begin{remark}
						Note that the right-hand side of \eqref{eqn-FDWF-CP} is already on the tangent plane $T_{u_m}\bM$, for $m\in\N$. Therefore, it is intrinsic that $u_m\in \bM$ if and only if the inequality \eqref{eqn-finite-S} holds.
				\end{remark}
				
				\vskip 1mm
				\noindent
				\textit{Step 3.} For fixed $m\in \N$. By using the fact that $u_m\in C^1((0,T))$ and substituting $u^\prime_m$ from \eqref{eqn-FDWF-CP}, we deduce
				\begin{align*}
					\frac{1}{2}\frac{d}{dt}\norm{u_m(t)}_{H_0^1(\bO)}^2 
					& = \left(u'_{m}(t), Au_{m}(t) \right)\\
					& = - \bigg\|\frac{du_m(t)}{dt}\bigg\|_{L^2(\bO)}^2 - \left( u'_m(t) , P_m(|u_m(t)|^{p-2}u_m(t))\right) \\
					& \quad+ \big(\norm{\nabla u_m(t)}_{L^2(\bO)}^2  + \norm{u_m(t)}_{L^p(\bO)}^p\big) \left( u'_m(t),u_m(t)\right).
				\end{align*}
				Thus, by using the equation \eqref{eqn-der} and self-adjointness of $P_m$,  we obtain
				\begin{align}\label{}
					\frac{1}{2}\frac{d}{dt}\norm{u_m(t)}_{H_0^1(\bO)}^2 
					& = - \norm{\frac{du_m(t)}{dt}}_{L^2(\bO)}^2 - \left( u'_m(t) , |u_m(t)|^{p-2}u_m(t)\right)\\
					& = - \norm{\frac{du_m(t)}{dt}}_{L^2(\bO)}^2 -\frac{1}{p}\frac{d}{dt}\norm{u_m(t)}^{p}_{L^p(\bO)} .\label{eqn-first-estimates}
				\end{align}
				Now, by integrating the above equality with respect to time from $0$ to $t$, we get
				\begin{align*}
					\frac{1}{2}\norm{u_m(t)}_{H_0^1(\bO)}^2	+ \int_0^t\norm{\frac{du_m(s)}{ds}}_{L^2(\bO)}^2 ds + \frac{1}{p}\norm{u_m(t)}^{p}_{L^p(\bO)} 
					& \leq \frac{1}{2}\norm{u_m(0)}_{H_0^1(\bO)}^2
					+ \frac{1}{p}\norm{u_m(0)}^{p}_{L^p(\bO)}.
				\end{align*}
				Observe from Proposition \ref{Prop-P_m} (vi) that $S_m u_0 \to u_0$ in $L^2(\bO)$. Since every convergent sequence is bounded and $\norm{u_0}_{L^2(\bO)}=1$, it implies $$\frac{1}{\norm{S_{m-1}u_0}_{L^2(\bO)}} \leq C(\norm{u_0}_{L^2(\bO)})=C.$$ Moreover, Proposition \ref{Prop-Sm} gives that $\|S_m\|_{\Ls(L^p(\bO))}<\infty$.
				Therefore, by clubbing these facts with the above estimates, we end up with
				\begin{align}
					\frac{1}{2}\norm{u_m(t)}_{H_0^1(\bO)}^2 + \frac{1}{2p}\norm{u_m(t)}^{p}_{L^p(\bO)}
					& + \int_0^t\norm{\frac{du_m(s)}{ds}}_{L^2(\bO)}^2 ds\\
					& \leq C (\norm{u_0}_{H_0^1(\bO)}^2
					+\norm{u_0}^{p}_{L^p(\bO)})< \infty,\label{eqn-estm}
				\end{align}
				for all $t\in[0,T]$. In particular, we have
				\begin{align}\label{eqn-estimate-Lp+H01}
					\sup_{t\in[0,T]}\big[\norm{u_m(t)}_{H_0^1(\bO)}^2 + \norm{u_m(t)}^{p}_{L^p(\bO)}\big]  
					\leq C \big(\norm{u_0}_{H_0^1(\bO)}^2 + \norm{u_0}^{p}_{L^p(\bO)}\big), 
				\end{align}
				\textit{i.e., the sequence $\{u_m\}_{m\in\N}$ is bounded in $L^{\infty}(0,T;L^p(\bO) \cap H_0^1(\bO))$.}
				
				From \eqref{eqn-estm}, we also infer
				\begin{align}\label{eqn-der-est}
					\| u'_{m}\|_{L^{2}(0,T;L^2(\bO))}
					\leq C (\norm{u_0}_{H_0^1(\bO)}^2 +\norm{u_0}^{p}_{L^p(\bO)}),
				\end{align}
				\textit{it implies, the sequence $\{u'_m\}_{m\in\N}$ is bounded in $L^{2}(0,T;L^2(\bO))$.}
				
				\vskip 1mm
				\noindent
				\textit{Step 4.} Next, let us prove in the following proposition that the sequences $\{Au_m\}_{m\in\N}$, $\{\abs{u_m}^{\frac{p-2}{2}}\nabla u_m\}_{m\in\N}$ and $\{P_m(\abs{u_m}^{p-2}u_m)\}_{m\in\N}$ are bounded in $L^2(0,T; L^2(\bO))$.
				
				\begin{proposition}\label{Prop-u_m-L2-D(A)}
					The sequences $\{Au_m\}_{m\in\N}$, $\{\abs{u_m}^{\frac{p-2}{2}}\nabla u_m\}_{m\in\N}$ and $\{P_m(\abs{u_m}^{p-2}u_m)\}_{m\in\N}$ are bounded in $L^2(0,T; L^2(\bO))$, where $u_m$ is defined in \eqref{eqn-def-u_m}, and the following estimate holds:
					\begin{align*}
						&\int_0^T \Big(\norm{\Delta u_m(t)}_{L^2(\bO)}^2 + 2(p-1)\|\abs{u_m(t)}^{\frac{p-2}{2}}\nabla u_m(t)\|_{L^2(\bO)}^2 + \norm{P_m(\abs{u_m(t)}^{p-2}u_m(t)}_{L^2(\bO)}^2 \Big) dt\\
						& \leq  C(T, \norm{u_0}_{L^p(\bO) \cap H_0^1(\bO)}).
					\end{align*}
				\end{proposition}
				
				\begin{proof}
					Fix $m\in\N$. Again utilizing the fact that $u_m \in C^1((0,T))$, all norms are equivalent in finite dimensions, \eqref{eqn-FDWF-CP}
					and, $P_m$ is self-adjoint and commutes with $A$, we obtain
					\begin{align}
						\frac{d}{dt}\bigg(\frac{\norm{u_m(t)}_{H_0^1(\bO)}^2}{2} + \frac{\norm{u_m(t)}_{L^p(\bO)}^p}{p}\bigg) 
						& = \left(u'_{m}(t), Au_{m}(t) + \abs{u_m(t)}^{p-2}u_m(t)\right)\\
						& = (\Gn_m(u_{m}(t)), Au_{m}(t) + \abs{u_m(t)}^{p-2}u_m(t)).
						\label{eqn-Exis-Apriori-1}
					\end{align}
					Performing integration by parts, we deduce
					\begin{align}
						\left( \abs{u_{m}}^{p-2}u_m , -\Delta u_{m} \right) 
						& = \int_\bO \nabla((\abs{u_m(x)}^2)^{\frac{p-2}{2}}u_m(x)) \nabla u_m(x)  dx\\
						&= \int_\bO \bigg[ \frac{p-2}{2}(2 u_m(x) \nabla u_m(x))((\abs{u_m(x)}^2)^{\frac{p-4}{2}}u_m(x))\\
						&\qquad \quad + \abs{u_m(x)}^{p-2}\nabla u_m(x) \bigg] \nabla u_m(x)  dx\\
						&= \int_\bO \left[ (p-2) \abs{u_m(x)}^{p-2} \abs{\nabla u_m(x)}^2 + \abs{u_m(x)}^{p-2}\abs{\nabla u_m(x)}^2 \right] dx\\
						&= (p-1)\norm{\abs{u_m}^{\frac{p-2}{2}} \nabla u_m}_{L^2(\bO)}^2. \label{eqn-Lap-Nn}
					\end{align}
					Thus, using integration by parts and the above equality in \eqref{eqn-Exis-Apriori-1}, we assert
					\begin{align*}
						& \left(\Gn_m(u_{m}), -\Delta u_{m} + \abs{u_m}^{p-2}u_m\right)\\
						& = (\Delta u_m - P_m(|u_m|^{p-2}u_m) + ( \norm{\nabla u_m}_{L^2(\bO)}^2 + \norm{u_m}_{L^p(\bO)}^p) u_m, -\Delta u_m + \abs{u_m}^{p-2}u_m)\\
						& = - \| A u_{m} + P_m(|u_m|^{p-2}u_m)\|_{L^2(\bO)}^{2}  + \big(\norm{\nabla u_m}_{L^2(\bO)}^2 + \norm{u_m}_{L^p(\bO)}^p \big)^2.
					\end{align*}
					By substituting the above equality in \eqref{eqn-Exis-Apriori-1}, we infer
					\begin{align}
						\frac{d}{dt}\bigg(\frac{\norm{u_m(t)}_{H_0^1(\bO)}^2}{2} + \frac{\norm{u_m(t)}_{L^p(\bO)}^p}{p}\bigg) 
						& + \| A u_{m}(t) + P_m(|u_m(t)|^{p-2}u_m(t))\|_{L^2(\bO)}^{2}\\ 
						& = \big(\norm{\nabla u_m(t)}_{L^2(\bO)}^2 + \norm{u_m(t)}_{L^p(\bO)}^p \big)^2.
					\end{align}
					Upon integration with respect to time from $0$ to $t$ and employing \eqref{eqn-estimate-Lp+H01}, we have 
					\begin{align}
						&	\int_0^t \| A u_{m}(s) + P_m(|u_m(s)|^{p-2}u_m(s))\|_{L^2(\bO)}^{2} ds \\
						& \leq C \bigg(\frac{\norm{u_0}_{H_0^1(\bO)}^2}{2} + \frac{\norm{u_0}_{L^p(\bO)}^p}{p}\bigg) + CT\big(\norm{u_0}_{H_0^1(\bO)}^2+ \norm{u_0}^{p}_{L^p(\bO)} \big)^2 \\
						& \leq C\big(T, \norm{u_0}_{L^p(\bO) \cap H_0^1(\bO)}\big).
					\end{align}
					In particular, we obtain
					\begin{align}
						&\int_0^T\| A u_{m}(t) + P_m(|u_m(s)|^{p-2}u_m(t))\|_{L^2(\bO)}^{2} dt \\
						\nonumber& = \int_0^T \Big(\norm{\Delta u_m(t)}_{L^2(\bO)}^2 + 2(p-1)\|\abs{u_m(t)}^{\frac{p-2}{2}}\nabla u_m(t)\|_{L^2(\bO)}^2 + \norm{P_m(\abs{u_m(t)}^{p-2}u_m(t)}_{L^2(\bO)}^2 \Big) dt\\
						& \leq  C(T, \norm{u_0}_{L^p(\bO) \cap H_0^1(\bO)}).\label{eqn-D(A)-est}
					\end{align}
					\textit{This implies, the sequences $\{Au_m\}_{m\in\N}$, $\{\abs{u_m}^{\frac{p-2}{2}}\nabla u_m\}_{m\in\N}$ and $\{P_m(\abs{u_m}^{p-2}u_m)\}_{m\in\N}$ remain  in a bounded set of $L^{2}(0,T;L^2(\bO))$.}
				\end{proof}
				
				\vspace{1mm}
				\noindent
				\textit{Step 5.} Combining all bounded sequences produced in the above steps and
				applying  the Banach-Alaoglu Theorem \cite[see Theorem 3.1]{JBC-90} yields a subsequence  $\{u_{m_j}\}_{m_j\in\N}$ of $\{u_m\}_{m\in\N}$ such that as $m_j \to \infty$
				\begin{align}\label{eqn-Convergences-weak}
					\left\{
					\begin{aligned}
						u_{m_{j}} & \overset{\ast}{\rightharpoonup} u\; \text{ in }\; L^{\infty}(0,T; L^p(\bO) \cap H_0^1(\bO)),\\
						u_{m_{j}} & \rightharpoonup u \;\text{ in }\; L^{2}(0,T;D(A)), \\
						P_{m_j}(|u_{m_j}|^{p-2}u_{m_j}) & \rightharpoonup \xi\; \text{ in }\; L^{2}(0,T; L^{2}(\bO)), \\
						|u_{m_j}|^{\frac{p-2}{2}}\nabla u_{m_j} & \rightharpoonup \vartheta\; \text{ in }\; L^{2}(0,T; L^{2}(\bO)) \; \; \mbox{and}\\
						u'_{m_{j}} & \rightharpoonup u'\; \text{ in }\; L^{2}(0,T; L^{2}(\bO)).
					\end{aligned}
					\right.
				\end{align}
				Moreover, an application of  the Aubin-Lions Lemma (see Theorem \ref{Aubin-Lions Lemma}) yields
				\begin{align}\label{eqn-strong}
					u_{m_{j}} &\to u \; \text{ in }\; L^2(0,T;H_0^1(\bO))\cap C([0,T];L^2(\bO)),
				\end{align}
				and by the Riesz-Fischer Theorem (\cite[p. 148]{HLR-10}), we infer
				\begin{equation}\label{eqn-strong-convergence}
					\left\{
					\begin{aligned}
						u_{m_j}(t,x)&\to u(t,x),\ \text{ for all }\ t\in[0,T] \ \text{ and  a.e. }  x\in\bO,\\
						\nabla u_{m_j}(t,x)&\to \nabla u(t,x),\ \text{ for a.e. }\ (t,x)\in[0,T]\times\bO,
					\end{aligned}
					\right.
				\end{equation}
				considering a further subsequence, still labeled by the same notation. The convergence \eqref{eqn-strong} implies that $u(0)=u_0$ in $L^2(\bO)$, and since $u_0\in\mathcal{M}$, we have $\|u(0)\|_{L^2(\bO)}=\|u_0\|_{L^2(\bO)}=1$.

				\vspace{1mm}
				\noindent
				\textit{Step 6.} For convenience of notation, we take $m_j=m$. In this step, we aim to show that $$\Gn_m(u_m)\rightharpoonup \Gn(u) \; \mbox{ in } \; L^{2}(0,T; L^{2}(\bO)),$$ where $\Gn_m(u_m) = P_m \Gn(u_m)$ and $\Gn(u)$ is defined in \eqref{eqn-op-G}. Since $C([0,T]; L^p(\bO) \cap H_0^1(\bO))$ is dense in $L^2(0,T; L^2(\bO))$, it is enough to show that the above convergence holds true in $C([0,T]; L^p(\bO) \cap H_0^1(\bO))$.
				
				Observe that we can rewrite the problem \eqref{eqn-FDWF-CP} as
				\begin{align}\label{eqn-fd}
					\left\{
					\begin{aligned}
						&(u'_m(t), \varphi) = (\Gn_m(u_m(t)), \varphi),\\
						&(u_m(0), \varphi) = \bigg( \frac{S_{m-1} u_0}{\norm{S_{m-1} u_0}_{L^2(\bO)}}, \varphi\bigg),
					\end{aligned}
					\right.
				\end{align}
				for all $\varphi \in V_m$. Thus, for all $t\in[0,T]$,  $u_m(t)$ satisfies 
				\begin{align}\label{eqn-fee}
					\norm{u_m(t)}_{L^2(\bO)}^2 = \norm{u_m(0)}_{L^2(\bO)}^2 + 2 \int_0^t (\Gn_m(u_m(s)), u_m(s))ds.
				\end{align}
				Furthermore, from the equation \eqref{eqn-fd} and the estimate \eqref{eqn-der-est}, we deduce
				\begin{align*}
					\int_0^T\|\Gn_{m}(u_{m}(t))\|_{L^2(\bO)}^2dt
					& = \int_0^T\|u_{m}'(t)\|_{L^2(\bO)}^2dt\lesssim \norm{u_0}_{H_0^1(\bO)}^2
					+ \norm{u_0}^{p}_{L^p(\bO)}.
				\end{align*}
				By an application of the Banach-Alaoglu Theorem, there exists a further subsequence (denoted as before), we obtain
				\begin{align*}
					\Gn_{m}(u_{m})& \rightharpoonup \zeta \; \mbox{ in } \; L^{2}(0,T; L^{2}(\bO)).
				\end{align*}
				Using local monotonicty and hemicontinuity properties of $\Gn$, our next aim is to show that $\zeta=\Gn(u)$.
				By taking the limit on the equation \eqref{eqn-fd}, we find
				\begin{align}\label{eqn-id}
					\left\{
					\begin{aligned}
						&\left(\frac{\partial u(t)}{\partial t}, \varphi\right) = (\zeta(t), \varphi),\\
						&(u(0), \varphi) = (u_0, \varphi),
					\end{aligned}
					\right.
				\end{align}
				for all $\varphi \in L^2(\bO)$.
				By an application of the Absolute continuity Lemma \ref{Thm-Abs-cont} and integration from $0$ to $t$ lead to
				\begin{align}\label{eqn-ee}
					\norm{u(t)}_{L^2(\bO)}^2 = \norm{u_0}_{L^2(\bO)}^2 + 2 \int_0^t (\zeta(s), u(s))ds,\ \mbox{ for every $t\in [0,T]$}.
				\end{align}
				On the other hand, from \eqref{eqn-fee}, we also have
				\begin{align*}
					\int_0^t (\Gn_m(u_m(s)), u_m(s))ds = \frac{1}{2}\left[\norm{u_m(t)}_{L^2(\bO)}^2 - \norm{u_m(0)}_{L^2(\bO)}^2\right],\ \mbox{ for all $t\in[0,T]$.}
				\end{align*}
				Taking limit infimum on both sides yields
				\begin{align*}
					\liminf_{m\to \infty} \int_0^t (\Gn_m(u_m(s)), u_m(s))ds = \frac{1}{2}\Big[\liminf_{m\to \infty}\norm{u_m(t)}_{L^2(\bO)}^2 - \limsup_{m\to \infty}\norm{u_m(0)}_{L^2(\bO)}^2\Big].
				\end{align*}
				Using the strong convergence of $u_m\to u$ in $L^2(\bO)$, for all $t\in[0,T]$ (see \eqref{eqn-strong-convergence}) and \eqref{eqn-ee}, we deduce
				\begin{align*}
					\liminf_{m\to \infty} \int_0^t (\Gn_m(u_m(s)), u_m(s))ds \leq \frac{1}{2}\left[\norm{u(t)}_{L^2(\bO)}^2 - \norm{u(0)}_{L^2(\bO)}^2\right] = \int_0^t (\zeta(s), u(s))ds.
				\end{align*}
				Next, 	for some $v\in L^\infty(0,T; V_n)$, where $n<m$,  let us utilize the local monotonicity of the operator $\Gn$ (see Lemma \ref{Lem-Gn-loc-mono}) to find
				\begin{align}
					& \liminf_{m\to \infty} \Bigg(\int_0^t (\Gn_m(u_m(s)) - \Gn(v(s)), u_m(s) - v(s))ds \\
					& - \int_0^t\bigg[\norm{\nabla u_m(s)}_{L^2(\bO)}^2 + \frac{1}{2}\big(\norm{\nabla u_m(s)}_{L^2(\bO)} + \norm{\nabla v(s)}_{L^2(\bO)}\big)^2 \norm{ v(s)}_{L^2(\bO)}^2 +\norm{u_m(s)}_{L^p(\bO)}^p\\
					& \quad  + C(p, \abs{\bO})\big(\norm{u_m(s)}_{L^p(\bO)}^{p-1} + \norm{v(s)}_{L^p(\bO)}^{p-1}\big)\norm{v(s)}_{L^2(\bO)}^2
					\bigg] \norm{u_m(s)-v(s)}_{L^2(\bO)}^2ds \Bigg)\\	
					& = \liminf_{m\to \infty} \Bigg(\int_0^t (\Gn(u_m(s)) - \Gn(v(s)), u_m(s) - v(s))ds \\
					& - \int_0^t\bigg[\norm{\nabla u_m(s)}_{L^2(\bO)}^2 + \frac{1}{2}\big(\norm{\nabla u_m(s)}_{L^2(\bO)} + \norm{\nabla v(s)}_{L^2(\bO)}\big)^2 \norm{ v(s)}_{L^2(\bO)}^2 +\norm{u_m(s)}_{L^p(\bO)}^p\\
					& \quad  + C(p, \abs{\bO})\big(\norm{u_m(s)}_{L^p(\bO)}^{p-1} + \norm{v(s)}_{L^p(\bO)}^{p-1}\big)\norm{v(s)}_{L^2(\bO)}^2
					\bigg] \norm{u_m(s)-v(s)}_{L^2(\bO)}^2ds \Bigg) \\
					& \leq 0.\label{eqn-loc-m}
				\end{align}
				Now let us consider
				\begin{align*}
					&\abs{\int_0^t \norm{\nabla u_m(s)}_{L^2(\bO)}^2 \norm{u_m(s)- v(s)}_{L^2(\bO)}^2 ds - \int_0^t \norm{\nabla u(s)}_{L^2(\bO)}^2 \norm{u(s)- v(s)}_{L^2(\bO)}^2 ds}\\
					& \leq \abs{\int_0^t \norm{\nabla u_m(s)}_{L^2(\bO)}^2 \big(\norm{u_m(s)- v(s)}_{L^2(\bO)}^2 - \norm{u(s)- v(s)}_{L^2(\bO)}^2\big) ds}\\
					& \quad + \abs{\int_0^t \big(\norm{\nabla u_m(s)}_{L^2(\bO)}^2 - \norm{\nabla u(s)}_{L^2(\bO)}^2\big) \norm{u(s)- v(s)}_{L^2(\bO)}^2 ds}\\
					& \leq \int_0^t \norm{\nabla u_m(s)}_{L^2(\bO)}^2 \big(\norm{u_m(s)- v(s)}_{L^2(\bO)} + \norm{u(s) - v(s)}_{L^2(\bO)}\big) \norm{u_m(s)- u(s)}_{L^2(\bO)} ds\\
					& \quad + \int_0^t \norm{\nabla u_m(s) - \nabla u(s)}_{L^2(\bO)} \norm{u(s)- v(s)}_{L^2(\bO)}^2 ds\\
					& \leq \sup_{t\in[0,T]} \norm{\nabla u_m(t)}_{L^2(\bO)}^2 \big(\norm{u_m(s)- v(s)}_{L^2(\bO)} + \norm{u(s) - v(s)}_{L^2(\bO)}\big) \int_0^t  \norm{u_m(s)- u(s)}_{L^2(\bO)} ds\\
					& \quad + \sup_{t\in[0,T]}\norm{u(t)- v(t)}_{L^2(\bO)}^2 \big(\norm{\nabla u_m(t)}_{L^2(\bO)} + \norm{\nabla u(t)}_{L^2(\bO)}\big)  \left(\int_0^t \norm{\nabla u_m(s) - \nabla u(s)}_{L^2(\bO)} ds\right)\\
					&  \to 0.
				\end{align*}
				Similarly, by using the same approach used in proof of Lemma \ref{Lem-F_3}
				and H\"older's inequality (with exponent $p$ and $p/p-1$), we deduce
				\begin{align*}
					&\abs{\int_0^t \norm{ u_m(s)}_{L^p(\bO)}^p \norm{u_m(s)- v(s)}_{L^2(\bO)}^2 ds - \int_0^t \norm{ u(s)}_{L^p(\bO)}^p \norm{u(s)- v(s)}_{L^2(\bO)}^2 ds}\\
					& \leq \int_0^t \norm{ u_m(s)}_{L^p(\bO)}^p \norm{u_m(s)- u(s)}_{L^2(\bO)}^2 ds\\
					& \quad + \abs{\int_0^t \big(\norm{ u_m(s)}_{L^p(\bO)}^p - \norm{ u(s)}_{L^p(\bO)}^p\big) \norm{u(s)- v(s)}_{L^2(\bO)}^2 ds}\\
					& \leq \sup_{t\in[0,T]}\norm{ u_m(s)}_{L^p(\bO)}^p\int_0^t  \norm{u_m(s)- u(s)}_{L^2(\bO)}^2 ds\\
					& \quad + p\sup_{t\in[0,T]} \left(\norm{ u_m(t)}_{L^p(\bO)} + \norm{ u(t)}_{L^p(\bO)}\right)^{p-1}\norm{u(t)- v(t)}_{L^2(\bO)}^2  \int_0^t \norm{ u_m(s) - u(s)}_{L^p(\bO)}  ds\\
					& \to 0.
				\end{align*}
				Therefore, the above convergences yield
				\begin{align*}
					&\limsup_{m\to \infty} \int_0^t \bigg[\norm{\nabla u_m(s)}_{L^2(\bO)}^2 + \frac{1}{2} \big(\norm{\nabla u_m(s)}_{L^2(\bO)} + \norm{\nabla v(s)}_{L^2(\bO)}\big)^2 \norm{ v(s)}_{L^2(\bO)}^2 +\norm{u_m(s)}_{L^p(\bO)}^p\\
					& \quad  + C(p, \abs{\bO})\big(\norm{u_m(s)}_{L^p(\bO)}^{p-1} + \norm{v(s)}_{L^p(\bO)}^{p-1}\big)\norm{v(s)}_{L^2(\bO)}^2
					\bigg] \norm{u_m(s)-v(s)}_{L^2(\bO)}^2 ds\\
					&= \int_0^t \bigg[\norm{\nabla u(s)}_{L^2(\bO)}^2 + \frac{1}{2}\big(\norm{\nabla u(s)}_{L^2(\bO)} + \norm{\nabla v(s)}_{L^2(\bO)}\big)^2 \norm{ v(s)}_{L^2(\bO)}^2 +\norm{u(s)}_{L^p(\bO)}^p\\
					& \quad  + C(p,\abs{\bO})\big(\norm{u(s)}_{L^p(\bO)}^{p-1} + \norm{v(s)}_{L^p(\bO)}^{p-1}\big)\norm{v(s)}_{L^2(\bO)}^2
					\bigg] \norm{u(s)-v(s)}_{L^2(\bO)}^2 ds.
				\end{align*}
				By using the fact that $\Gn_{m}(u_{m}) \rightharpoonup \zeta \; \mbox{ in } \; L^{2}(0,T; L^{2}(\bO))$ and \eqref{eqn-strong}, together in \eqref{eqn-loc-m}, we obtain
				\begin{align}
					&\int_0^t (\zeta (s)- \Gn(v(s)), u(s) - v(s))ds\\
					& \quad- \int_0^t\bigg[\norm{\nabla u(s)}_{L^2(\bO)}^2 + \frac{1}{2}\big(\norm{\nabla u(s)}_{L^2(\bO)} + \norm{\nabla v(s)}_{L^2(\bO)}\big)^2 \norm{ v(s)}_{L^2(\bO)}^2 +\norm{u(s)}_{L^p(\bO)}^p\\
					&  \quad  + C(p, \abs{\bO})\big(\norm{u(s)}_{L^p(\bO)}^{p-1} + \norm{v(s)}_{L^p(\bO)}^{p-1}\big)\norm{v(s)}_{L^2(\bO)}^2
					\bigg] \norm{u(s)-v(s)}_{L^2(\bO)}^2ds\\
					&  \leq 0.\label{eqn-loc-f}
				\end{align}
				Since the above estimate is valid for any $v\in L^{\infty}(0, T ; V_n)$, for every $n\in\mathbb{N}$. By a standard density argument, this inequality extends to any $v\in L^\infty(0,T; L^p(\bO) \cap H_0^1(\bO))$.

				Let us set $u=v+\lambda w$, $\lambda>0$, where $w\in L^\infty(0,T; L^p(\bO) \cap H_0^1(\bO)) \cap L^2(0,T; D(A))$. Inserting this expression into \eqref{eqn-loc-f} yields
				\begin{align*}
					&\int_0^t (\zeta (s)- \Gn((u-\lambda w)(s)), w(s))ds\\
					& \quad-\lambda \int_0^t\bigg[\norm{\nabla u(s)}_{L^2(\bO)}^2 + \frac{1}{2}(\norm{\nabla u(s)}_{L^2(\bO)} + \norm{\nabla (u-\lambda w)(s)}_{L^2(\bO)})^2 \norm{ (u-\lambda w)(s)}_{L^2(\bO)}^2 \\
					&\quad +\norm{u(s)}_{L^p(\bO)}^p  + C(p)\big(\norm{u(s)}_{L^p(\bO)}^{p-1} + \norm{(u-\lambda w)(s)}_{L^p(\bO)}^{p-1}\big)\norm{v(s)}_{L^2(\bO)}^2
					\bigg] \norm{w(s)}_{L^2(\bO)}^2ds \\
					& \leq 0.
				\end{align*}
				Passing to the limit $\lambda\to 0$  and applying the hemicontinuity condition \eqref{eqn-hemicont} of $\Gn(\cdot)$, we obtain
				\begin{align}
					\int_0^t (\zeta (s)- \Gn(u(s)), w(s))ds\leq 0.
				\end{align}
				Since $w$ is arbitrary, we finally deduce
				\begin{align}
					\zeta(t)=\Gn(u(t)) \ \text{ in }\ L^2(\bO),\ \text{ for a.e. }\ t\in[0,T],
				\end{align}
				as required.
				
				Finally, upon passing to the limits in \eqref{eqn-fd}, from the convergences \eqref{eqn-Convergences-weak} and \eqref{eqn-strong-convergence}, we deduce  
				\begin{align}
					\int_{0}^{T}\left( \frac{\partial u(t)}{\partial t}  - \Gn(u(t)),v(t)\right)  \,d t = 0,\label{eqn-limiting-strong}
				\end{align}
				for all $v\in C([0,T]; L^p(\bO) \cap H_0^1(\bO))$. Using  the fact that $C([0,T]; L^p(\bO) \cap H_0^1(\bO))$ is dense in $L^2(0,T; L^2(\bO))$ and $\displaystyle\frac{\partial u}{\partial t} - \Gn(u) \in L^2(0,T; L^2(\bO))$. The above equality holds true, for every $v \in L^2(0,T; L^2(\bO))$.
				
				\vskip 1mm
				\noindent
				\textit{Step 7.} By utilizing the convergence properties available to us and weakly lower semicontinuous property, the following proposition shows that the sequence $\{P_m(\abs{u_m}^{p-2}u_m)\}_{m\in\N}$ converges weakly to $\abs{u}^{p-2}u$ in $ L^{2}(0,T; L^{2}(\bO))$, so that $u\in L^{2p-2}(0,T;L^{2p-2}(\bO))$.
				
				\begin{proposition}\label{Prop-u_m-L2p-2}
					The sequence $\{P_m(\abs{u_m}^{p-2}u_m)\}_{m\in\N}$ converges weakly to $\abs{u}^{p-2}u$ in\break $ L^{2}(0,T; L^{2}(\bO))$, as $m\to\infty$.
				\end{proposition}
				
				\begin{proof}
					Fix $m\geq 1$. Let $\varphi \in C([0,T];H^s(\bO))$ for $s\geq \frac{d(p-2)}{2p},$ so that $H^s(\bO)\embed L^p(\bO)$.  We consider
					\begin{align*}
						\langle P_m(\abs{u_m}^{p-2}u_m) - \abs{u}^{p-2}u, \varphi \rangle 
						& = \langle P_m(\abs{u_m}^{p-2}u_m) - \abs{u_m}^{p-2}u_m + \abs{u_m}^{p-2}u_m - \abs{u}^{p-2}u, \varphi \rangle\\
						& = \langle (P_m - I)\abs{u_m}^{p-2}u_m, \varphi \rangle + \langle \abs{u_m}^{p-2}u_m - \abs{u}^{p-2}u, \varphi \rangle\\
						& = \langle \abs{u_m}^{p-2}u_m, (P_m-I)\varphi \rangle + \langle \abs{u_m}^{p-2}u_m - \abs{u}^{p-2}u, \varphi \rangle\\
						& =: \mathcal{K}_1 + \mathcal{K}_2.
					\end{align*}
					Now, by applying H\"older's inequality (with exponent $p/(p-1)$ and $p$), using \eqref{eqn-estimate-Lp+H01} and the fact that $P_m$ and $A^{s/2}$ commutes, we have
					\begin{align*}
						\mathcal{K}_1 & \leq \norm{u_m}_{L^p(\bO)}^{p-1} \norm{(P_m - I)\varphi}_{L^p(\bO)} \leq C \norm{u_m}_{L^p(\bO)}^{p-1} \norm{(P_m - I)\varphi}_{H^s(\bO)}\\
						& \leq C( \norm{u_0}_{L^p(\bO) \cap H_0^1(\bO)}) \norm{(P_m - I)A^{s/2} \varphi}_{L^2(\bO)}\\
						& \to 0\ \text{ as }\ m\to\infty.
					\end{align*}
					The above convergence holds true, for all $t\in[0,T]$. 	Observe from the \eqref{eqn-strong-convergence} that
					\begin{align*}
						\Nn(u_m(t,x)) 
						\to 
						\Nn(u((t,x)))\ \text{ for  a.e.}\ (t,x) \in (0,T)\times \bO.
					\end{align*}
					Moreover, by using \eqref{eqn-estimate-Lp+H01}, we deduce
					\begin{align*}
						\norm{\Nn(u_m)}_{L^{\frac{p}{p-1}}(0,T; L^{\frac{p}{p-1}}(\bO))} 
						& = \norm{\abs{u_m}^{p-2}u_m}_{L^{\frac{p}{p-1}}(0,T; L^{\frac{p}{p-1}}(\bO))} = \int_0^T \norm{\abs{u_m(t)}^{p-2}u_m(t)}_{L^{\frac{p}{p-1}}(\bO)}^{\frac{p}{p-1}} dt\\
						&  = \int_0^T \norm{u_m(t)}_{L^p(\bO)}^{p} dt  \leq C(T,\norm{u_0}_{L^p(\bO) \cap H_0^1(\bO)}) < \infty.
					\end{align*}
					Note also that $u\in L^{\infty}(0,T;L^p(\bO))$. Then, by the Lions Lemma \ref{Lem-Lions}, we assert 
					\begin{align}\label{eqn-N(u_m)-conv}
						\Nn(u_m) \rightharpoonup \Nn(u) \ \text{ in }\ L^{\frac{p}{p-1}}(0,T; L^{\frac{p}{p-1}}(\bO)).
					\end{align}
					Thus, using the above convergence \eqref{eqn-N(u_m)-conv}, we also have
					\begin{align*}
						\mathcal{K}_2 = \langle \abs{u_m}^{p-2}u_m - \abs{u}^{p-2}u, \varphi \rangle \to 0\ \text{ as }\ m\to\infty,
					\end{align*}
					for all $\phi\in C([0,T];H^s(\bO))$. 	It implies that 
					\begin{align}\label{eqn-P_m(N(u_m))-convergence}
						\langle P_m(\abs{u_m}^{p-2}u_m) - \abs{u}^{p-2}u, \varphi \rangle \to 0,\ \text{ for every }\ \varphi\in C([0,T]; H^s(\bO)).
					\end{align}
					Since $C([0,T]; H^s(\bO))$ is dense in $L^{p}(0,T; L^p(\bO)),$ for every $s \geq \frac{d(p-2)}{2p}$, it implies for a given $\eps>0$, there exists a sequence $\{\varphi_\eps\}_{\eps>0}$ such that 
					\begin{align}\label{eqn-density-converg}
						\|\varphi_\eps - \varphi\|_{L^p(0,T;L^p(\bO))} <\eps.
					\end{align}
					Therefore, using H\"older's inequality (with exponent in $p/(p-1)$ and $p$), we write
					\begin{align*}
						&	\abs{\langle P_m(\abs{u_m}^{p-2}u_m) - \abs{u}^{p-2}u, \varphi \rangle} \\
						& \leq \abs{\langle P_m(\abs{u_m}^{p-2}u_m) - \abs{u}^{p-2}u, \varphi - \varphi_\eps \rangle} + \abs{\langle P_m(\abs{u_m}^{p-2}u_m) - \abs{u}^{p-2}u, \varphi_\eps \rangle}\\
						& \leq \norm{P_m(\abs{u_m}^{p-2}u_m) - \abs{u}^{p-2}u}_{L^{\frac{p}{p-1}}(0,T; L^{\frac{p}{p-1}}(\bO))} \norm{\varphi - \varphi_\eps}_{L^p(0,T;L^p(\bO))}\\
						&\quad + \abs{\langle P_m(\abs{u_m}^{p-2}u_m) - \abs{u}^{p-2}u, \varphi_\eps \rangle}.
					\end{align*}
					Thanks to convergences \eqref{eqn-Convergences-weak} and \eqref{eqn-P_m(N(u_m))-convergence}, we obtain the following result:
					\begin{align*}
						\abs{\langle P_m(\abs{u_m}^{p-2}u_m) - \abs{u}^{p-2}u, \varphi \rangle} \to 0, \ \text{ for every }\ \varphi \in L^p(0,T; L^p(\bO)),
					\end{align*} 
					so that 
					\begin{align*}
						P_m(\abs{u_m}^{p-2}u_m) \rightharpoonup \abs{u}^{p-2}u\ \text{ in }\ L^{\frac{p}{p-1}}(0,T; L^{\frac{p}{p-1}}(\bO)).
					\end{align*}
					On the other hand, from \eqref{eqn-Convergences-weak}, we also have
					\begin{align*}
						P_m(\abs{u_m}^{p-2}u_m) \rightharpoonup \xi \ \text{ in }\ L^2(0,T;L^2(\bO)).
					\end{align*}
					Consequently, the fact that weak limits are unique asserts that $\xi = \abs{u}^{p-2}u$ and 
					\begin{align*}
						P_m(\abs{u_m}^{p-2}u_m) \rightharpoonup \abs{u}^{p-2}u\ \text{ in }\ L^2(0,T;L^2(\bO)).
					\end{align*}
					Hence the sequence $P_m(\abs{u_m}^{p-2}u_m)$ converges weakly to $\abs{u}^{p-2}u$ in $L^{2}(0,T; L^{2}(\bO))$, $m\to\infty$.
				\end{proof}
				
				\begin{remark}\label{Rmk-L{2p-2}-bdd}
					Proposition \ref{Prop-u_m-L2p-2}, \eqref{eqn-D(A)-est} and the weakly lower semicontinuity property of norms yield
					\begin{align*}
						\norm{\abs{u}^{p-2}u}_{L^2(0,T;L^2(\bO))}^2 \leq \liminf_{m\to \infty} \norm{P_m(\abs{u_m}^{p-2}u_m)}_{L^2(0,T;L^2(\bO))}^2 < \infty.
					\end{align*}
					Therefore, we have 
					\begin{align*}
						\int_0^T \norm{u(t)}_{L^{2p-2}(\bO)}^{2p-2} dt \leq C(T,\norm{u_0}_{L^p(\bO) \cap H_0^1(\bO)})< \infty.
					\end{align*}
				\end{remark}
				
				\vspace{1mm}
				\noindent
				\textit{Step 8.} Similar to the previous step, by an application of the Lions Lemma \ref{Lem-Lions}, we now show that $\{\abs{u_m}^{\frac{p-2}{2}}\nabla u_m)\}_{m\in\N}$ converges weakly to $\abs{u}^{\frac{p-2}{2}}\nabla u$ in $ L^{2}(0,T; L^{2}(\bO))$.
				
				\begin{proposition}\label{Prop-u_m-grad-u_m}
					The sequence $\{\abs{u_m}^{\frac{p-2}{2}}\nabla u_m)\}_{m\in\N}$ converges weakly to $\abs{u}^{\frac{p-2}{2}}\nabla u$ in \break $L^{2}(0,T; L^{2}(\bO))$, as $m\to \infty$.
				\end{proposition}
				
				\begin{proof}
					Fix $m\in\N$. Observe from the \eqref{eqn-strong-convergence} that
					\begin{align*}
						\abs{u_m(t,x)}^{\frac{p-2}{2}}\nabla u_m(t,x) \to \abs{u(t,x)}^{\frac{p-2}{2}}\nabla u(t,x)\ \text{ for  a.e.}\ (t,x)\in (0,T)\times \bO.
					\end{align*} 
					By utilizing the fact $p/(p-1) \leq 2$, for every $2\leq p < \infty$ and \eqref{eqn-D(A)-est}, we immediately have
					\begin{align}
						\big\|\abs{u_m(t)}^{\frac{p-2}{2}}\nabla u_m(t)\big\|_{L^{\frac{p}{p-1}}(0,T;L^{\frac{p}{p-1}}(\bO))}
						& \leq C\big\|\abs{u_m(t)}^{\frac{p-2}{2}}\nabla u_m(t)\big\|_{L^{2}(0,T;L^{2}(\bO))}\\
						& \leq C(T, \norm{u_0}_{L^p(\bO) \cap H_0^1(\bO)}) < \infty.
					\end{align}
					On the other hand
					\begin{align}\label{eqn-u-grad-u_1}
						\int_0^T \big\|\abs{u(t)}^{\frac{p-2}{2}}\nabla u(t)\big\|_{L^{\frac{p}{p-1}}(\bO)}^{\frac{p}{p-1}} dt =  \int_0^T \int_\bO \abs{u(t,x)}^{\frac{p(p-2)}{2p-2}} \abs{\nabla u(t,x)}^{\frac{p}{p-1}} dx dt.
					\end{align}
					By applying H\"older's inequality (with exponents $(2p-2)^2/(p^2-p)$ and $(2p-2)^2/(3p^2 -6p +4)$), we deduce
					\begin{align*}
						\big\|\abs{u}^{\frac{p-2}{2}}\nabla u\big\|_{L^{\frac{p}{p-1}}(\bO)}^{\frac{p}{p-1}} 
						& \leq \left( \int_\bO\abs{u(x)}^{2p-2}dx\right)^{\frac{p^2-2p}{(2p-2)^2}}  \left(\int_\bO \abs{\nabla u(x)}^{\frac{4p^2-4p}{3p^2 - 6p +4}} dx\right)^{\frac{3p^2 - 6p +4}{(2p-2)^2}}\\
						& = \|u\|_{L^{2p-2}(\bO)}^{\frac{p^2-2p}{(2p-2)^2}} \|\nabla u\|_{L^{\frac{4p^2-4p}{3p^2 - 6p +4}}(\bO)}^{\frac{p}{p-1}}.
					\end{align*}
					For every $2\leq p <\infty$, it holds that $(p^2-p)/(2p-2)^2 \leq 2p-2 $ and $4(p^2-p)/(3p^2 -6p +4)\leq 2$. Thus, using the Sobolev embedding, Remark \ref{Rmk-L{2p-2}-bdd} and the above inequality in \eqref{eqn-u-grad-u_1} imply 
					\begin{align*}
						\int_0^T \big\|\abs{u(t)}^{\frac{p-2}{2}}\nabla u(t)\big\|_{L^{\frac{p}{p-1}}(\bO)}^{\frac{p}{p-1}} dt
						& \leq C \|u\|_{L^{\frac{p^2-2p}{2p-2}}(0,T; L^{2p-2}(\bO))}^{\frac{p^2-2p}{2p-2}} \norm{u}_{L^\infty(0,T; H_0^1(\bO))}^{\frac{p}{p-1}}\\
						& \leq C \|u\|_{L^{2p-2}(0,T; L^{2p-2}(\bO))}^{\frac{p^2-2p}{2p-2}} \norm{u}_{L^\infty(0,T; H_0^1(\bO))}^{\frac{p}{p-1}}\\
						& \leq C(T,\norm{u_0}_{L^p(\bO) \cap H_0^1(\bO)})<\infty.
					\end{align*}
					Therefore, an application of the Lions Lemma \ref{Lem-Lions} gives 
					\begin{align*}
						\abs{u_m}^{\frac{p-2}{2}}\nabla u_m \rightharpoonup \abs{u}^{\frac{p-2}{2}}\nabla u\ \text{ in }\ L^{\frac{p}{p-1}}(0,T; L^{\frac{p}{p-1}}(\bO)).
					\end{align*}
					But from \eqref{eqn-Convergences-weak}, we have
					\begin{align*}
						\abs{u_m}^{\frac{p-2}{2}}\nabla u_m \rightharpoonup \vartheta \ \text{ in }\ L^2(0,T; L^2(\bO)).
					\end{align*}
					By utilizing the fact that weak limits are unique, we obtain that $\vartheta = \abs{u}^{\frac{p-2}{2}}\nabla u$. Thus
					\begin{align*}
						\abs{u_m}^{\frac{p-2}{2}}\nabla u_m \rightharpoonup \abs{u}^{\frac{p-2}{2}}\nabla u\ \text{ in }\ L^{2}(0,T; L^{2}(\bO)).
					\end{align*}
					Hence the sequence $\abs{u_m}^{\frac{p-2}{2}}\nabla u_m$ converges weakly to $\abs{u}^{\frac{p-2}{2}}\nabla u$ in $L^{2}(0,T; L^{2}(\bO))$, as $m\to\infty$.
				\end{proof}
				
				\begin{remark}\label{Rmk-u_m-grad-u_m}
					Observe from Proposition \ref{Prop-u_m-grad-u_m}, \eqref{eqn-D(A)-est} and the weakly lower semicontinuity property of norms that
					\begin{align*}
						\big\|\abs{u}^{\frac{p-2}{2}}\nabla u\big\|_{L^2(0,T;L^2(\bO))}^2 \leq \liminf_{m\to \infty} \big\|\abs{u_m}^{\frac{p-2}{2}}\nabla u_m\big\|_{L^2(0,T;L^2(\bO))}^2 \leq C\big(T,\norm{u_0}_{L^p(\bO) \cap H_0^1(\bO)}\big)< \infty.
					\end{align*}
				\end{remark}
				
				\vspace{1mm}
				\noindent
				\textit{Step 9.} In this step, we show that $u\in C([0,T]; L^p(\bO) \cap H_0^1(\bO))$. Note that, by the Sobolev embedding Theorem \cite[Ch.\;5, Sec.\;6]{DM+DZ-98}, we have
				\begin{align*}
					H_0^1(\bO) \embed L^p(\bO) & \ \text{ for }\ d \geq 2 \ \text{ and }\ 2 \leq p \leq\frac{2d}{d-2};\\
					H_0^1(\bO) \embed L^p(\bO) & \ \text{ for }\ d =1,2  \ \text{ and }\ 2 \leq p < \infty.
				\end{align*}
				For the above two cases, since $u\in L^{\infty}(0,T; L^p(\bO) \cap H_0^1(\bO))$ and $u'\in L^{2}(0,T; L^{2}(\bO))$, an application of Theorem \ref{Thm-Abs-cont} yields $u\in C([0,T];H_0^1(\bO))$.
				
				Therefore, for $d\geq 2,$ it is enough to show that $u \in \; C([0,T]; L^p(\bO) \cap H_0^1(\bO))$ for $\frac{2d}{d-2} <p<\infty$. Let us first fix $\frac{2d}{d-2} < p < \infty$ and define $z := \abs{u}^{\frac{p}{2}}$. Then,  we observe that
				\begin{align*}
					\norm{z}_{L^2(\bO)}^2 = \||u|^{\frac{p}{2}}\|_{L^2(\bO)}^2 = \norm{u}_{L^p(\bO)}^p.
				\end{align*}
				By using the fact that $u\in L^{\infty}(0,T; L^p(\bO) )$, we have
				\begin{align*}
					z \in L^{\infty}(0,T; L^2(\bO) ).
				\end{align*}
				Moreover, consider
				\begin{align}
					\norm{\nabla z}_{L^2(\bO)}^2 
					& = \int_{\bO}\abs{\nabla \big((\abs{u(x)}^2)^{\frac{p}{4}}\big)}^2 dx
					=  \frac{p^2}{4^2} \int_{\bO}\abs{(2 u(x) \nabla u(x)) (\abs{u(x)}^2)^{\frac{p-4}{4}}}^2 dx \\
					& = \frac{p^2}{4} \int_{\bO}\abs{u(x)}^{p-4} \abs{u(x)}^2 \abs{\nabla u(x)}^2 dx
					=  \frac{p^2}{4} \int_{\bO}\abs{u(x)}^{p-2} \abs{\nabla u(x)}^2 dx\\
					& = \frac{p^2}{4} \big\|\abs{u}^{\frac{p-2}{2}} \nabla u\big\|_{L^2(\bO)}^2. \label{eqn-grad-|u|^{p/2}}
				\end{align}
				Integrating from $0$ to $T$, we deduce
				\begin{align*}
					\int_0^T \norm{\nabla z(t)}_{L^2(\bO)}^2 dt = \norm{ z}_{L^2(0,T;H_0^1(\bO))}^2 = 
					& \frac{p^2}{4}\int_0^T \big\|\abs{u(t)}^{\frac{p-2}{2}} \nabla u(t)\big\|_{L^2(\bO)}^2 dt.
				\end{align*}
				Moreover, from Remark \ref{Rmk-u_m-grad-u_m}, we also have
				\begin{align*}
					\int_{0}^{T}\|\abs{u(t)}^{\frac{p-2}{2}} \nabla u(t)\|_{L^2(\bO)}^2 dt
					\leq C\big(T, \norm{u_0}_{L^p(\bO) \cap H_0^1(\bO)}\big).
				\end{align*}
				Thus, $z \in L^2(0,T; H_0^1(\bO))$. Moreover, by definition of $z$ and Remark \ref{Rmk-L{2p-2}-bdd}, we write
				\begin{align*}
					\int_0^T\|z(t)\|_{L^{\frac{4p-4}{p}}(\bO)}^{\frac{4p-4}{p}}dt& = \int_0^T\|u(t)\|_{L^{2p-2}(\bO)}^{2p-2}dt \leq C(T, \norm{u_0}_{L^p(\bO) \cap H_0^1(\bO)}).
				\end{align*}
				On the other hand, if we consider the time derivative of  $z$, i.e.,
				\begin{align*}
					\frac{\partial z(t)}{\partial t} = \frac{\partial}{\partial t} \left((\abs{u(t)}^2)^{\frac{p}{4}}\right) 
					= \frac{p}{4}\left(\abs{u(t)}^2\right)^{\frac{p-4}{4}} (2 u(t) \frac{\partial u(t)}{\partial t}) = \frac{p}{2}\abs{u(t)}^{\frac{p-4}{2}} u(t) \frac{\partial u(t)}{\partial t}.
				\end{align*}
				By taking $L^q(\bO)-$norm on both sides, H\"older's inequality yields
				\begin{align*}
					\norm{\frac{\partial z(t)}{\partial t}}_{L^q(\bO)} 
					& = \frac{p}{2} \norm{\abs{u(t)}^{\frac{p-2}{2}} \frac{\partial u(t)}{\partial t}}_{L^q(\bO)}\\
					&\leq \frac{p}{2} \norm{\frac{\partial u(t)}{\partial t}}_{L^2(\bO)}  \|\abs{u(t)}^{\frac{p-2}{2}}\|_{L^{\frac{2q}{2-q}}(\bO)} \\
					& = \frac{p}{2} \norm{\frac{\partial u(t)}{\partial t}}_{L^2(\bO)}  \norm{u(t) }_{L^{\frac{q(p-2)}{2-q}}(\bO)}^{\frac{p-2}{2}}.
				\end{align*}
				To use the fact that $u\in L^{2p-2}(0,T; L^{2p-2}(\bO))$, we choose 
				$$\frac{q(p-2)}{2-q} = 2p -2 \implies \frac{2-q}{q} = \frac{p -2}{2p-2} \implies q = \frac{4p-4}{3p-4}>1.$$ 
				Therefore, we assert 
				\begin{align*}
					\norm{\frac{\partial z(t)}{\partial t}}_{L^{\frac{4p-4}{3p-4}}(\bO)} = \frac{p}{2} \norm{\frac{\partial u(t)}{\partial t}}_{L^2(\bO)}  \norm{u(t) }_{L^{2p-2}(\bO)}^{\frac{p-2}{2}}.
				\end{align*}
				Since $\frac{\partial u}{\partial t} \in L^2(0,T; L^2(\bO))$ and $u\in L^{2p-2}(0,T; L^{2p-2}(\bO))$, it implies
				\begin{align*}
					\int_0^T\norm{\frac{\partial z(t)}{\partial t}}_{L^{\frac{4p-4}{3p-4}}(\bO)}^{\frac{4p-4}{3p-4}}dt
					&\leq \left(\frac{p}{2}\right)^{\frac{4p-4}{3p-4}}\bigg(\int_0^T \norm{\frac{\partial u(t)}{\partial t}}_{L^2(\bO)}^2dt\bigg)^{\frac{2p-2}{3p-4}}\left(\int_0^T\norm{u(t) }_{L^{2p-2}(\bO)}^{2p-2}dt\right)^{\frac{p-2}{3p-4}}\\
					& = \left(\frac{p}{2}\right)^{\frac{4p-4}{3p-4}} \norm{\frac{\partial u}{\partial t}}_{L^2(0,T; L^2(\bO))}^{\frac{4p-4}{3p-4}}\norm{u }_{L^{2p-2}(0,T; L^{2p-2}(\bO))}^{\frac{(p-2)(2p-2)}{3p-4}}  < \infty.
				\end{align*}
			Since $0\in L^2(0,T; H^{-1}(\bO))$	and $$z \in L^{\frac{4p-4}{p}}(0,T; L^{\frac{4p-4}{p}}(\bO))  \cap L^2(0,T; H_0^1(\bO)),$$	the above estimate implies   $$\frac{\partial z}{\partial t} \in L^{\frac{4p-4}{3p-4}}(0,T; L^{\frac{4p-4}{3p-4}}(\bO))+ L^2(0,T; H^{-1}(\bO)),$$
				Now by considering the embedding
				$$H_0^1(\bO) \embed L^2(\bO) \embed H^{-1}(\bO) \embed L^{\frac{4p-4}{3p-4}}(\bO) + H^{-1}(\bO),$$
				and
				$$L^{\frac{4p-4}{p}}(\bO) \embed L^2(\bO) \embed L^{\frac{4p-4}{3p-4}}(\bO) \embed L^{\frac{4p-4}{3p-4}} (\bO)+ H^{-1}(\bO),$$
				as a consequence of Theorem \ref{Thm-Abs-cont}, we deduce that
				\begin{itemize}
					\item[(i)] $z = \abs{u}^{\frac{p}{2}}$ is in the space $C([0,T];L^2(\bO))$, i.e., $\abs{u} \in C([0,T];L^p(\bO))$,
					\item[(ii)] the function $[0,T] \ni t \mapsto  \fourIdx{}{}{2}{L^2(\bO)}{\|z(t)\|} \in \R$ is absolutely continuous  and
					\begin{align*}
						\frac{d}{dt}\fourIdx{}{}{2}{L^2(\bO)}{\|z(t)\|}
						& = 2\,\fourIdx{}{L^{\frac{4p-4}{p}} (\bO)}{}{L^{\frac{4p-4}{3p-4}} (\bO)}{\left\langle z(t), z'(t)\right\rangle}.
					\end{align*}
				\end{itemize}
				This implies that the function $[0,T] \ni t \mapsto \|u(t)\|_{L^p(\bO)}^p \in \R$ is absolutely continuous and it satisfies
				\begin{align*}
					\frac{1}{p}\frac{d}{dt}{\|u(t)\|_{L^p(\bO)}^p}
					& = \left( \abs{u(t)}^{p-2}u(t), \frac{\partial u(t)}{\partial t}\right),
				\end{align*}
				for a.e. $t\in[0,T]$. Since $u\in L^{2p-2}(0,T;L^{2p-2}(\bO))$, one can choose $\varphi=|u|^{p-2}u$ in \eqref{eqn-test} and use the above fact to obtain
				\begin{align*}
					\|u(t)\|_{L^p(\bO)}^p 
					& + p(p-1)\int_0^t\||u(s)|^{\frac{p-2}{2}}\nabla u(s)\|_{L^p(\bO)}^pds+p\int_0^t\|u(s)\|_{L^{2p-2}(\bO)}^{2p-2}ds\\
					& =\|u_0\|_{L^p(\bO)}^p + p\int_0^t\big( \norm{\nabla u(s)}_{L^2(\bO)}^2 + \norm{u(s)}_{L^p(\bO)}^p \big)\|u(s)\|_{L^p(\bO)}^pds,
				\end{align*}
				for all $t\in[0,T]$. Note that from the above argument, for any $t\in[0,T]$, we have $\norm{u(t)}_{L^p(\bO)} \to \norm{u_0}_{L^p(\bO)}$, as $t\to 0$. Let us now show that $u\in C_{w}([0,T]; L^p(\bO)\cap H_0^1(\bO))$. From  \eqref{eqn-strong}, we observe that
				\begin{align*}
					u\in C([0,T];L^2(\bO)) \embed C_{w}([0,T];L^2(\bO)).
				\end{align*}
				On the other hand, we also have $u \in L^\infty(0,T; L^p(\bO)\cap H_0^1(\bO))$.
				By an application of Theorem \ref{Strauss-Lemma}, i.e., choosing $\X = L^p(\bO)\cap H_0^1(\bO)$ and $\Y = L^2(\bO)$, we immediately have 
				\begin{align*}
					u\in C_{w}([0,T]; L^p(\bO)\cap H_0^1(\bO)).
				\end{align*}
				Since $u\in C_{w}([0,T]; L^p(\bO))$, it follows that $u(t) \rightharpoonup u_0$ in $L^p(\bO)$, as $t\to 0$. Since every Hilbert space and $L^p(\bO)$ spaces, for $p>1$, are uniformly convex \cite[see Section 3]{JAC-1926} (or \cite[Section 3.7]{HB-11}), by the Radon-Riesz property (see Proposition \ref{Prop-Radon-Riesz}), it implies, $u(t) \to u_0$ in $L^p(\bO)$ as $t\to 0$, so that $u \in C([0,T];L^p(\bO))$.
				
				Hence, the existence of a strong solution to \eqref{eqn-main-problem} is established.
			\end{proof}
			

			\subsection{Invariance of manifold} 
			Let us now address the second claim of Theorem \ref{Thm-main-wellposed}, i.e., the invariance of strong solutions of the constrained problem \eqref{eqn-main-problem} in the manifold $\bM$ .
			
			\begin{proof}
				Let us take $u_0 \in \bM$. For $t\in[0,T)$, applying the Absolute Continuity Lemma \ref{Thm-Abs-cont} and integrating from $0$ to $t$, we get 
				\begin{align*}
					\frac{1}{2}(\norm{u(t)}_{L^2(\bO)}^2 -1) 
					&= \frac{1}{2}(\norm{u_0}_{L^2(\bO)}^2 -1) + \int_0^t \left(\frac{\partial u(s)}{\partial s}, u(s)\right) ds = \int_0^t \left(\Gn(u(s)), u(s)\right)ds\\
					& = -\int_0^t \norm{\nabla u(s)}_{L^2(\bO)}^2ds - \int_0^t \left(|u(s)|^{p-2}u(s) , u(s)\right) ds\\
					& \quad + \int_0^t \left( \big(\norm{\nabla u(s)}_{L^2(\bO)}^2  + \norm{u(s)}_{L^p(\bO)}^p\big) u(s), u(s)\right)ds\\
					& =  \int_0^t  \big(\norm{\nabla u(s)}_{L^2(\bO)}^2  + \norm{u(s)}_{L^p(\bO)}^p\big) (\norm{u(s)}_{L^2(\bO)}^2 -1)ds.
				\end{align*}
				Let us denote $\Theta(t) = (\norm{u(t)}_{L^2(\bO)}^2 -1)$. Then by applying the variation of constant formula, we obtain
				\begin{align}\label{eqn-equality}
					\Theta(t) = \Theta(0) \exp\left(\int_0^t  \left(\norm{\nabla u(s)}_{L^2(\bO)}^2  + \norm{u(s)}_{L^p(\bO)}^p \right)ds\right) = 0,
				\end{align}
				where it is justified by the fact that $u\in L^\infty(0,T;L^p(\bO)\cap H_0^1(\bO))$.
				Since $u_0\in\bM$, it implies
				$$\Theta(t) = \norm{u(t)}_{L^2(\bO)}^2 -1 =0.$$
				Hence $u(t)\in\bM$, for every $t\geq 0$ and the strong solution $u$ of the problem \eqref{eqn-main-problem} is invariant in the manifold $\bM$.
			\end{proof}
			
			\begin{proof}[Proof of gradient flow equality \eqref{eqn-Eu}]
				Note that
				\begin{align}
					\nabla_{\bM}\bE(u) 
					& = \pi_u(\nabla\bE(u))=\pi_u(-\Delta u+|u|^{p-2}u)\\
					&=  -\Delta u+|u|^{p-2}u - \big( \norm{\nabla u}_{L^2(\bO)}^2 + \norm{u}_{L^p(\bO)}^p \big)u,\ \text{ for } u\in\bM.\label{eqn-grad}
				\end{align}
				Hence, for $u\in\bM$, we deduce
				\begin{align}
					\|\nabla_{\bM}\bE(u)\|_{L^2(\bO)}^2
					& =\|\nabla\bE(u)\|_{L^2(\bO)}^2+ \big( \norm{\nabla u}_{L^2(\bO)}^2 + \norm{u}_{L^p(\bO)}^p \big)^2\\ 
					& \quad-2 \big( \norm{\nabla u}_{L^2(\bO)}^2 + \norm{u}_{L^p(\bO)}^p \big)\left(\nabla\bE(u),u\right)\\
					& = \|\nabla\bE(u)\|_{L^2(\bO)}^2- \big( \norm{\nabla u}_{L^2(\bO)}^2 + \norm{u}_{L^p(\bO)}^p \big)^2\\
					& = \|\nabla\bE(u)\|_{L^2(\bO)}^2-|\bS(u)|^2, \label{eqn-grad_M-grad-rel}
				\end{align}
				where 
				\begin{align}\label{eqn-S(u)}
					\bS(u)= \norm{\nabla u}_{L^2(\bO)}^2 + \norm{u}_{L^p(\bO)}^p = (\nabla\bE(u),u).
				\end{align}
				Therefore, for $u\in\bM$, we have
				\begin{align*}
					\frac{d}{dt}\bE(u(t)) 
					& = \left(\nabla_{\bM}\bE(u(t)),\frac{\partial u(t)}{\partial t}\right)=\left(\nabla_{\bM}\bE(u(t)),-\nabla_{\bM}\bE(u(t))\right)\\
					& = -\|\nabla_{\bM}\bE(u(t))\|_{L^2(\bO)}^2, \ \mbox{ for a.e. $t\in[0,T]$}
				\end{align*}
				and so
				\begin{align}\label{eqn-energy-eqn}
					\bE(u(t))+\int_0^t\|\nabla_{\bM}\bE(u(s))\|_{L^2(\bO)}^2ds = \bE(u_0),\ t>0.
				\end{align}
				It implies that $\bE(u(t))$ is dissipative. 
			\end{proof}

				\begin{remark}\label{Rmk-partial_t(u)}
				Using \eqref{eqn-grad}, the problem \eqref{eqn-main-problem} can be re-written as
				\begin{align*}
					\left\{
					\begin{aligned}
						\frac{\partial u(t)}{\partial t} & = -\nabla_{\bM}\bE(u(t)),\\
						u(0) & = u_0, \\
						u(t)|_{\partial\bO} & = 0,
					\end{aligned}
					\right.
				\end{align*}
				so that from \eqref{eqn-energy-eqn}, for all $t\geq 0$, we also have 
				\begin{align}\label{eqn-dec-ener}
					\bE(u(t))+\int_0^t\bigg\|\frac{\partial u(s)}{\partial s}\bigg\|_{L^2(\bO)}^2ds= \bE(u_0),
				\end{align}
				where $\nabla_{\bM}$ is gradient of $\bE$ on the tangent $\bM$, i.e., $\nabla_{\bM}\bE(u)=\pi_u(\nabla\bE)$.
			\end{remark}

			\subsection{Uniqueness}\label{Subsec-Uniqueness} 
			Now, let us show the uniqueness of  strong solutions of the constrained problem \eqref{eqn-main-problem}.
			
			\begin{proof}
				Let $u_1, u_2$ be two strong solutions to the problem \eqref{eqn-main-problem} such that
				\begin{align}\label{eqn-reg}
					\left\{
					\begin{aligned}
						u_i & \in C([0,T]; L^p(\bO) \cap H_0^1(\bO)) \cap  L^{2p-2}(0,T;L^{2p-2}(\bO)),\\
						u_i & \in L^2(0,T; D(A))\ \text{ and }\ \frac{\partial u_i}{\partial t} \in L^2(0,T; L^2(\bO)), 	
					\end{aligned}
					\right.
				\end{align}
				for each $i=1,2$. Then, subtracting their corresponding equations, for any $\varphi \in C_0^\infty((0,T);L^2(\bO))$, we have
				\begin{align*}
					& \left(\frac{\partial}{\partial t}(u_1(t)-u_2(t)), \varphi\right) = (\Gn(u_1(t)) - \Gn(u_2(t)), \varphi),
				\end{align*}
				where we have employed that $\Gn(u_i(t)) \in L^2(\bO)$, for a.e. $t\in(0,T)$, which is because of the regularity of each strong solution $u_i,$ from \eqref{eqn-reg}, for each $i=1,2$. By using the Absolute Continuity Lemma \ref{Thm-Abs-cont} and the fact that $ C_0^\infty((0,T); L^2(\bO) )$ is dense in $C([0,T]; L^p(\bO) \cap H_0^1(\bO)) \cap L^2(0,T; D(A))\cap L^{2p-2}(0,T;L^{2p-2}(\bO))$, we deduce
				\begin{equation} \label{eqn-strong-uniqueness}
					\frac{1}{2}\frac{d}{dt}\norm{u_1(t)-u_2(t)}_{L^2(\bO)}^2 = (\Gn(u_1(t)) -\Gn(u_2(t)), u_1(t) -u_2(t)),\ \text{ for a.e. } t\in[0,T].
				\end{equation}
				Utilizing the monotonicity inequality \eqref{eqn-G-mono-inequality} in \eqref{eqn-strong-uniqueness}, we deduce
				\begin{align*}
					&\frac{1}{2}\frac{d}{dt}\norm{u_1(t)-u_2(t)}_{L^2(\bO)}^2\\
					&\leq -\frac{1}{2}\|\nabla(u_1-u_2)\|_{L^2(\bO)}^2-\langle|u_1|^{p-2}u_1 -|u_2|^{p-2}u_2, u_1-u_2\rangle\\
					&\quad + \bigg[\norm{\nabla u_1}_{L^2(\bO)}^2 + \frac{1}{2}(\norm{\nabla u_1}_{L^2(\bO)} + \norm{\nabla u_2}_{L^2(\bO)})^2 \norm{ u_2}_{L^2(\bO)}^2\\
					&\qquad \quad+\norm{u_1}_{L^p(\bO)}^p + C(p,\abs{\bO})\big(\norm{u_1}_{L^p(\bO)}^{p-1} + \norm{u_2}_{L^p(\bO)}^{p-1}\big)\norm{u_2}_{L^2(\bO)}^2 \bigg]\norm{u-u_2}_{L^2(\bO)}^2.
				\end{align*}
				Clearly, using the relation \eqref{eqn-nonlinear-est}, we can write the above inequality as
				\begin{align*}
					&\frac{d}{dt}\norm{u_1(t)-u_2(t)}_{L^2(\bO)}^2+\|\nabla(u_1 (t)- u_2(t))\|_{L^2(\bO)}^2+\frac{1}{2^{p-1}}\|u_1(t)-u_2(t)\|_{L^p(\bO)}^p\\
					& \leq  2 \bigg[\norm{\nabla u_1(t)}_{L^2(\bO)}^2 + \frac{1}{2}(\norm{\nabla u_1(t)}_{L^2(\bO)} + \norm{\nabla u_2(t)}_{L^2(\bO)})^2 \norm{ u_2(t)}_{L^2(\bO)}^2\\
					&\qquad \quad+\norm{u_1(t)}_{L^p(\bO)}^p + C(p,\abs{\bO})\left(\norm{u_1(t)}_{L^p(\bO)}^{p-1} + \norm{u_2(t)}_{L^p(\bO)}^{p-1}\right)\norm{u_2(t)}_{L^2(\bO)}^2\bigg]\\
					&\quad\quad\times\norm{u_1(t) - u_2(t)}_{L^2(\bO)}^2, \ \text{ for a.e. } t\in[0,T].
				\end{align*}
				By means of Gr\"onwall's inequality 
				\begin{align}
					& \norm{u_1(t)-u_2(t)}_{L^2(\bO)}^2 + \int_0^t\|\nabla(u_1(s) -u_2(s))\|_{L^2(\bO)}^2 ds + \frac{1}{2^{p-1}}\int_0^t\|u_1(s)-  u_2(s)\|_{L^p(\bO)}^p ds \\
					& \leq \norm{u_1(0)-u_2(0)}_{L^2(\bO)}^2\exp\bigg\{\int_0^T2 \bigg[\norm{\nabla u_1(t)}_{L^2(\bO)}^2+ \norm{u_1(t)}_{L^p(\bO)}^p\\
					& \quad +  \frac{1}{2}(\norm{\nabla u_1(t)}_{L^2(\bO)}  + \norm{\nabla u_2(t)}_{L^2(\bO)})^2 \norm{ u_2(t)}_{L^2(\bO)}^2\\
					& \quad + C(p,\abs{\bO})\left(\norm{u_1(t)}_{L^p(\bO)}^{p-1} + \norm{u_2(t)}_{L^p(\bO)}^{p-1}\right)\norm{u_2(t)}_{L^2(\bO)}^2 \bigg]\\
					& \quad \times\norm{u_1(t) - u_2(t)}_{L^2(\bO)}^2\bigg\}, \ \text{ for all } t\in[0,T]. \label{eqn-uniqueness}
				\end{align}
				Utilizing the assumption that $u_1(0) = u_2(0)$ and the fact that $u_1$ and $u_2$ are strong solutions to the problem \eqref{eqn-main-problem} having the regularity given in \eqref{eqn-reg}, we finally obtain $u_1 (t)= u_2(t),$ for a.e. $t\in[0,T]$ in $L^2(\bO)$.
			\end{proof}
			Hence, the proof of Theorem \ref{Thm-main-wellposed} is concluded.

			\begin{remark}
				Note that using \eqref{eqn-grad_M-grad-rel} in \eqref{eqn-energy-eqn}, we deduce
				\begin{align}
					\bE(u(t))+\int_0^t\|\nabla \bE(u(s))\|_{L^2(\bO)}^2ds = \int_0^t\abs{\bS(u(s))}^2ds + \bE(u_0),\ \text{ for } t\in[0,T]. \label{eqn-energy-dissipative-2}
				\end{align}
				Now, using the definition of $\bE$ (see \eqref{Eu}) and the formula \eqref{eqn-Lap-Nn}, we infer that
				\begin{align}
					\|\nabla \bE(u)\|_{L^2(\bO)}^2 & = \norm{-\Delta u + \abs{u}^{p-2}u }_{L^2(\bO)}^2\\
					& = \norm{\Delta u}_{L^2(\bO)}^2 + \norm{u}_{L^{2p-2}(\bO)}^{2p-2} - 2(\Delta u, \abs{u}^{p-2}u)\\
					& = \norm{\Delta u}_{L^2(\bO)}^2 + \norm{u}_{L^{2p-2}(\bO)}^{2p-2}  + 2(p-1)\|\abs{u}^{\frac{p-2}{2}} \nabla u\|_{L^2(\bO)}^2. \label{eqn-L^2-grad-E(u)}
				\end{align}
				Finally, using \eqref{Eu} and \eqref{eqn-S(u)} in \eqref{eqn-energy-dissipative-2}, we assert that
				\begin{align*}
					&\frac{1}{2} \norm{\nabla u(t)}_{L^2(\bO)}^2 + \frac{1}{p} \norm{ u(t)}_{L^p(\bO)}^p \\
					&\quad + \int_0^t \left[\norm{\Delta u(s)}_{L^2(\bO)}^2 + \norm{u(s)}_{L^{2p-2}(\bO)}^{2p-2} + 2(p-1)\big\|\abs{u(s)}^{\frac{p-2}{2}} \nabla u(s)\big\|_{L^2(\bO)}^2\right] ds\\
					& =  \int_0^t(\norm{\nabla u(s)}_{L^2(\bO)}^2 + \norm{u(s)}_{L^p(\bO)}^p)^2 ds + \frac{1}{2} \norm{\nabla u_0}_{L^2(\bO)}^2 + \frac{1}{p} \norm{ u_0}_{L^p(\bO)}^p.
				\end{align*}
				By taking supremum over $t\in [0,T]$, we obtain
				\begin{align*}
					&\frac{1}{2} \sup_{t\in[0,T]} \norm{\nabla u(t)}_{L^2(\bO)}^2 + \frac{1}{p} \sup_{t\in[0,T]} \norm{ u(t)}_{L^p(\bO)}^p \\
					& \quad+ \int_0^T \left[\norm{\Delta u(s)}_{L^2(\bO)}^2 + \norm{u(s)}_{L^{2p-2}(\bO)}^{2p-2} + 2(p-1)\big\|\abs{u(s)}^{\frac{p-2}{2}} \nabla u(s)\big\|_{L^2(\bO)}^2\right] ds\\
					& = \int_0^T(\norm{\nabla u(s)}_{L^2(\bO)}^2 + 2\norm{u(s)}_{L^p(\bO)}^p)^2 ds + \frac{1}{2} \norm{\nabla u_0}_{L^2(\bO)}^2 + \frac{1}{p} \norm{ u_0}_{L^p(\bO)}^p\\
					& \leq C(T, \norm{u_0}_{L^p(\bO) \cap H_0^1(\bO)}).
				\end{align*}
				In particular, we infer that
				\begin{align}
					& \int_0^T  \left[\norm{\Delta u(s)}_{L^2(\bO)}^2 + \norm{u(s)}_{L^{2p-2}(\bO)}^{2p-2} + (p-1)\|\abs{u(s)}^{\frac{p-2}{2}} \nabla u(s)\|_{L^2(\bO)}^2\right] ds\\
					& \leq C(T, \norm{u_0}_{L^p(\bO) \cap H_0^1(\bO)}).\label{eqn-est-D(A)-L^{2p-2}}
				\end{align}
				This shows that $u\in L^2(0,T; D(A)) \cap L^{2p-2}(0,T; L^{2p-2}(\bO))$. Moreover,
				\begin{align*}
					\int_0^T \big\|\abs{u(s)}^{\frac{p-2}{2}} \nabla u(s)\big\|_{L^2(\bO)}^2 ds\leq C(T, \norm{u_0}_{L^p(\bO) \cap H_0^1(\bO)}).
				\end{align*}
				On the other hand, by Sobolev's embedding, we have
				\begin{align*}
					W^{1,2}(\bO) = H^1(\bO) \embed L^p
					(\bO) & \ \text{ for }\ d >2 \ \text{ and }\ 2 \leq p \leq\frac{2d}{d-2},
					\\
					W^{1,2}(\bO) = H^1(\bO) \embed L^p(\bO) & \ \text{ for }\ d =2  \ \text{ and }\ 2 \leq p < \infty.
				\end{align*}
				Therefore, by using above embeddings and Poincar\'e's inequality, we deduce
				\begin{align*}
					\int_0^t\norm{u(s)}_{L^{\frac{pd}{d-2}}(\bO)}^p ds 
					& \leq C(p)\int_0^t \big\|\abs{u(s)}^{\frac{p-2}{2}} \nabla u(s)\big\|_{L^2(\bO)}^2 ds < \infty.
				\end{align*}
				Thus $u\in L^p (0,T; L^{\frac{pd}{d-2}}(\bO))$
				for $d>2,\ \displaystyle p\in \left[2, \frac{2d}{d-2}\right]$.
			\end{remark}

			Next, it is worth noting that some other regularity properties of the solution $u$ follow from the interpolation inequality.
			
			\begin{remark}
				Observe that for $\gamma\in(0,1)$, using interpolation inequality and \eqref{eqn-est-D(A)-L^{2p-2}},  we have
				\begin{align*}
					\int_0^T\|A^{\gamma}u(t)\|_{L^2(\bO)}^{\frac{2}{\gamma}}dt 
					& \leq \int_0^T\|u(t)\|_{L^2(\bO)}^{\frac{2(1-\gamma)}{\gamma}}\|Au(t)\|_{L^2(\bO)}^2dt\\
					&\leq \sup_{t\in[0,T]}\|u(t)\|_{L^2(\bO)}^{\frac{2(1-\gamma)}{\gamma}}
					\int_0^T\|Au(t)\|_{L^2(\bO)}^2dt\\
					&\leq C\big(T, \|u_0\|_{L^p(\bO)\cap H_0^1(\bO)}\big)\sup_{t\in[0,T]}\|u(t)\|_{L^2(\bO)}^{\frac{2(1-\gamma)}{\gamma}} < \infty.
				\end{align*}
				Thus $u \in L^{\frac{2}{\gamma}}(0,T; D(A^\gamma))$, for any dimension $d\geq 1$.
			\end{remark}

			\begin{remark}
				If one considers the following heat equation:
				\begin{align}
					\left\{
					\begin{aligned}
						\frac{\partial u(t)}{\partial t} & = \Delta u(t), \  t>0\\
						u(0) & = u_0, \\
						u(t)|_{\partial\bO} & = 0,
					\end{aligned}
					\right.
				\end{align}
				then the unique solution can be represented as $u(t)=e^{\Delta t}u_0$. From \cite[Proposition 48.4]{PQ+PS-07}, we infer that for any $1 \leq p < q \leq\infty$, $t > 0$ and $f\in L^p(\mathcal{O})$,
				\begin{align}
					\|e^{\Delta t}f\|_{L^q(\bO)}\leq\frac{1}{(4\pi t)^{\frac{d}{2}\left(\frac{1}{p}-\frac{1}{q}\right)}}\|f\|_{L^p(\bO)}.
				\end{align}
				Therefore, for $u_0\in L^p(\bO)$, we estimate
				\begin{align}\label{eqn-heat_est}
					\|e^{\Delta t}u_0\|_{L^q(0,T;L^r(\bO))}^q
					& = \int_0^T\|e^{\Delta t}u_0\|_{L^r(\bO)}^qdt\leq\int_0^T\bigg(\frac{1}{(4\pi t)^{\frac{d}{2}\left(\frac{1}{p}-\frac{1}{q}\right)}}\|u_0\|_{L^p(\bO)}\bigg)^qdt\\ 
					& = \|u_0\|_{L^p(\bO)}^q(4\pi T)^{-\frac{dq}{2}\left(\frac{1}{p}-\frac{1}{r}\right)+1},
				\end{align}
				provided $\frac{dq}{2}\left(\frac{1}{p}-\frac{1}{r}\right)<1$. For $2\leq p<\frac{2d}{d-2}$, taking $q=r=2p-2$ in \eqref{eqn-heat_est}, we deduce
				\begin{align}
					\|e^{\Delta t}u_0\|_{L^{2p-2}(0,T;L^{2p-2}(\bO))}^{2p-2}\leq \|u_0\|_{L^p(\bO)}^{2p-2}(4\pi T)^{-\frac{d(p-2)}{2p}+1}.
				\end{align}
				Therefore, for large $p$, i.e., $p\geq\frac{2d}{d-2}$, the nonlinear term $|u|^{p-2}u$ helps us to obtain the extra regularity $u\in L^{2p-2}(0,T;L^{2p-2}(\bO))$.
			\end{remark}

			\begin{remark}
				Observe from the estimate \eqref{eqn-estimate-Lp+H01} that the $L^p(\bO)\cap H_0^1(\bO)-$norm remains bounded uniformly with respect to time $T$. Consequently, we may express the solution space for $u$ as
				\begin{align*}
					C ([0,\infty); L^p(\bO) \cap H_0^1(\bO)\cap\bM)\cap L_{\loc}^2(0,\infty; D(A))\cap L_{\loc}^{2p-2}(0,\infty; L^{2p-2}(\bO)).
				\end{align*}
				Hence this allows us to analyze the asymptotic behaviour of the strong solution.
			\end{remark}
			
			\section{Asymptotic analysis}\label{Sec-Asy-anal-Antonelli}
			This section focuses on investigating the existence of a  ground state solution for the stationary equation \eqref{eqn-stationary}, as well as examining the long-time behavior of solutions to the time-dependent problem \eqref{eqn-main-problem-copy}. The analysis follows the framework developed by Antonelli et al. \cite{PA+PC+BS-24}.
			  We first introduce the concept of a ground state, followed by a proof that the corresponding energy functional $\bE$ is coercive and weakly lower semicontinuous. Next, we establish that the ground state is a local minimizer of the variational problem \eqref{eqn-min-energy-problem}. Subsequently, we prove the uniqueness of the positive stationary solution, based on the framework developed in \cite{TO-92}. In the final step, we establish that the unique strong positive solution of the time-dependent problem \eqref{eqn-main-problem-copy} converges, in the strong topology, to the unique positive ground state.
			
			In what follows, we choose and fix $2\leq p<\infty$ and suppose $\bO$ is any bounded smooth domain in $\R^d$ for the entire section (until unless specified), and examine the asymptotic dynamics of the solution to the following problem in $(0,T)\times\bO$:
			\begin{align}\label{eqn-main-problem-copy}
				\left\{
				\begin{aligned}
					\frac{\partial u(t)}{\partial t} & = \Delta u(t) -|u(t)|^{p-2}u(t) + \big( \norm{\nabla u(t)}_{L^2(\bO)}^2 + \norm{u(t)}_{L^p(\bO)}^p \big) u(t),\\
					u(0) & = u_0, \\
					u(t)|_{\partial\bO} & = 0,
				\end{aligned}
				\right.
			\end{align}
			where $u:[0,\infty) \times\bO \to \R$ with $u(t)=u(t,x)$.
			
			\begin{remark}
				Note that if we consider 
				$$p < \left\{
				\begin{aligned}	
					& \infty, \ d\leq 2,\\
					& \frac{2d}{d-2},\ d \geq 3, 
				\end{aligned}\right.$$
				then we have the compact embedding $H_0^1(\bO) \embed L^p(\bO)$. In this case,  our results, Theorem \ref{Thm-main-wellposed} and Theorem \ref{Thm-main-assymp}, coincide with \cite[Theorem 1.2 and Theorem 1.7]{PA+PC+BS-24}, respectively. By considering initial data in the space $L^p(\bO) \cap H_0^1(\bO) \cap \bM$, we extend and refine the results of Antonelli et al. Specifically, we generalize \cite[Theorem 1.2]{PA+PC+BS-24} to hold for all $2 \leq p < \infty$ and in any spatial dimension $d \geq 1$. Furthermore, when $u_0 > 0$ and $\mathcal{O}$ is a smooth bounded domain, our results also cover \cite[Theorem 1.7]{PA+PC+BS-24} under the same broader conditions on $p$ and $d$.
			\end{remark}

			\subsection{Ground state solutions}
			In this subsection, we begin by defining the notion of ground states and proceed to establish the existence of a ground state solution to problem \eqref{eqn-main-problem-copy}, utilizing the weak lower semicontinuity and coercivity of the energy functional $\bE$ defined in \eqref{def-new-energy}.

			Next, we show that the minimizer of the energy functional defined as
			\begin{equation}\label{def-new-energy}
				L^p(\bO) \cap H_0^1(\bO) \ni u \mapsto \bE(u) := \frac{1}{2} \int_\bO \abs{\nabla u(x)}^2 dx + \frac{1}{p} \int_\bO \abs{ u(x)}^p dx \in \R,
			\end{equation}
			under the  constraint on the total mass:
			\begin{align}
				\inf \big\{\bE(u) : u \in L^p(\bO) \cap H_0^1(\bO), \norm{u}_{L^2(\bO)} = 1\big\}, \label{eqn-min-energy-problem}
			\end{align}
			is a weak solution to the stationary problem corresponding to \eqref{eqn-main-problem-copy}, i.e.,
			\begin{align}
				\Delta u - \abs{u}^{p-2}u +  \norm{\nabla u}_{L^2(\bO)}^2 u + \norm{u}_{L^p(\bO)}^pu
				=0, \ \text{ in }\ L^{\frac{p}{p-1}}(\bO)+H^{-1}(\bO), \label{eqn-stationary}
			\end{align}
			which is usually known as the \emph{ground state solution}.
			\begin{definition}\label{Def-ground-state}
				We say that $u$ is a \emph{ground state}, if $u$ is a minimizer of the energy functional \eqref{Eu} and solves the elliptic equation \eqref{eqn-stationary}. 
			\end{definition}
			
			
			\begin{lemma}\label{Lem-E-wls-coer}
					The functional $\bE$, defined in \eqref{def-new-energy}, is weakly lower semicontinuous, i.e., if $\{u_n\}_{n\in\N}$ is a sequence in $L^p(\bO) \cap H_0^1(\bO),$ that converges weakly to $u$, then
					\begin{align}
						\bE(u) \leq \liminf_{n\to\infty} \bE(u_n),
					\end{align}
					and coercive, i.e.,
					\begin{align}
						\frac{\bE(u)}{\|u\|_{L^p(\bO) \cap H_0^1(\bO)}} \to \infty\ \ \text{as}\ \ \|u\|_{L^p(\bO) \cap H_0^1(\bO)} \to \infty.
					\end{align}
			\end{lemma}
			
			\begin{proof}
				First note from the definition of $\bE$ that, for any $u \in L^p(\bO) \cap H_0^1(\bO),$ 
				\begin{align*}
					\bE(u) = \frac{1}{2}  \norm{u}_{H_0^1(\bO)}^2 + \frac{1}{p} \norm{ u}_{L^p(\bO)}^p.
				\end{align*}
				Let $\{u_n\}_{n\in\N}$ be a sequence in $L^p(\bO) \cap H_0^1(\bO),$ that converges weakly to $u$ so that 
				\begin{align*}
					u_n \rightharpoonup u \text{ in }\ L^p(\bO) \ \text{ and }\ u_n \rightharpoonup u\ \text{ in }\ H_0^1(\bO). 
				\end{align*}
				Since every norm is weakly lower semicontinuous \cite[see Ch.\;7 Sec.\;3]{DM+DZ-98}, it implies
				\begin{align}\label{eqn-WLC-1}
					\frac{1}{2^{1/2}}\norm{u}_{H_0^1(\bO)} \leq \frac{1}{2^{1/2}}\liminf_{n\to\infty} \norm{u_n}_{H_0^1(\bO)} \ \text{ and }\  \frac{1}{p^{1/p}}\norm{u}_{L^p(\bO)} \leq \frac{1}{p^{1/p}}\liminf_{n\to\infty} \norm{u_n}_{L^p(\bO)}.
				\end{align}
				Squaring the first term of the inequality \eqref{eqn-WLC-1} and utilizing the standard properties of limit infimum yields
				\begin{align}\label{eqn-WLC-2}
					\frac{1}{2}\norm{u}_{H_0^1(\bO)}^2 & \leq \frac{1}{2}\left(\liminf_{n\to\infty} \norm{u_n}_{H_0^1(\bO)}\right)^2 \leq \frac{1}{2}\liminf_{n\to\infty} \norm{u_n}_{H_0^1(\bO)}^2.
				\end{align}
				Similarly, taking $p^{\mathrm{th}}$ power on both sides of the second term in the inequality \eqref{eqn-WLC-1}, we deduce
				\begin{align}\label{eqn-WLC-3}
					\frac{1}{p}\norm{u}_{L^p(\bO)}^p \leq \frac{1}{p}\left(\liminf_{n\to\infty} \norm{u_n}_{L^p(\bO)}\right)^p  \leq \frac{1}{p} \liminf_{n\to\infty} \norm{u_n}_{L^p(\bO)}^p.
				\end{align}
				Now, adding the inequalities \eqref{eqn-WLC-2} and \eqref{eqn-WLC-3}, we get 
				\begin{align}
					\bE(u) = \frac{1}{2}\norm{u}_{H_0^1(\bO)}^2 + \frac{1}{p}\norm{u}_{L^p(\bO)}^p \leq 
					\frac{1}{2}\liminf_{n\to\infty} \norm{u_n}_{H_0^1(\bO)}^2 + \frac{1}{p} \liminf_{n\to\infty} \norm{u_n}_{L^p(\bO)}^p.
				\end{align}
				Once again utilizing the standard properties of limit infimum, we conclude 
				\begin{align*}
					\bE(u) \leq  \liminf_{n\to\infty} \left(\frac{1}{2}\norm{u_n}_{H_0^1(\bO)}^2 + \frac{1}{p} \norm{u_n}_{L^p(\bO)}^p \right) =  \liminf_{n\to\infty} \bE(u_n).
				\end{align*}
				Hence $\bE$ is weakly lower semicontinuous.
				
				For the coercivity of $\bE$, we consider 
				\begin{align*}
					\bE(u) = \frac{1}{p} \left[\frac{p}{2}\norm{u}_{H_0^1(\bO)}^2 + \norm{ u}_{L^p(\bO)}^p \right] \geq \frac{1}{p}\big[\norm{u}_{H_0^1(\bO)}^2 + \norm{ u}_{L^p(\bO)}^2-1\big],
				\end{align*} 
				where we have used the fact that $x^2\leq 1+x^p$, for all $x\geq 0$ and $2\leq p <\infty$. The inequality obtained above leads directly yields
				\begin{align*}
					\frac{\bE(u)}{\sqrt{\norm{u}_{H_0^1(\bO)}^2 + \norm{ u}_{L^p(\bO)}^2}}\geq \frac{\frac{1}{p}\big[\norm{u}_{H_0^1(\bO)}^2 + \norm{ u}_{L^p(\bO)}^2-1\big]}{\sqrt{\norm{u}_{H_0^1(\bO)}^2 + \norm{ u}_{L^p(\bO)}^2}} \to \infty\ \ \text{as}\ \ \|u\|_{L^p(\bO) \cap H_0^1(\bO)} \to \infty.
				\end{align*}
				Hence $\bE$ is coercive and it concludes the proof.
			\end{proof}

			\begin{lemma}\label{lem-weakly-closed}
				The set $$\bA := \big\{u \in L^p(\bO) \cap H_0^1(\bO): \norm{u}_{L^2(\bO)} = 1\big\}$$ is a weakly sequentially closed subset of $L^p(\bO) \cap H_0^1(\bO)$.
			\end{lemma}
			
			\begin{proof}
				First, observe that $(L^p(\bO) \cap H_0^1(\bO), \norm{\cdot}_{L^p(\bO) \cap H_0^1(\bO)})$ is a reflexive Banach space \cite[Theorem 26.10]{JCR-20}.
				Consider a sequence $\{u_n\}_{n\in\N}$ in $\bA$ satisfying $u_n \rightharpoonup u$ in $L^p(\bO) \cap H_0^1(\bO)$. Since the embedding $H_0^1(\bO) \embed L^2(\bO)$ is compact, we may extract a (not relabeled) subsequence such that $u_n \to u$ in $L^2(\bO)$.
				Since $\|u_n\|_{L^2(\bO)}=1$, the strong convergence in  $L^2(\bO)$ implies that $\|u\|_{L^2(\bO)}=1$ and $u\in\bA$.
			\end{proof}
			
			In the following result (motivated from \cite[Theorem 2.1]{MS-08}), we show that
			$\bU $ is a minimizer of the energy functional $\bE$ defined in \eqref{def-new-energy} and it satisfies the stationary problem \eqref{eqn-stationary}. Therefore, $\bU $ is a ground state solution.
			
				\begin{theorem}\label{Thm-ground-soln}
					Suppose $\bE$ be a energy functional defined in \eqref{def-new-energy}. Then there exists a ground state $\bU$, i.e., $\bU$ is a solution to the minimization problem
					\begin{align}\label{eqn-mini}
						\min_{\bU\in L^p(\bO) \cap H_0^1(\bO)} \big\{\bE(\bU) : \norm{\bU}_{L^2(\bO)} = 1\big\},
					\end{align}
					which also solves the corresponding stationary equation \eqref{eqn-stationary}.
			\end{theorem}
			
			\begin{proof}
				We want to find a minimizer of the functional $\bE(\cdot)$ defined in \eqref{def-new-energy} under the constraint that $\|u\|_{L^2(\bO)} = 1$.
				This is a constrained optimization problem, and we use the method of \textit{Lagrange multipliers} to solve it.
				
				Since $L^p(\bO) \cap H_0^1(\bO)$ is a reflexive Banach space with norm $\|\cdot\|_{L^p(\bO) \cap H_0^1(\bO)}$ and from Lemma \ref{lem-weakly-closed}, $\bA$ is a weakly closed subset in $L^p(\bO) \cap H_0^1(\bO)$. Furthermore, by using  Lemma \ref{Lem-E-wls-coer}, we infer that $\bE$ is coercive and weakly lower semicontinuous. Therefore, by \cite[Theorem 1.2]{MS-08}, $\bE$ attains its infimum at a point $\wtilde{u}\in\bA$.
				
				To derive the variational problem for $\bE$, let us denote the first order Fréchet-derivatives at $v\in L^p(\bO) \cap H_0^1(\bO)$ as
				\begin{align*}
					\d_v\bE &\in \mathcal{L}(L^p(\bO) \cap H_0^1(\bO) ;\mathbb{R}).
				\end{align*}
				Note that $\bE$ is Fréchet-differentiable in $L^p(\bO) \cap H_0^1(\bO)$ with
				\begin{align*}
					\d_v\bE(u) = \int_{\bO}  \left[\nabla u(x)\cdot\nabla v(x) + \abs{u(x)}^{p-2}u(x) v(x)\right]dx.
				\end{align*}
				
				Moreover, to incorporate the constraint \( \| u \|_{L^2(\bO)} = 1 \), let us introduce a functional $G: L^p(\bO) \cap H_0^1(\bO) \to \R$ by
				\begin{align*}
					G(u) := \frac{1}{2} \left( \int_\bO |u(x)|^2 \, dx - 1 \right),
				\end{align*}
				which is also Fréchet-differentiable with
				\begin{align*}
					\d_vG(u) = \int_{\bO} u(x)v(x)dx.
				\end{align*}
				In particular, at any point $u\in \bA$
				\begin{align*}
					\d_uG(u) = \int_{\bO}u^2(x) dx = \norm{u}_{L^2(\bO)}^2=1 \ne 0,
				\end{align*}
				and by the implicit function theorem, the set $\bA = G^{-1}(0)$.
				
				The functional $G$ is in fact the Lagrange multiplier associated with the sphere constraint in \eqref{eqn-Sphere-bM}.
				Thus, there exists a parameter $\lambda\in\R$ and a functional $\Ln : L^p(\bO) \cap H_0^1(\bO)\times\R \to \R$ such that
				\begin{align}\label{eqn-lagrange}
					\d_v \Ln(u,\lambda) = \d_v\bE(u) - \lambda \d_vG(u) & =\langle- \Delta u + \abs{u}^{p-2}u - \lambda u, v\rangle,
				\end{align}
				for any $v\in L^p(\bO) \cap H_0^1(\bO)$.
				Then, the first variation of $\Ln$, i.e., $\d_{v}\Ln(u, \lambda) =0$ gives the following \textit{Euler-Lagrange} system:
				\begin{align*}
					\int_{\bO}\left[\nabla u(x) \nabla v(x) + |u(x)|^{p-2} u(x) v(x) - \lambda u(x) v(x)\right] dx= 0.
				\end{align*}
				Inserting $v=u$ in the equation above, we infer
				\begin{align*}
					\lambda = \|\nabla u\|_{L^2(\bO)}^2+\|u\|_{L^p(\bO)}^p.
				\end{align*}
				Therefore, a local minimizer $u=\bU $ of the problem \eqref{eqn-min-energy-problem} satisfies the stationary equation \eqref{eqn-stationary}.
			\end{proof}
			
			\begin{remark}
				Since $\bE(u) = \bE(|u|)$, we may assume that the ground state solution obtained above is non-negative, i.e., $u \geq 0$.
			\end{remark}

			\subsection{Uniqueness of positive ground state solution}\label{Sec-uni-ground}
			This subsection is motivated by the work of Ouyang \cite{TO-92}. The paper  \cite{TO-92} considered the problem of uniqueness of solutions to the system \eqref{eqn-stationary-trans} without constraint in compact Riemannian  manifolds and bounded domains (\cite[Theorems 1 and 2]{TO-92}). The author in   \cite{TO-92} provided a detailed proof in the case of compact Riemannian  manifolds. Using the methodology adopted in  \cite{TO-92}, we discuss the case of bounded domains in detail.  Our aim is to show that a positive classical solution $u$ to the problem \eqref{eqn-stationary-trans} is unique. First, we discuss the strong maximum principle, the concept of sub- and super-solutions of the problem \eqref{eqn-stationary-trans} which corresponds to the problem \eqref{eqn-stationary}. Then, by using a well-known result on the sub-super-solution method \cite[Lemma 2.6]{JLK+FWW-75}, we conclude this section by proving that the positive solution to the problem \eqref{eqn-stationary-trans} on any bounded smooth domain is, in fact, unique.
			
			Let us now look at a positive classical solution $u$ of the problem \eqref{eqn-stationary}. We infer from \eqref{eqn-lagrange} that $u$ satisfies the following problem:
			\begin{align}\label{eqn-stationary-trans}
				\left\{
				\begin{aligned}
					\Delta u - u^{p-1} + \lambda u = 0, & \ \text{ in }\ \bO,\\
					u>0, & \ \text{ in }\ \bO,\\
					u = 0, &\ \text{ on }\ \partial\bO,\\
					\int_{\bO}|u(x)|^2dx=1,&
				\end{aligned}
				\right.
			\end{align}
			where 
			$\lambda > 0$.

			First, we provide some important definitions and results, followed by a strong maximum principle, taken from \cite[Section 2]{TO-92} (in smooth bounded domain case):
			
			\begin{theorem}[Strong maximum principle]\label{Thm-Str-max-prin}
				Let $u\in C^2(\bO)\cap C(\overline{\bO})$ be a classical solution to the problem
				\begin{equation}
					\left\{
					\begin{aligned}
						\Delta u+h u\leq 0, & \ \text{ in }\ \bO,\\
						u\geq 0, & \ \text{ in }\ \bO,\\
						u = 0, &\ \text{ on }\ \partial\bO,
					\end{aligned}
					\right.
				\end{equation}
				where $h$ is a bounded function. Then $u>0$ on $\bO$.
			\end{theorem}
			
			\begin{proof}
				Let $M_+:=\{x\in\bO:u(x)>0\}\subset\mathcal{O}$ and $M_0:=\{x\in\bO:u(x)=0\}$. Let us choose $x_0\in M_+$ such that
				\begin{align*}
					\mathrm{dist}(x_0,M_0)<	\mathrm{dist}(x_0,\partial\bO),
				\end{align*}
				and consider the largest ball $B\subset M_+$ centered at $x_0$. Then, there exists a point $y\in\partial B\cap M_0$ such that 
				$$u(y)=0\ \text{ and }\ u>0\ \text{ in }\ B.$$ 
				Thus, by the Hopf Lemma (\cite[Theorem 3.5]{DG+NST-01} or \cite{EH-27}), we immediately have $\nabla u(y)\neq 0$, which contradicts the fact that $y$ is an interior minimum of $\bO$.
			\end{proof}

			\begin{definition}[{\cite[Definition 1]{TO-92}}]\label{def-super/sub-sol}
				Let $f:\mathbb{R}\to\mathbb{R}$ be a $C^1$ function. A function $v\in C^2(\bO)\cap C(\overline{\bO})$ is said to be a super-solution (sub-solution) of the problem
				\begin{align}\label{eqn-gen-elliptic}
					-\Delta u = f(u) \ \text{ in }\ \bO,\ u=0\ \text{ on }\ \partial\bO,
				\end{align}
				if $v$ satisfies the inequality
				\begin{align*}
					-\Delta v \geq (\leq)\; f(v) \ \text{ in }\ \bO, \ v=0\ \text{ on }\ \partial\bO.
				\end{align*}
				Moreover, $v$ is called a super-solution and sub-solution in the weak sense, if $v\in H_0^1(\bO)$ and it  satisfies
				\begin{align}\label{eqn-sup-weak}
					\int_{\bO}\nabla v(x)\cdot\nabla\phi(x) dx \geq \; \int_{\bO}f(v(x))\phi(x)dx,\ \text{ for all }\ \phi\in C_0^{\infty}(\bO),\ \phi\geq 0,
				\end{align}
				and 
				\begin{align}\label{eqn-sub-weak}
					\int_{\bO}\nabla v(x)\cdot\nabla\phi(x) dx \leq\; \int_{\bO}f(v(x))\phi(x)dx,\ \text{ for all }\ \phi\in C_0^{\infty}(\bO),\ \phi\geq 0,
				\end{align}
				respectively.
			\end{definition}

			\begin{proposition}\label{Prop-sub-soln}
				Let $f:\mathbb{R}\to\mathbb{R}$ be a $C^1$ function, and let $u_1,u_2\in C^2(\bO)\cap C(\overline{\bO})$ be sub-solutions of problem \eqref{eqn-gen-elliptic}.
				\dela{\begin{align}\label{eqn-sup-sol}
						\Delta u_i + f(u_i) = 0 \ \text{ in }\ \bO,\ u_i=0\ \text{ on }\ \partial\bO, \ i=1,2.
				\end{align}}
				Let us define a function $u$ by 
				\begin{align*}
					u(x)=\max\{u_1(x),u_2(x)\},\ \text{ for }\ x\in\overline{\mathcal{O}}.
				\end{align*}
				Then, $u$ is a sub-solution of \eqref{eqn-gen-elliptic} in the weak sense, i.e., it satisfies \eqref{eqn-sub-weak}.
			\end{proposition}
			
			\begin{proof}
				Let us take
				\begin{align}
					K_1&=\{x\in\mathcal{O}:u_1(x) > u_2(x)\},\\
					K_2&=\{x\in\mathcal{O}:u_1(x) \leq u_2(x)\}.
				\end{align}
				We first assume that $\partial K_1$ is having a piecewise $C^1$ boundary.  The case of $\partial K_1$ without  piecewise $C^1$ boundary the result can still be obtained by closely following the argument used in proof of \cite[Proposition 1]{TO-92}. For all $\phi\in C_0^{\infty}(\bO)$ with $\phi>0$, we have
				\begin{align}
					& \int_{\bO}\nabla u(x)\cdot\nabla\phi(x) dx-\int_{\bO}f(u(x))\phi(x)dx\\
					& =  \int_{K_1}\nabla u_1(x)\cdot\nabla\phi(x) dx-\int_{K_1}f(u_1(x))\phi(x)dx\\
					&\quad+\int_{K_2}\nabla u_2(x)\cdot\nabla\phi(x) dx-\int_{K_2}f(u_2(x))\phi(x)dx.\label{eqn-min-sol}
				\end{align}
				By the divergence theorem, it follows that
				\begin{align}
					& \int_{K_i}\nabla u_i(x)\cdot\nabla\phi(x) dx-\int_{K_i}f(u_i(x))\phi(x)dx\\
					&=\int_{\partial K_i}\frac{\partial u_i(x)}{\partial \nu}\phi(x)dS(x)-\int_{K_i}\Delta u_i(x)\phi(x)dx-\int_{K_i}f(u_i(x))\phi(x)dx. \label{eqn-min-int}
				\end{align}
				Utilizing \eqref{eqn-min-int} in \eqref{eqn-min-sol}, we deduce
				\begin{align}
					&\int_{\bO}\nabla u(x)\cdot\nabla\phi(x) dx-\int_{\bO}f(u(x))\phi(x)dx\\
					&=-\int_{K_1}[\Delta 	u_1(x)+f(u_1(x))]\phi(x)dx-\int_{K_2}[\Delta u_2(x)+f(u_2(x))]\phi(x)dx\\
					&\quad+\int_{\partial K_1}\frac{\partial u_1(x)}{\partial \nu}\phi(x)dS(x)+\int_{\partial K_2}\frac{\partial u_2(x)}{\partial \nu}\phi(x)dS(x) \\
					&=-\int_{K_1}[\Delta u_1(x)+f(u_1(x))]\phi(x)dx-\int_{K_2}[\Delta u_2(x)+f(u_2(x))]\phi(x)dx\\
					&\quad+\int_{\partial K_1}\bigg[\frac{\partial }{\partial \nu}(u_1(x)-u_2(x))\bigg]\phi(x)dS(x)=:J_1+J_2+J_3.
				\end{align}
				Since $u_1$ and $u_2$ are sub-solutions of \eqref{eqn-gen-elliptic}, we infer from Definition \ref{def-super/sub-sol} that $J_1\leq 0$ and $J_2\leq 0$. We only need to show that $J_3\leq 0$. In order to show $J_3\leq 0$, we note from the Hopf's Lemma that
				\begin{align}
					u_1(x)-u_2(x) > 0,\ \text{ for all } \ x\in K_1\ \text{ and }\ u_1(x)-u_2(x)=0,\ \text{ for all }\ x\in \partial K_1.
				\end{align}
				Therefore, we have
				\begin{align}
					\frac{\partial }{\partial \nu}(u_1(x)-u_2(x))\leq 0,\ \text{ for all }\ x\in\partial K_1,
				\end{align}
				so that $J_3\leq 0$. Hence $u$ is a sub-solution of \eqref{eqn-gen-elliptic}  in the  weak sense given in \eqref{eqn-sub-weak}.
			\end{proof}
			
			Next, we have a well-known result on sub-super-solution method in smooth bounded domain case, which is a consequence of \cite[Lemma 2.6]{JLK+FWW-75}.
			
			\begin{proposition}[Sub-super-solution method]\label{Prop-sub-super}
				Let $\overline{u}$ and $\underline{u}$ be super-solution and sub-solution, respectively, of the equation \eqref{eqn-gen-elliptic} with $f(u) = \lambda u - u^{p-1}$ and satisfy
				\begin{align*}
					\underline{u} < \overline{u},\ \text{ in }\ \bO.
				\end{align*}
				Then, there exists a solution $u$ of the equation \eqref{eqn-gen-elliptic} satisfying
				\begin{align*}
					\underline{u} < u < \overline{u},\ \text{ in }\ \bO.
				\end{align*}
			\end{proposition}
			
			\begin{proof}
				The proof follows from the well-known result \cite[Lemma 2.6]{JLK+FWW-75}.	
				The core of the proof relies on a classical iteration argument outlined as follows. Let us define
				\begin{align*}
					f(u) := \lambda u - u^{p-1}, \ k = \sup \{f^\prime (u): \underline{u}(x) \leq u \leq \overline{u}(x),\ \text{ for }\ x\in\bO\}
				\end{align*}
				and if necessary, add a positive constant to $k$ to insure that $k>0$. Now set $u_0 = \overline{u}$, and then, define $u_{j+1}$ inductively as the unique solution on $\bO$ of
				\begin{align}\label{eqn-iterative-prob}
					\left\{
					\begin{aligned}
						-\Delta u_{j+1}(x) + k u_{j+1}(x) & = f(u_j(x)) + k u_j(x),& \ \text{ for }\ x\in \bO,\\
						u_{j+1}(x) & = 0,& \ \text{ for }\ x \in \partial\bO.
					\end{aligned}
					\right.
				\end{align}
				Note that at $j=0$ iteration, we obtain $-\Delta u_{1} + k u_{1} = f(u_0) + k u_0$. Since $u_0 = \overline{u}$ is a super-solution of \eqref{eqn-gen-elliptic}, it implies
				\begin{align*}
					-\Delta u_{1} + k (u_{1}- \overline{u}) = f(\overline{u}) \leq -\Delta \overline{u} \implies -\Delta (u_{1} - \overline{u}) + k (u_{1}- \overline{u}) \leq 0.
				\end{align*}
				By using the maximum principle (or Hopf's Lemma), we deduce that there exists a point $x_0 \in \partial \bO$ such that
				\begin{align*}
					u_{1}(x) - \overline{u}(x) \leq u_{1}(x_0) - \overline{u}(x_0) \implies u_{1}(x) - \overline{u}(x) \leq 0 \implies u_{1}(x) \leq \overline{u}(x),\ \text{ for }\ x\in \bO,
				\end{align*} 
				where we have used the fact that $u_1, \overline{u} =0$, on $\partial\bO$.
				
				Similarly, by induction one can show that
				\begin{align*}
					\underline{u}\leq \cdots \leq u_{j+1} \leq u_j \leq \cdots \leq \overline{u},\ \text{a.e. in }\ \bO.
				\end{align*}
				The approach employed in \cite[p.\;370]{RC+DH-62} shows that the iterative sequence $\{u_j\}_{j\in \N}$ converges to a classical solution $u$ of the problem \eqref{eqn-gen-elliptic}. Since $0 < \underline{u} \leq u \leq \overline{u}$, one has $u>0$ too.
			\end{proof}
			
			We next summarize the uniqueness result from \cite[Lemma 1]{TO-92} for positive solutions of \eqref{eqn-stationary-trans}.
			
			\begin{proposition}\label{Prop-uniq-positive}
			There exists at most one solution to the constrained stationary problem \eqref{eqn-stationary-trans}.
		\end{proposition}
		
		\begin{proof}
			Assume that for some $\lambda >0$, there exist two positive solutions $u_1$ and $u_2$ of the problem \eqref{eqn-stationary-trans} with $u_1 \ne u_2$. We may assume
			\begin{align}
				u_1 \geq u_2,\ \text{ in }\ \bO.\label{eqn-Prop-uniq-1}
			\end{align}
			If $u_1 \ngeq u_2$ and $u_2 \ngeq u_1$, then we set
			\begin{align*}
				\underline{u}(x) = \max\{u_1(x), u_2(x)\},\ \text{ for }\ x\in\bO.
			\end{align*}
			It is obvious that $\underline{u} >0$ in $\bO$ and from the Proposition \ref{Prop-sub-soln} that $\underline{u}$ is a sub-solution of \eqref{eqn-stationary-trans}. By using the definition of super-solution it is easy to check that
			\begin{align*}
				0< u_c(x) = \text{constant} > \max\Big\{\lambda^{\frac{1}{p-1}}, \max_{x\in\overline{\bO}}\underline{u}(x)\Big\},
			\end{align*}
			is a super-solution of \eqref{eqn-stationary-trans}. By the sub-super-solution method (see Proposition \ref{Prop-sub-super}) there is a solution $v$ of \eqref{eqn-stationary-trans} satisfying
			\begin{align*}
				\underline{u} \leq v \leq u_c, \ \text{ in }\ \bO.
			\end{align*}
			So we may choose $v$ to replace $u_2$ such that the new pair of solutions satisfy 	\eqref{eqn-Prop-uniq-1}. Without loss of generality, we may assume that 
			\begin{align*}
				u_1 > u_2, \ \text{ on }\ \bO.
			\end{align*}
			\textit{Claim: } $u_1 > u_2,\ \text{ on }\ \bO.$
			\begin{proof}[Proof of the Claim:]
				Suppose on the contrary that $u_1 \geq u_2$ and $u_1(x_0) = u_2(x_0)$, for some $x_0 \in \bO$. Then, by setting
				\begin{align*}
					w(x) = u_1(x) - u_2(x),
				\end{align*}
				it follows from the assumptions on $u_1,u_2$ that
				\begin{align*}
					\left\{
					\begin{aligned}
						\Delta w(x) + f(x)(u_1(x) - u_2(x)) + \lambda (x)w = 0, & \ \text{ for }\ x\in \bO,\\
						w(x) \geq 0, & \ \text{ for }\ x\in\bO,\\
						w(x) = 0, &\ \text{ for }\ x\in\partial\bO,
					\end{aligned}
					\right.
				\end{align*}
				where $f(x)=-\int_0^1 (\theta u_1(x) + (1-\theta) u_2(x))^{p-1} d\theta$. 
				Since $\bO$ is bounded and $u_i\in C^2(\bO)\cap C(\overline{\bO})$, for $i=1,2$, it implies $f(x) $ is bounded on $\overline{\bO}$. Therefore, from the above system and using the fact $\lambda>0$, we deduce 
				\begin{align*}
					\left\{
					\begin{aligned}
						\Delta w(x) + f(x)w(x) \leq 0, & \ \text{ for }\ x\in\bO,\\
						w(x) \geq 0, & \ \text{ for }\ x\in\bO,\\
						w(x) = 0, &\ \text{ for }\ x\in\partial\bO.
					\end{aligned}
					\right.
				\end{align*}
				By using the strong maximum principle (see Theorem \ref{Thm-Str-max-prin}), we have $w>0$ on $\bO$, which contradicts the fact that $w(x_0) = 0$, for $x_0\in\bO$. Hence $u_1> u_2$ on $\bO$.
			\end{proof}
			Next, observe that $u_1$ and $u_2$ are solutions of \eqref{eqn-stationary-trans}, i.e.,
			\begin{align}
				\Delta u_1 - u_1^{p-1} + \lambda u_1 = 0, & \ \text{ in }\ \bO, \label{eqn-Prop-uniq-2}\\
				\Delta u_2 - u_2^{p-1} + \lambda u_2 = 0, & \ \text{ in }\ \bO.\label{eqn-Prop-uniq-3}
			\end{align}
			After multiplying \eqref{eqn-Prop-uniq-2} by $u_2$ and integrating by parts over 
			$\bO$, it follows that
			\begin{align}
				- \int_\bO \nabla u_1(x) \cdot \nabla u_2(x) dx - \int_\bO u_1^{p-1}(x) u_2(x) dx + \lambda\int_\bO u_1(x) u_2(x) dx = 0. \label{eqn-Prop-uniq-4}
			\end{align}
			Similarly, we obtain
			\begin{align}
				- \int_\bO \nabla u_1 (x)\cdot \nabla u_2(x) dx - \int_\bO u_2^{p-1}(x) u_1(x) dx + \lambda\int_\bO u_1(x) u_2(x) dx = 0. \label{eqn-Prop-uniq-5}
			\end{align}
			Now subtracting \eqref{eqn-Prop-uniq-5} from \eqref{eqn-Prop-uniq-4}, we extract
			\begin{align*}
				\int_\bO u_1(x) u_2(x) (u_1^{p-1}(x) - u_2^{p-1}(x)) dx = 0.
			\end{align*}
			But $u_1 > u_2 >0$, so the left hand side of the above integral equation must be positive. This contradiction implies $u_1 \equiv u_2$.
		\end{proof}

		\subsection{Asymptotic profile}
		In this subsection, we first produce a result that for any initial data in $L^p(\bO) \cap H_0^1(\bO)\cap\bM$, there exists a sequence of times, along which, the unique strong solution of \eqref{eqn-main-problem-copy} converges in $L^p(\bO) \cap H_0^1(\bO)$ to a stationary solution of the problem \eqref{eqn-stationary}. Subsequently, for positive initial data $u_0$, we make use of the uniqueness of positive ground states established in Section \ref{Sec-uni-ground} to further strengthen the result, demonstrating that the solution converges as time approaches infinity.

		
		The following proposition is motivated from  \cite[Proposition 4.1 and Corollary 4.2]{PA+PC+BS-24}.
		
		\begin{proposition}\label{Prop-Stationary}
				Let us choose and fix 
				$u_0 \in L^p(\bO) \cap H_0^1(\bO)\cap\bM$. Suppose
				\begin{align*}
					u & \in C([0,\infty); L^p(\bO) \cap H_0^1(\bO)\cap \bM),
				\end{align*}
				is the unique strong solution to \eqref{eqn-main-problem-copy} guaranteed by Theorem \ref{Thm-main-wellposed} such that
				\begin{align*}
					\frac{\partial u}{\partial t} \in L^2(0,\infty; L^2(\bO)).
				\end{align*}
				Then, one may extract a sequence of times $\{\tau_n\}_{n\in \N}$, $\lim_{n\to\infty}\tau_n =\infty$, such that
				\begin{align}
					u(\tau_n) \to u_{\infty} \ \text{ in }\  L^p(\bO) \cap H_0^1(\bO), \ \ \bS(u(\tau_n)) \to \bS(u_{\infty}) \ \text{ in }\ \R,\ \text{ as }\ n \to \infty,
				\end{align}
				where $\bS(u)= \norm{\nabla u}_{L^2(\bO)}^2 + \norm{u}_{L^p(\bO)}^p$ and $u_{\infty}\in\bM$ solves 
				\begin{align}\label{eqn-stationary-1}
					\Delta u_{\infty} - \abs{u_{\infty}}^{p-2} u_{\infty}
					+ \bS(u_{\infty}) u_{\infty} = 0 \ \text{ in }\ L^{\frac{p}{p-1}}(\bO)+H^{-1}(\bO).
				\end{align}
		\end{proposition}
		
		\begin{proof}
			 Let us first choose and fix 
				$u_0 \in L^p(\bO) \cap H_0^1(\bO)\cap\bM$. Then, there exists a function $u$, which is the unique strong solution to problem \eqref{eqn-main-problem-copy} (see Theorem \ref{Thm-main-wellposed}). From Definition \ref{def-strong-soln}, note that 
			$$\frac{\partial u}{\partial t}
			\in L^2(0,\infty, L^2(\bO)).$$
			Thanks to the Banach-Alaoglu Theorem \cite[see Theorem 3.1]{JBC-90}, there exists a subsequence (still labeled by the same notation) $\{\tau_n\}_{n\in \N}$, $\tau_n \to \infty$ as $n \to \infty$, such that
			\begin{align}\label{Prop-Stationary-1}
				\frac{\partial u(\tau_n)}{\partial t}\to 0
				\ \text{ in }\ L^2(\bO).
			\end{align}
			Moreover, observe from the estimates \eqref{eqn-Eu} that
			$$\sup_{t>0} \abs{\bS(u(t))} \leq C,$$
			and it implies 
			$$ \norm{u}_{L^\infty(0,\infty;L^p(\bO) \cap H_0^1(\bO))} \leq C.$$
			Once again the Banach-Alaoglu Theorem, as $n\to \infty$, yields
			\begin{align}\label{Prop-Stationary-2}
				\left\{
				\begin{aligned}
					u(\tau_n) & \rightharpoonup u_{\infty} \ \ \, \text{ in }\  L^p(\bO) \cap H_0^1(\bO),\\
					\bS(u(\tau_n)) & \to \bS_\infty \ \ \text{ in }\ \ \R.
				\end{aligned}
				\right.
			\end{align}
			Due to the compactness of the embedding $H_0^1(\bO)\embed L^2(\bO)$, taking a further subsequence of $\{\tau_n\}_{n\in \N}$ (denoted as before) such that $$u(\tau_n)\to u_{\infty}\ \text{ in }\ L^2(\bO),$$ and $u(\tau_n,x)\to u_\infty(x),$ for a.e. $x\in\mathcal{O}$, along a further subsequence.  Since $\|u(\tau_n)\|_{L^2(\bO)}=1$,  it is immediate that $\|u_{\infty}\|_{L^2(\bO)}=1$, i.e., $u_{\infty}\in\bM$.
			Thus, by using the above strong convergence and the convergences given in \eqref{Prop-Stationary-2},
			we obtain
			\begin{align}\label{Prop-Stationary-3}
				\left\{
				\begin{aligned}
					\abs{u(\tau_n)}^{p-2} u(\tau_n) & \rightharpoonup \abs{u_{\infty}}^{p-2} u_{\infty}  \ \ \text{ in }\ \ L^{\frac{p}{p-1}}(\bO)+H^{-1}(\bO),\\
					\Delta u(\tau_n) & \rightharpoonup \Delta u_{\infty} \ \ \text{ in }\ \ L^{\frac{p}{p-1}}(\bO)+H^{-1}(\bO).
				\end{aligned}
				\right.
			\end{align}
			Therefore, by utilizing all the convergences \eqref{Prop-Stationary-1}, \eqref{Prop-Stationary-2} and  \eqref{Prop-Stationary-3} in the equation \eqref{eqn-main-problem-copy}, we obtain
			\begin{align*}
				\langle \Delta u(\tau_n) - \abs{u(\tau_n)}^{p-2}u(\tau_n)+ \bS(u(\tau_n)) u(\tau_n), \psi \rangle
				\to \langle \Delta u_{\infty} - \abs{u_{\infty}}^{p-2}u_{\infty} + \bS_\infty u_{\infty}, \psi \rangle,
			\end{align*}
			for every $\psi\in L^{p}(\bO)\cap H_0^{1}(\bO)$. In particular, we have
			\begin{align*}
				\Delta u_{\infty} - \abs{u_{\infty}}^{p-2}u_{\infty} + \bS_\infty u_{\infty} = 0  \ \text{ in }\ L^{\frac{p}{p-1}}(\bO)+H^{-1}(\bO).
			\end{align*}
			Now, let us take $L^2$ inner-product of the above equation with $u_{\infty}$, we deduce
			\begin{align*}
				-\norm{\nabla u_{\infty}}_{L^2(\bO)}^2 - \norm{u_{\infty}}_{L^p(\bO)}^p
				+ \bS_\infty \|u_{\infty}\|_{L^2(\bO)}^2 = 0 \implies \bS_\infty = \bS (u_{\infty}).
			\end{align*}
			Note that
			\begin{align*}
				\bS(u(\tau_n))=\norm{\nabla u(\tau_n)}_{L^2(\bO)}^2 + \norm{u(\tau_n)}_{L^p(\bO)}^p\to  \norm{\nabla u_{\infty}}_{L^2(\bO)}^2 + \norm{u_{\infty}}_{L^p(\bO)}^p=\bS (u_{\infty}),
			\end{align*}
			implies the norm convergence of $\{u(\tau_n)\}_{n\in\mathbb{N}}$ in $L^p(\bO)\cap H_0^1(\bO)$. Together with this fact, the weak convergence given in \eqref{Prop-Stationary-2} provides the strong convergence (see Theorem \ref{Prop-Radon-Riesz}), i.e., $u(\tau_n) \to u_{\infty} \ \text{in}\  L^p(\bO) \cap H_0^1(\bO)$.
		\end{proof}
		
		
		By combining Propositions \ref{Prop-uniq-positive} and Proposition \ref{Prop-Stationary}, we have the main result of this section, which is motivated from  \cite[Theorem 4.3]{PA+PC+BS-24}. We infer from the  maximum principle \ref{sec-max-principle} that if $u_0 \in L^p(\bO) \cap H_0^1(\bO)\cap \bM$ is non-negative, then the unique strong solution to the problem \eqref{eqn-main-problem-copy} is also non-negative.

			\begin{theorem}\label{Thm-main-assymp}
				Suppose $\bU  \in L^p(\bO) \cap H_0^1(\bO)\cap\bM$ is the unique positive solution to the minimization problem \eqref{min-problem}, which also solves the stationary equation \eqref{eqn-stationary-copy} (guaranteed by Theorem \ref{Thm-ground-soln} and Proposition \ref{Prop-uniq-positive}).
				If $u_0 \in L^p(\bO) \cap H_0^1(\bO)\cap \bM^+$, where $$ \bM^+ := \{u\in L^2(\bO) : u \geq 0 \ \text{ and }\  \norm{u}_{L^2(\bO)} = 1\},$$ and
				\begin{align*}
					&u \in C([0,\infty);  L^p(\bO) \cap H_0^1(\bO)\cap \bM^+)
				\end{align*}
				is the unique strong solution to the problem \eqref{eqn-main-problem-copy} (see Theorem \ref{Thm-main-wellposed}), then
				\begin{equation}\label{eqn-convergence full set}
					u(t) \rightarrow \bU  \ \mbox{ in } \  L^p(\bO) \cap H_0^1(\bO),\ \mbox{ as } \ t\rightarrow \infty.
				\end{equation}
			\end{theorem}
			
			\begin{remark}
				Antonelli et al. \cite{PA+PC+BS-24} proved the above theorem on balls only for $p < \infty$ when $d \leq 2$, and for $p < \frac{2d}{d-2}$ when $d \geq 3$. In contrast, our result generalizes this to any bounded smooth domains, all dimensions $d \geq 1$ and for all $2 \leq p < \infty$. However, it is worth noting that their analysis also covers the case where the nonlinearity exhibits a pumping effect, corresponding to a positive sign.
			\end{remark}
			
			\begin{proof}
				First note that, Theorem \ref{Thm-ground-soln} and Proposition \ref{Prop-uniq-positive} provide the existence of the unique positive minimizer $\bU$ of the energy functional $\bE$ under the constraint $\|\bU \|_{L^2} =1$, which solves the stationary equation \eqref{eqn-stationary}. 
				
				\vspace{1mm}
				\noindent
				\textit{Step 1.}
				Now, if we choose and fix an initial datum $u_0 \in L^p(\bO) \cap H_0^1(\bO)\cap\bM^+$. Then, there exists a function $u$, guaranteed by Theorem \ref{Thm-main-wellposed}, which is the unique strong solution to the problem \eqref{eqn-main-problem}. By an application of Proposition \ref{Prop-Stationary}, there exists a sequence of times $\{\tau_n\}_{n\in \N}$, with $\lim_{n\to \infty}\tau_n \rightarrow \infty$, be such that
				\begin{align*}
					u(\tau_n) \rightarrow u_{\infty}\ \text{ in }\ L^p(\bO) \cap H_0^1(\bO) \ \text{ and }\ \| u_{\infty}\|_{L^2} = 1 >0,
				\end{align*}
				where $u_{\infty}$ solves the stationary equation \eqref{eqn-stationary}. Since $u_0\in \bM^+$, i.e., $u_0\ge0$, it follows from the maximum principle (cf. subsection \ref{sec-max-principle}) that 
				$u(\tau_n) \geq 0$, and consequently $$u_{\infty} \geq 0.$$ 
				By Proposition \ref{Prop-uniq-positive}, we also know that the positive solution to the stationary problem \eqref{eqn-stationary} is unique, which implies
				\begin{equation}\label{eqn-asy-1}
					u_{\infty} = \bU.
				\end{equation}
				
				\vspace{1mm}
				\noindent
				\textit{Step 2.}
				Next, note from Remark \ref{Rmk-En-dissipative} that the functional $\bE$ varies continuously and decreases over time, ensuring the existence of the limit, say $\bE_\infty$, i.e.,
				\begin{equation}\label{eq:grt}
					\bE(u(t)) \searrow \bE_{\infty} \geq \bE(\bU ),\ \text{ as } \ t \to \infty.
				\end{equation}
				Since $\bU $ is the unique positive minimizer of the optimization problem \eqref{eqn-mini} and $u_\infty = \bU$ from the above step (see \eqref{eqn-asy-1}), we identify $\bE_\infty$ as the ground state energy, due to
				\begin{equation}
					\bE_\infty = \lim_{t \rightarrow \infty} \bE(u(t)) = \lim_{n \rightarrow \infty} \bE(u(\tau_n)) = \bE(u_\infty) = \bE(\bU).\label{eqn-asy-2}
				\end{equation}
				
				\vspace{1mm}
				\noindent
				\textit{Step 3.}
				Lastly, our goal is to establish the convergence $u(t) \rightarrow \bU $ in $L^p(\bO) \cap H_0^1(\bO)$, $t\to\infty$. Assume, to the contrary, that this does not hold. Then, there exists a sequence of times $\{\tau_k\}_{k\in \N}$ with $\tau_k \to \infty$ such that by using \eqref{eqn-asy-2}, we get
				$$
				\bE(u(\tau_k)) \rightarrow \bE(\bU), \ \text{ as }\ k \to \infty$$
				and there exists $\eps > 0,$
				\begin{equation*}
					\| u(\tau_k) - \bU \|_{L^p(\bO) \cap H_0^1(\bO)} \geq \eps, \ \mbox{ for all } \ k \in \N.
				\end{equation*}
				Since $u\in C([0,\infty); L^p(\bO)\cap H_0^1(\bO))$, it follows that $\sup_{k\in\N}\|u(\tau_k)\|_{L^p(\bO) \cap H_0^1(\bO)} \leq C < \infty$. Thus, by the Banach-Alaoglu Theorem, there exists a subsequence (considered by same notation), $\{\tau_k\}_{k\in \N}$ with $\tau_k \to \infty$ and a profile $\wtilde{u} \in L^p(\bO) \cap H_0^1(\bO)$ such that 
				\begin{align}
					u(\tau_k) \rightharpoonup \wtilde{u}\  \text{ in }\ L^p(\bO) \cap H_0^1(\bO), \ \text{ as } \ k \to \infty.\label{eqn-u(tau_k)-u}
				\end{align}
				On other hand, from Proposition \ref{Prop-Stationary}, we also have
				\[ \bS(u(\tau_k)) \to\bS(\wtilde{u}) \ \text{ in }\ \R,\ \text{ as }\ k \to \infty.\]
				Thus, the Radon-Riesz property (see Proposition \ref{Prop-Radon-Riesz}) implies \begin{equation}\label{eqn-asy-3}
					u(\tau_k) \to \wtilde{u}\  \text{ in }\ L^p(\bO) \cap H_0^1(\bO).
				\end{equation}
				By the weakly lower semicontinuity of the functional $\bE$ (see Lemma \ref{Lem-E-wls-coer}), it follows that
				\begin{equation*}
					\bE(\wtilde{u}) \leq \liminf_{k\to\infty} \bE(u(\tau_k)) = \bE(\bU ).
				\end{equation*}
				But, $\bU$ is the unique minimizer of $\bE$ (from \eqref{eqn-asy-2}), which implies $\bE(\wtilde{u}) = \bE(\bU )$. In particular, we deduce
				$$\| u(\tau_k)\|_{L^p(\bO) \cap H_0^1(\bO)} \rightarrow \| \bU \|_{L^p(\bO) \cap H_0^1(\bO)},\ \text{ as }\ k \to \infty.$$ 
				Hence from \eqref{eqn-asy-3}  and the uniqueness of $\bU$, we finally have
				$$u(\tau_k) \rightarrow \bU \ \text{ in }\ L^p(\bO) \cap H_0^1(\bO),\ \text{ as }\ k \to \infty.$$ This contradicts the assumption $\| u(\tau_k) - \bU \|_{L^p(\bO) \cap H_0^1(\bO)} \geq \eps$.
				The proof of the Theorem \ref{Thm-main-assymp} is thus complete.
		\end{proof}

		\appendix
		\renewcommand{\thesection}{\Alph{section}}
		\numberwithin{equation}{section}
		\section{}\label{Appendix}
		This section aims to provide some well-known and celebrated results like the Aubin-Lions Lemma, the Strauss Lemma, the Lions Lemma and the Radon-Riesz property used in this work. Moreover, we give a maximum principle for a damped heat equation used in the proof of Theorem \ref{Thm-main-assymp}.
		
		\subsection{Auxiliary results}
		\begin{theorem}[Aubin-Lions Lemma {\cite[Theorem 5]{JS-87}}]\label{Aubin-Lions Lemma}
			Let $\X, \Bb$ and $\Y$ are Banach spaces with compact embedding $\X \embed \Bb \embed \Y$. Assume $1 \leq p \leq \infty$ and
			\begin{itemize}
				\item[(i)] the set $F$ is bounded in $L^p(0,T;\X)$,
				\item[(ii)] $\norm{\tau_h f - f}_{L^p(0,T-h;\Y)} \to 0\ $ as $\ h\to0$, uniformly for $f\in F$,
			\end{itemize}
			where $\tau_hf(t) = f(t+h)$ for $h>0$. Then the set $F$ is relatively compact in $L^p(0,T;\Bb)$ and in $C([0,T];\Bb)$ if $p=\infty$.
		\end{theorem}
		
		We now present a generalization of the celebrated Lions–Magenes Lemma. \cite{JLL+EM-I-72}.
		
		\begin{theorem}[{\cite[Theorem 1.8]{VVC+MIV-02}}]\label{Thm-Abs-cont}
			Let $\H$ be a Hilbert space, and let $\V, \E, \X$ be Banach spaces, satisfying the inclusions
			\[\V \embed \H \embed \V^\prime \embed \X\  \mbox{ and } \ \E \embed \H \embed \E^\prime \embed \X,\]
			where the spaces $\V^\prime$ and $\E^\prime$ are the duals of $\V$ and $\E$, respectively. Here the space $\H^\prime$ is identified with $\H$. Assume that $p>1$ and $u\in L^{2}(0,T;\V)\cap L^p (0,T; \E)$,  $u'\in D^\prime(0,T; \X)$ and $u^\prime = u_1 + u_2$, where $u_1 \in L^{2}(0,T;\V^\prime)$ and $u_2 \in L^{p^\prime}(0,T;\E^{\prime})$. Then,
			\begin{itemize}
				\item[(i)] $u \in C([0,T]; \H),$
				\item[(ii)] the function $[0,T] \ni t \mapsto\|u(t)\|_{\H}^2 \in \mathbb{R}$ is absolutely continuous on $[0,T]$, and
				\begin{align}
					\|u(t)\|_{\H}^2 = \left\langle u(t),u'(t)\right\rangle = 2\left\langle u(t),u_1(t)\right\rangle + 2\left\langle u(t),u_2(t)\right\rangle,
				\end{align}
			\end{itemize}
			for a.e.  $t\in [0,T]$, i.e.,
			\begin{align*}
				\|u(t)\|_{\H}^2 = \|u(t)\|_{\H}^2 + 2\int_0^t \big[ \left\langle u(s),u_1(s)\right\rangle + \left\langle u(s),u_2(s)\right\rangle\big] ds,
			\end{align*}
			for all $t\in[0,T]$. 
		\end{theorem}

		
		\begin{theorem}[Strauss Lemma {\cite[Theorem 2.1]{WAS-66}}]\label{Strauss-Lemma}
			Let $\X \subset \Y$ be two Banach spaces such that $\X$ is reflexive and the embedding $\X\embed\Y$ is dense and continuous. Assume also that $\X^{\prime}$ is separable. If $T>0$,  then
			\begin{align*}
				L^\infty(0,T;\X)\cap C_{w}([0,T]; \Y) \cong C_{w}([0,T]; \X).
			\end{align*}
		\end{theorem}
		
		\begin{lemma}[Lions Lemma {\cite[Lemma 1.3]{JLL-69}}]\label{Lem-Lions}
			Let $\bO_T$ be a bounded open set of $\R^n\times \R$, $g_m$ and $g$ be functions in $L^q(\bO_T)$, for $m\in\N$ and $1<q <\infty$, such that
			\begin{align*}
				\norm{g_m}_{L^q(\bO_T)} \leq C,\ \mbox{ for every } \ m\in\N \; \; \mbox{and} \ \ g_m \to g\;\; \mbox{a.e. in}\ \ \bO_T,\  \mbox{as}\ \ m \to \infty.
			\end{align*}
			Then, $g_m \rightharpoonup g$ in $L^q(\bO_T)$, as $m\to \infty$.
		\end{lemma}
		
		\begin{proof}
			Let us choose and fix $N\in\N$ and define the sequence of sets $\{E_N\}_{N\in\N}$ such that 
			\begin{align*}
				E_N := \{y\in \bO_T : \abs{g_m(y) - g(y)} \leq 1 \;\; \mbox{for}\;\; m\geq N\}.
			\end{align*}
			Note that, each $E_N$ is measurable  and increases with $N$. Furthermore, $|E_N|$ converges to  $|\bO_T|$, as $N$ tending to infinity, 
			
			Suppose $\{\Phi_N\}_{N\in\N}$ is a sequence of set of functions in $L^{q^\prime}(\bO_T)$ ($1/q + 1/{q'} =1$) defined as
			\[\Phi_N:= \{\varphi \in L^{q^\prime}(\bO_T) : \supp(\varphi) \subseteq E_N \} \ \ \mbox{and}\ \ \Phi := \bigcup\limits_{N\in \mathbb{N}} \Phi_N=\lim_{N\to\infty}\Phi_N.\] 
			Then by the definition of $E_N,$ it is clear that $\Phi$ is dense in $L^{q^\prime}(\bO_T)$.
			
			Let us choose and fix $N_0\in\N$. Using the definition of $E_{N_0}$, for any $\varphi \in \Phi_{N_0}$, we have
			\begin{align}\label{eqn-JLL-lemma-1}
				| (g_m - g) \varphi |=|g_m - g||\varphi|\leq|\varphi| \ \text{ and }\  g_m \to g\;\; \mbox{a.e. in}\ \ \bO_T.
			\end{align}
			According to the Lebesgue Dominated Convergence Theorem \cite[see Theorem 1.34]{WR-87}, for every $\varphi \in \Phi_{N_0}$
			\[\int_{\bO_T}(g_m(y)-g(y))\varphi(y)dy = \langle g_m - g, \varphi \rangle \to 0\ \ \mbox{as}\  \ m\to \infty.\]
			
			Since $\Phi$ is dense in $L^{q^\prime}(\bO_T)$ and the above convergence holds for any $N_0\in \N$, it implies that for any $\varphi \in \Phi$
			\[ \langle g_m - g, \varphi \rangle \to 0\ \ \mbox{as}\  \ m\to \infty,\]
			which completes the proof.
		\end{proof}
		
		\begin{proposition}[Radon-Riesz property]\label{Prop-Radon-Riesz}
			Assume that $\X$ is a uniformly convex Banach space. Let $\{f_n\}_{n\in\N}$ be a sequence in $\X$ such that $f_n \rightharpoonup f$ in $\X$ and
			\begin{align*}
				\norm{f_n}_{\X} \to \norm{f}_{\X}.
			\end{align*}
			Then $f_n \to f$ in $\X$.
		\end{proposition}
		
		\begin{proof}
			Suppose $\{f_n\}_{n\in\N}$ is a weakly convergent sequence in a uniformly convex Banach space $\X$ such that $\norm{f_n}_{\X} \to \norm{f}_\X$. Let us consider
			\begin{align*}
				g_n = \frac{f_n}{\norm{f_n}_\X}\ \text{ and }\ g = \frac{f}{\norm{f}_\X}.
			\end{align*}
			Then, we have $\norm{g_n}_\X = 1$, for $n\in\N$ and $\norm{g}_\X =1$, and also $g_n \rightharpoonup g$ in $\X$ and $\norm{g_n}_\X \to 1$.
			
			Now, by the Hahn-Banach Theorem, take $h\in\X^\prime$ with $\norm{h}_{\X^\prime} = 1$ and ${\langle h, g \rangle} =1$ (\cite[Lemma 20.2]{JCR-20}).
			Then due to weak convergence, we obtain
			\begin{align*}
				{\left\langle h, g_n \right\rangle} \to 1.
			\end{align*}
			Thus, we have
			\begin{align*}
				\abs{\frac{\langle h, g_n\rangle + \langle h, g \rangle}{2}} \to 1,
			\end{align*}
			and also
			\begin{align*}
				\abs{\frac{\langle h, g_n\rangle + \langle h, g \rangle}{2}} \leq \norm{h}_{\X^\prime} \norm{\frac{g_n + g}{2}}_{\X} = \norm{\frac{g_n + g}{2}}_{\X} \leq 1.
			\end{align*}
			By sandwich principle, we assert that
			\begin{align*}
				\norm{\frac{g_n + g}{2}}_{\X} \to 1,
			\end{align*}
			i.e., for every $\delta >0,$ there exists $N\in \N$ such that
			\begin{align*}
				1 - \delta < \norm{\frac{g_n + g}{2}}_{\X},\ \text{ for every }\ n \geq N.
			\end{align*}
			Now, using the uniform convexity of $\X$, it follows that
			\begin{align*}
				\norm{g_n -g}_{\X} \to 0.
			\end{align*}
			Hence using the fact that $\norm{f_n}_{\X} \to \norm{f}_{\X}$, we deduce that $f_n \to f$ in $\X$.
		\end{proof}
		
		To conclude this section, we give a maximum principle for a damped heat equation.
		
		\subsection{Maximum principle}\label{sec-max-principle}
		Let us choose and fix $2\leq p<\infty$.	We consider a damped heat equation in $(0,T)\times\mathcal{O}$ as follows:
		\begin{align}\label{eqn-damp-heat-1}
			\left\{
			\begin{aligned}
				& \frac{\partial u(t)}{\partial t} -\Delta u(t) +\beta|u(t)|^{p-2}u(t)  \geq 0, \\
				& u(0)  = u_0\geq 0, \\
				& u(t)|_{\partial\mathcal{O}}  = 0,
			\end{aligned}
			\right.
		\end{align}
		where the function $\beta = \beta(x,t)\geq 0$ and $u:[0,\infty) \times\mathcal{O} \to \mathbb{R}$ with $u(t)=u(t,x)$ be a smooth solution.
		
		Let us first suppose the strict inequality, i.e.,
		\begin{align}\label{eqn-111}
			\frac{\partial u}{\partial t} -\Delta u +\beta|u|^{p-2}u + \mu u > 0\ \text{ in }\ (0,T)\times\mathcal{O},
		\end{align}
		but there exists a point $(t_0,x_0)\in (0,T]\times\mathcal{O}$ with $$u(t_0,x_0)=\min_{[0,T]\times\overline{\mathcal{O}}}u(t, x)<0.$$
		Given that $0 < t_0 < T$, the point $(t_0, x_0)$ is an interior point of $(0, T) \times \mathcal{O}$. Hence, 
		$$
		\partial_t u(t_0, x_0) = 0,
		$$ 
		due to the fact that $u$ reaches its maximum at $(t_0, x_0)$. 			
		On the other hand $\Delta u\geq 0$ at $(t_0,x_0)$. Therefore, we infer
		\begin{align*}
			\frac{\partial u}{\partial t} -\Delta u +\beta |u|^{p-2}u \leq 0 \ \text{ at }\ (t_0,x_0),
		\end{align*}
		a contradiction to \eqref{eqn-111}. Now, if suppose $t_0=T$, then, since $u$ attains its minimum over $\{T\}\times\mathcal{O}$ at $(t_0,x_0),$ we see that $$\frac{\partial u}{\partial t} \leq 0 \ \text{ at }\ (t_0,x_0).$$ Since the inequality $\Delta u\leq 0$ still holds true at $(t_0,x_0)$, once again we arrive at the contradiction
		\begin{align*}
			\frac{\partial u}{\partial t} -\Delta u +\beta |u|^{p-2}u  \leq 0 \ \text{ at }\ (t_0,x_0),
		\end{align*}
		and so
		\begin{align}\label{eqn-112}
			\min_{[0,T]\times\overline{\mathcal{O}}}u(t,x)=\min_{\{0\}\times\partial\mathcal{O}}u(t,x)=u(0,x)\big|_{\partial\mathcal{O}}\geq 0.
		\end{align}
		
		Let us now consider the general case that \eqref{eqn-111}  holds. We take $u^{\eps}(t,x)=u(t,x)+\eps t,$ where $\eps>0$. Then, we have
		\begin{align}
			\frac{\partial u^{\eps}}{\partial t}-\Delta u^{\eps}+\beta|u^{\eps}|^{p-2}u^{\eps}
			& = \frac{\partial u}{\partial t}+\eps-\Delta u+\beta|u+\eps t|^{p-2}(u+\eps t)\\
			& = \frac{\partial u}{\partial t}-\Delta u+\beta|u|^{p-2}u  +\eps+\beta|u+\eps t|^{p-2}(u+\eps t)-\beta|u|^{p-2}u\\
			& \geq \eps+\beta(p-1)|(u+\eps t)+\theta u|^{p-2}(\eps t) >0 \ \text{ in }\ (0,T)\times\mathcal{O},
		\end{align}
		for some $0<\theta<1$. Therefore, by using the previous case, we find $$\min_{[0,T]\times\overline{\mathcal{O}}}u^{\eps}(t,x)=\min_{\{0\}\times\partial\mathcal{O}}u^{\eps}(t,x)=u(0,x)\big|_{\partial\mathcal{O}}+\eps t.$$ Letting $\eps\to 0$, we arrive at \eqref{eqn-112}. Note that \eqref{eqn-112} easily gives $u(t,x)\geq 0$, for all $(t,x)\in [0,T]\times\overline{\mathcal{O}}$.
		
		Let us now consider the problem \eqref{eqn-main-problem-copy}. We know from \eqref{eqn-Eu} that
		\begin{align}\label{eqn-113}
			\|\nabla u(t)\|_{L^2(\mathcal{O})}^2+\|u(t)\|_{L^p(\mathcal{O})}^p\leq C,\ \text{ for all }\ t\geq 0.
		\end{align}
		Consider the transformation $v(t,x)=e^{-Ct} u(t,x)$, where $C$ is the constant appearing in \eqref{eqn-113}. Then $v$ satisfies
		\begin{align*}
			& \frac{\partial v(t)}{\partial t} -\Delta v(t) +e^{(p-2)Ct}|v(t)|^{p-2}v(t) \\
			& =-Ce^{-Ct}u(t)+e^{-Ct}\left[\frac{\partial u(t)}{\partial t}-\Delta u(t) + |u(t)|^{p-2}u(t) \right]\\
			& = -Ce^{-Ct}u(t)+e^{-Ct}\left(\|\nabla u(t)\|_{L^2(\mathcal{O})}^2+\|u(t)\|_{L^p(\mathcal{O})}^p\right)u(t)\\
			& = -\left[C-\left(\|\nabla u(t)\|_{L^2(\mathcal{O})}^2+\|u(t)\|_{L^p(\mathcal{O})}^p\right)\right]v(t)\leq 0.
		\end{align*}
		Since $v(0,x)=u(0,x)\geq 0$, from the above result, we infer $u(t,x)\geq 0,$ for all $(t,x)\in [0,T]\times\overline{\mathcal{O}}$.

		\appendix

		\medskip\noindent
		\textbf{Acknowledgments:} A. Bawalia wants to thank University Grants Commission (UGC), Govt. of India for financial assistance (File No.: 368/2022/211610061684).
		M. T. Mohan would  like to thank the Department of Science and Technology (DST) Science $\&$ Engineering Research Board (SERB), India for a MATRICS grant (MTR/2021/000066).

		\bibliographystyle{plain}
		\bibliography{Final_arxiv_damped_heat}
		
	\end{document}